\RequirePackage{fix-cm}
\documentclass[smallextended]{svjour3}       % onecolumn (second format)
\smartqed  % flush right qed marks, e.g. at end of proof
\usepackage{graphicx}
\usepackage{a4,amssymb, amsmath}
\usepackage{graphics,color}
\usepackage{epsfig}
\usepackage{subfig}
\usepackage{color}
\usepackage{enumitem}
\usepackage{ulem}

\newtheorem{thm}{Theorem}[section]
\newtheorem{cor}{Corollary}

\newtheorem{alg}{Algorithm}

\newcommand*\diff{\mathop{}\!\mathrm{d}}

\DeclareMathOperator{\dt}{\mathrm{d}t}

\DeclareMathOperator{\R}{\mathbb{R}}

\DeclareMathOperator{\N}{\mathbb{N}}

\DeclareMathOperator{\dx}{\mathrm{dx}}
\DeclareMathOperator{\dy}{\mathrm{dy}} 

\definecolor{DarkGreen}{rgb}{0,0.4,0.1}

\newcommand{\js}[1]{{\color{black}#1}}
\newcommand{\cu}[1]{{\color{black}#1}}
\definecolor{Orange}{rgb}{1.0, 0.3, 0.0}
\newcommand{\cc}[1]{{\color{black}#1}}
\newcommand{\sj}[1]{{\color{black}#1}}
\newcommand{\rcu}[1]{{\color{black}#1}}
\newcommand{\jak}[1]{{\color{black}#1}}
\newcommand{\rcc}[1]{{\color{black}#1}}
\newcommand{\rc}[1]{{\color{black}#1}}
\newcommand{\fj}[1]{{\color{black}#1}}

\newcommand{\ba}{\mathsf{a}}
\newcommand{\bb}{\mathsf{b}}
\newcommand{\bc}{\mathsf{c}}
\newcommand{\bd}{\mathsf{d}}
\newcommand{\calB}{\mathcal{B}}
\DeclareMathOperator{\OB}{\overline{\calB}}
\newcommand{\OV}{\operatorname{\mathsf{V}}}
\newcommand{\OK}{\operatorname{\mathsf{K}}}
\newcommand{\mOV}{\overline{\OV}}
\newcommand{\OW}{\operatorname{\mathsf{W}}}
\newcommand{\mOW}{\overline{\OW}}

\newcommand{\dual}[2]{\left\langle #1\,,\,#2\right\rangle}

\begin{document}

\title{Optimal operator preconditioning for pseudodifferential boundary problems\thanks{The authors thank Gerd Grubb for helpful discussions and for pointing them to \cite{ajs}.\newline
J.~S.~was supported by The Maxwell Institute Graduate School in Analysis and its
Applications, a Centre for Doctoral Training funded by the UK Engineering and Physical
Sciences Research Council (grant EP/L016508/01), the Scottish Funding Council, Heriot-Watt
University and the University of Edinburgh.
}}
\subtitle{}


\author{Heiko Gimperlein         \and
        Jakub Stocek			\and
        Carolina Urz\'{u}a-Torres
}

\authorrunning{{Gimperlein, Stocek, and Urz\'{u}a-Torres}} % if too long for running head

\institute{H. Gimperlein \at
              Maxwell Institute for Mathematical Sciences \& \\
Department of Mathematics, {Heriot--Watt University}\\ Edinburgh, EH14 4AS\\ United 
Kingdom\\
              \email{h.gimperlein@hw.ac.uk}           %  \\
%             \textit{Present address:} of F. Author  %  if needed
           \and
J. Stocek \at
              \rc{British Antarctic Survey \\
High Cross, Madingley Road\\ Cambridge, CB3 0ET\\ United 
Kingdom\\
              \email{jaksto@bas.ac.uk}}
              \and
              C. Urz\'{u}a-Torres \at
\rcu{
Delft Institute for Applied Mathematics, \\
Delft University of Technology\\
The Netherlands\\
\email{c.a.urzuatorres@tudelft.nl}}
}

\date{Received: date / Accepted: date}

\maketitle
%%%%%%%%%%%%%%%%%%%%%%%%%%%%%%%%%%%%%%%%%%%%%%%%%%%%%%%%%%%%%%%%%%%%%%%%%%%%%%%
\begin{abstract}
We propose an operator preconditioner for general elliptic
pseudodifferential equations in a domain $\Omega$, where $\Omega$ is 
either in $\mathbb{R}^n$ or in a Riemannian
manifold. For linear systems of equations arising from low-order Galerkin
discretizations, we obtain condition numbers that are independent of the 
mesh size and of the choice of bases for test and trial functions. The basic
ingredient is a classical formula by Boggio for the fractional Laplacian, 
which is extended analytically. In the
special case of the weakly and hypersingular operators on a line segment or a
screen, our approach gives a unified, independent proof for a series of recent
results by  Hiptmair,  Jerez-Hanckes, N{\'e}d{\'e}lec and Urz\'{u}a-Torres. 
We also study the increasing relevance of the regularity assumptions on the 
mesh with the order of the operator. Numerical examples validate our theoretical 
findings and illustrate the performance of the proposed preconditioner on 
quasi-uniform, graded and adaptively generated meshes. 
\end{abstract}

%\rcu{CU: I marked all references that I think truly require specific citation  with \rcc{orange}. In addition I marked as '\rcc{SREF?}' those that I believe  would be better to add a specific citation to (but would be ok too, if this is  not possible).}\\

%\rc{JS: I have gone through most of the references. Heiko, you know better for refs \cite{lh,t}. I am not sure about \cite{CHR02}?}

%%%%%%%%%%%%%%%%%%%%%%%%%%%%%%%%%%%%%%%%%%%%%%%%%%%%%%%%%%%%%%%%%%%%%%%%%%%%%%%
\section{Introduction}
This article considers the Dirichlet problem for an elliptic pseudodifferential
operator $A$ \js{of order $2s$} in a bounded Lipschitz domain $\Omega$, where 
$\Omega$ is either a subset of $\mathbb{R}^n$, or, more generally, in a 
Riemannian manifold $\Gamma$:
\begin{align}\label{mainproblem}
\begin{cases}
 Au = f \qquad & \text{in } \Omega,\\
u = 0 \qquad & \text{in } \Gamma\setminus \overline{\Omega}.
\end{cases}
\end{align}
Such pseudodifferential boundary problems are of interest in several applications.
For instance, the integral fractional Laplacian $A=(-\Delta)^s$ and its variants 
$A=\ $ $\mathrm{div}(c(x)\nabla^{2s-1} u)$ in a domain $\Omega \subset \mathbb{
R}^n$ arise in the pricing of stock options \rc{\cite[Chapter 12]{tankov}}, image processing \cite{go}, 
continuum mechanics \cite{d}, and in the movement of biological organisms 
\cite{egp,egps} or swarm robotic systems \cite{eg2}. \cu{By considering $\Omega\subset
\Gamma$ (with $\Gamma$ a Riemannian manifold), one can also study the equations 
for the weakly singular ($A=\OV$) or hypersingular ($A=\OW$) operators arising 
from boundary integral formulations of the first kind for an elliptic boundary 
problem on curve segments or on open surfaces \rcu{\cite[Section~3.5.2]{ss}}}. Another 
interesting example 
would be, in potential theory, where boundary problems of negative order arise 
for the Riesz potential \rc{\cite[Chapter 1, Section~3]{l}}.

On the one hand, the bilinear form associated to $A$ %\eqref{mainproblem} 
is nonlocal, and its Galerkin discretization results in dense matrices. On the 
other hand, the condition number of the Galerkin matrices \rcu{when using low-order 
\fj{piecewise polynomial} basis function} is of order 
$\mathcal{O}(h^{-2|s|})$,  where $h$ is the size of the smallest cell of the 
mesh. Therefore, the solution of the resulting linear system via iterative 
solvers becomes prohibitively slow on fine meshes.

The preconditioning of pseudodifferential equations has been considered in
different contexts. Classically, boundary element methods have been of
interest, where multigrid and additive Schwarz methods \rc{\cite{am1,dirk,ts}, \cite[Chapter 6]{ss}}, 
as well as operator preconditioners \cite{sw} have been studied. A popular 
choice is operator preconditioning based on an %opposite order 
elliptic pseudodifferential operator \js{of the opposite order $-2s$}, yet it 
leads to growing condition numbers when boundary conditions are not respected. 
\textcolor{black}{Indeed, in the case $s=\frac{1}{2}$, the achieved condition number grows 
like $\vert\log(h)\vert^{n+1}$ for $n=1$ \cite[Theorem 4.1]{ms} and $n=2$ \cite[Proposition 1.3.5]{CHR02}.}  
We \cu{prove} that \rcc{the situation worsens} for $\vert s \vert>\frac{1}{2}$, 
\rcc{and the condition number may increase like $h^{1-2|s|}$}, as we \cu{discuss in Section~\ref{sec:precond}}.
Therefore, the ``opposite order'' strategy for $A$ in \eqref{mainproblem} could 
be far from optimal. \js{This motivates the approach we pursue here, which 
incorporates the boundary conditions.}
% This motivates that we pursue a different approach here.}
% We remark that the proposed preconditioner incorporates the singular behavior of the solution $u$ of \eqref{mainproblem} near $\partial\Omega$. The engineering literature captures this singularity by using heuristic weights \cite{bl2} to get rid of the aforementioned logarithmic growth for boundary elements on smooth screens, corresponding to $s = \pm \frac{1}{2}$.

The aforementioned suboptimality was recently overcome for the weakly singular 
and hypersingular operators $\OV$ and $\OW$ on open $2d$ surfaces \cite{hju} 
and curve segments \cite{hju0}, respectively. The proposed preconditioners were 
based on new exact formulas for the inverses of these operators on the flat 
disk \cite{hju1} and interval $[-1,1]$ \cite{jn}. It is important to mention 
that, in this context, this article provides a \cu{novel } and independent
approach to the preconditioners used in \cite{hju1,hju}: \sj{As discussed in 
Remark \ref{reltoBEM}, by identifying $\Omega \subset \mathbb{R}^n$ with the 
flat screen $\Omega\times \{0\} \subset \mathbb{R}^{n+1}$, $\OW$ coincides with the fractional Laplacian
$\frac{1}{2}(-\Delta)^s$ for $s=\frac{1}{2}$, while $\OV$ coincides with $\frac{
1}{2}(-\Delta)^s$ for $s=-\frac{1}{2}$. Boggio's classical formula (Equation 
\eqref{e:green} below) for the fractional Laplacian in the unit ball of 
$\mathbb{R}^n$, respectively its analytic extension to $s \in \mathbb{C}$, 
therefore recovers the exact formulas for $\OV^{-1}$ and $\OW^{-1}$ from 
\cite{hju1,jn} as \rcu{special cases}.} \cu{This connection between the fractional 
Laplacian and boundary integral equations was only known in $1D$ \cite{jn}, and 
we extend it to arbitrary dimension. As a consequence, we obtain a unified and general preconditioning 
strategy for pseudodifferential problems, which includes $\OV$, $\OW$ and $(-\Delta)^s$.}

Recently, the fractional Laplacian has attracted interest in itself. Multigrid
preconditioners have been briefly mentioned in \cite{ag}, while additive Schwarz
preconditioners of BPX-type are \jak{currently being} investigated \cite{bpxnochetto,fmpp}.
Applied to this particular operator $A$, our results lead to the first
\textit{operator} preconditioner. This offers the advantage of benefiting from 
all the rigorous results of the operator preconditioning theory, including 
its applicability to non-uniformly refined meshes, while being easily
implementable. \cu{Indeed, 
%given the mapping properties of the operator $A$, 
solutions to \eqref{mainproblem} feature edge singularities, analogous to those 
for the fractional Laplacian \rc{\cite[Theorem 4]{g3}}. Therefore, when discretizing 
with low-order finite elements, one requires local refinement to recover optimal 
convergence rates. Hence, it becomes mandatory that preconditioners can deal with 
these non-uniform meshes.}\\

% The current article presents a general strategy to overcome this issue for the
% pseudodifferential Dirichlet problem \eqref{mainproblem}. 
Our main result for preconditioning can be found in Theorem~\ref{thm:main}; the 
proposed preconditioner $\mathbf{P}$ is optimal in the sense that the 
bound for the condition number neither depends on the mesh refinement, nor 
on the choice of bases for trial and test spaces.\\
%, as a consequence of the general framework for operator preconditioning \cite[Theorem~1]{h}.\\

\iffalse
\noindent \textbf{Theorem A.} Let $\mathbf{A}$ be the Galerkin matrix of $A$ and
$\mathbf{P}$ the preconditioner in \eqref{precondformula}. Then there exists a
constant \js{$C>0$} independent of $h$ and such that for any discretization satisfying \eqref{eq:infsupAh},
\eqref{eq:infsupdh} and \eqref{eq:dim} the spectral condition number
$\kappa\left(  \mathbf{P}  \mathbf{A}\right)$ is bounded by \js{$C$}.\\

For $\vert s \vert \leq 1$, the requirements \eqref{eq:infsupdh} and 
\eqref{eq:dim} are known to be satisfied for discretizations based on dual 
meshes, under some regularity conditions on the mesh \cite{s2}. \cu{When the 
bilinear form associated to $A$ is elliptic, \eqref{eq:infsupAh} holds for any 
conforming discretization.}\\

\fi
We verify that the preconditioner may be used on shape regular 
algebraically graded meshes, which lead to quasi-optimal convergence rates for 
piecewise linear elements. 
We \cu{prove} that the required mesh assumptions also 
hold for a natural class of adaptively refined meshes. \cu{By doing this, we 
show for the first time that operator preconditioning with standard \rcc{low-order} 
primal-dual finite element discretization does apply to these 
adaptive meshes. Our proof in fact shows the $H^1(\Omega)$-stability of the generalized 
(Petrov-Galerkin) $L^2(\Omega)$ projection on low-order finite element spaces, 
which may also have applications beyond preconditioning.} \\

\noindent \textit{Outline of this article:} Section \ref{sec:notation} recalls
basic notions of fractional Sobolev spaces. The fractional Laplacian and
Boggio's formula are discussed in Section \ref{sec:green}. There we also
explain how to use the latter to define a bilinear form associated to the 
solution operator in the ball. As special cases, we
recover the recent solution formulas for the weakly and hypersingular operators
$\OV$ and $\OW$. Section \ref{sec:bvp} introduces the pseudo\-differential
Dirichlet problem \eqref{mainproblem}. Next, in Section \ref{sec:precond}, 
we recall the operator preconditioning theory and summarize discretization strategies 
under which Theorem~\rcu{\ref{thm:main}} holds. \rcc{In particular,} Section 
\ref{sec:adaptive} discusses the case 
of adaptively refined meshes. The article concludes
with numerical experiments and their discussion in Section \ref{sec:exp}.

%%%%%%%%%%%%%%%%%%%%%%%%%%%%%%%%%%%%%%%%%%%%%%%%%%%%%%%%%%%%%%%%%%%%%%%%%%%%%%%
\section{Sobolev Spaces}\label{sec:notation}

We recall some basic definitions and properties related to Sobolev spaces of 
non-integer order and to the fractional Laplacian. For further details we refer 
to %\cite{hitch}, see also 
\cite{ajs,t,fkv}. 

Let $\Omega \subset \mathbb{R}^n$ be a bounded \sj{Lipschitz} domain, and for $s \in \mathbb{N
}_0$, $H^s(\Omega)$ the Sobolev space of functions in $L^2(\Omega)$ whose 
distributional derivatives of order $s$ belong to $L^2(\Omega)$. For $s \in (0
,\infty)$, we write $m = \lfloor s \rfloor$ and $\sigma = s-m$ and define the 
Sobolev space $H^s(\Omega)$  as 
\fj{\begin{equation*}
H^s(\Omega) := \lbrace v \in H^m(\Omega): |\partial^\alpha v|_{H^\sigma(\Omega)} 
< \infty \ \ \forall |\alpha|= m\rbrace\ .
\end{equation*}}
Here $|\cdot|_{H^\sigma(\Omega)}$ is the Aronszajn-Slobodeckij seminorm
\fj{\begin{equation*}
|v|_{H^\sigma(\Omega)}^2 := \iint_{\Omega \times \Omega} \frac{(v(x)-v(y))^2}{
\vert x - y \vert^{n+2\sigma}} \diff y \diff x.
\end{equation*}}
$H^s(\Omega)$ is a Hilbert space endowed with the norm
\fj{\begin{equation*}
\|v\|_{H^s(\Omega)}^2 := \|v\|_{H^m(\Omega)}^2 + \sum_{|\alpha|=m}|\partial^\alpha 
v|_{H^\sigma(\Omega)}^2.
\end{equation*}}
Particularly relevant for this article are  the Sobolev spaces \rc{\cite[Chapter 4.1]{gs}, \cite[Chapter 3]{ml}}
\fj{\begin{equation*}
{
\widetilde{H}^s(\Omega) := \lbrace v \in H^s(\mathbb{R}^n): \mathrm{supp}\;
v \subset \overline{\Omega} \rbrace}
\end{equation*}}
of distributions whose extension by $0$ belongs to $H^s(\mathbb{R}^n)$. In the 
literature, the spaces $\widetilde{H}^s(\Omega)$ are sometimes denoted by 
${H}^s_{00}(\Omega)$.\\
%\cc{CU: This is assuming $\Omega$ Lipschitz!} \sj{JS: This is true when $\Omega$ is $C^s$, but reference may not by clear. Restrict to Lipschitz everywhere?}
We recall that when $\Omega$ is Lipschitz and $ \frac{1}{2} \neq s \in (0,1)$, 
$\widetilde{H}^s(\Omega)$ coincides with the space $H^s_0(\Omega)$, which is 
the closure of $C^{\infty}_0(\Omega)$ with respect to the $H^s$\rcc{-}norm. Moreover, 
for $s\in(0,\frac{1}{2})$, $\widetilde{H}^s(\Omega)=H^s(\Omega)=H^s_0(\Omega)$. 
All three spaces differ when $s=\frac{1}{2}$.\\

For negative $s$ the Sobolev spaces are defined by duality, \textcolor{black}{and in this article we denote the duality pairing between $\widetilde{H}^s(\Omega)$ and $H^{-s}(\Omega)$ by $\dual{\cdot}{\cdot}_{\Omega}$.} Using local coordinates, the definition of the Sobolev spaces extends to a bounded domain $\Omega$ of a 
Riemannian manifold $\Gamma$. For $|s|\leq 1$ the definition is independent of 
the choice of local coordinates, if $\Omega$ is Lipschitz \rc{\cite[Section 9]{t}}.

%%%%%%%%%%%%%%%%%%%%%%%%%%%%%%%%%%%%%%%%%%%%%%%%%%%%%%%%%%%%%%%%%%%%%%%%%%%%%%%
\section{The Fractional Laplacian}\label{sec:green}

For $s \in (0,1)$, we define the fractional Laplacian of a function 
$u$ \textcolor{black}{in the Schwartz space $\mathcal{S}(\mathbb{R}^n)$} by
\fj{\begin{equation}\label{e:fl}
(-\Delta)^{s} u (x) %= c_{n,s}\, P.V. \int_{\mathbb{R}^n} \frac{u(x)-u(y)}{\vert x - y \vert^{n+2s}} \diff y 
:= c_{n,s} \lim_{\varepsilon \to 0^+} 
\int_{\mathbb{R}^n \setminus \sj{\OB_{\varepsilon}(x)}} \frac{u(x)-u(y)}{
\vert x - y \vert^{n+2s}} \diff y\ ,
\end{equation}}
where %$P.V.$ denotes the Cauchy principal value and 
\sj{$\calB_{\varepsilon}(x)$ 
the $n$-dimensional ball of radius $\varepsilon>0$ centered at 
$x$}. The normalization 
constant $c_{n,s}$ is defined in terms of $\mathsf{\Gamma}$ functions:
\fj{\begin{equation*}
c_{n,s} := \dfrac{2^{2s} s \mathsf{\Gamma}\left(\frac{n+2s}{2}\right)}{\pi^{\frac{
n}{2}}\mathsf{\Gamma}\left(1-s\right)}.
\end{equation*}}
For general $s>0$, we set $m := \lfloor s \rfloor$, $\sigma := s-m$, and define 
$(-\Delta)^s u =  (-\Delta)^m (-\Delta)^{\sigma} u$ for  $u$ in the Schwartz space. 

Equivalently, the fractional Laplacian may be defined in terms of the Fourier 
transform on $\mathbb{R}^n$ as \begin{equation} \label{fourierdef}\mathcal{F}((-\Delta)^s u) = \vert\xi \vert^{2s} 
\mathcal{F}u, \end{equation}\jak{see for example \cite[Equation 25.2]{classicalbook}.} 
\textcolor{black}{For $s>0$ \rcc{t}his formula extends $(-\Delta)^s$ to an unbounded 
operator on $L^2(\mathbb{R}^n)$, as well as to an operator on the space of tempered 
distributions $\mathcal{S}'(\mathbb{R}^n)$. }

\textcolor{black}{To continue the definition of $(-\Delta)^s$ 
to complex values of $s$, recall that the homogeneous function $0\neq \xi \mapsto 
\vert\xi \vert^{2s}$ admits an extension to a (tempered) homogeneous distribution 
on $\R^n$ for $s\in \mathbb{C}\setminus \mathcal{P}$ \rc{\cite[Equation 25.19]{classicalbook}}
, {with
\fj{\begin{equation}\label{eq:defP}
 \mathcal{P} := \{m \in \frac{1}{2}\mathbb{Z}: m\leq -\frac{n}{2}\}.
\end{equation}} }
Formula \eqref{fourierdef} then defines $(-\Delta
)^s$  for $s \in \mathbb{C}\setminus \mathcal{P}$. 
As $\vert\xi \vert^{2s}$ extends to a meromorphic function of $s \in \mathbb{C}$ with values in the space of tempered distributions, in the sense of  \cite{gro}, so does $(-\Delta
)^s = \mathcal{F}^{-1} \circ \vert \xi\vert^{2s} \circ \mathcal{F}$ as a meromorphic family in the space of operators from $\mathcal{S}(\mathbb{R}^n)$ to $\mathcal{S}'(\mathbb{R}^n)$. We refer to \cite[Section 25]{classicalbook} for details, as well as for the fact that $(-\Delta
)^s$ admits a holomorphic continuation to $s=m \in \mathcal{P}$ on the subspace \fj{\begin{equation}\label{orthogonality}\Phi_m := \left\{u \in \mathcal{S}(\mathbb{R}^n) : \forall \alpha \in \mathbb{N}^n_0\ \text{ with }\ |\alpha| \leq -2m - n:  \langle x^\alpha, u \rangle_{\mathbb{R}^n} = 0\right\}\ .
\end{equation} }
For a careful investigation when $s=-\frac{n}{2}$, where $\Phi_m$ consists of functions of mean $0$, see also \cite[Section 3]{stinga}.}

\textcolor{black}{Formula \eqref{fourierdef} finally 
shows that  $(-\Delta)^s$ is an operator of order $2\mathrm{Re}(s)$ and that for $s=1$ one 
recovers the ordinary Laplace operator. \rcu{For a bounded domain 
$\Omega\subset\R^n$, the former can be stated as: there holds the 
continuity $(-\Delta)^s: \widetilde{H}^s(\Omega) \to H^{-s}(\Omega)
$ for $s \in \mathbb{R}$.}}

%%%%%%%%%%%%%%%%%%%%%%%%%%%%%%%%%%%%%%%%%%%%%%%%%%%%%%%%%%%
\subsection{Dirichlet problem for the fractional Laplacian}

In this article the homogeneous Dirichlet problem for the fractional Laplacian 
plays a special role \cu{as an ``auxiliary problem'', which will help us 
construct preconditioners for \eqref{mainproblem}}. 

For a bounded 
Lipschitz domain $\Omega \subset \mathbb{R}^n$ and
$f \in L^2(\Omega)$, it is formally given by:
\begin{align}\label{e:FL}
\begin{cases}
 (-\Delta)^s u = f \qquad & \text{in } \Omega,\\
u = 0 \qquad & \text{in } \mathbb{R}^n\setminus \overline{\Omega}.
\end{cases}
\end{align}
For $s\in (0,1)$, its variational formulation is expressed in terms of the 
bilinear form $\cu{\bc}$ on $\widetilde{H}^s(\Omega)$,
\fj{\begin{equation}\label{e:bilinear}
{\bc}(u,v) := \frac{c_{n,s}}{2} \iint_{D} \frac{(u(x)-u(y))(v(x)-v(y))}{\vert x - 
y \vert^{n+2s}} \diff y \diff x\ ,
\end{equation}
where $D := (\mathbb{R}^n \times \Omega) \cup (\Omega \times \mathbb{R}^n) 
{=(\mathbb{R}^n \times \mathbb{R}^n) \setminus (\Omega^c \times \Omega^c)}$. }
Similar formulas for $s>1$ may be found in \jak{\cite[Section 1.1]{ajs}}.  

Note that formally
$$
\cu{\bc}(u,v) = \langle (-\Delta)^s u,v \rangle_{\rcu{\mathbb{R}^n}} - 
\jak{\frac{c_{n,s}}{2}}\iint_{\Omega^c \times \Omega^c} \frac{(u(x)-u(y))(v(x)-v(y))}{\vert x - 
y \vert^{n+2s}} \diff y \diff x,
$$
when $u,v \in H^s(\mathbb{R}^n)$, and the second term vanishes on $\widetilde{H
}^s(\Omega)$. \textcolor{black}{Here $\langle\cdot,\cdot\rangle_{\mathbb{R}^n}$ denotes the duality pairing from Section \ref{sec:notation}.}

\textcolor{black}{Using the Fourier definition \eqref{fourierdef}, the bilinear form 
\begin{equation}\label{cextend}
\cu{\bc}(u,v) = \langle (-\Delta)^s u,v \rangle_{\rcu{\mathbb{R}^n}} = \langle \mathcal{F}^{-1}(\vert\xi \vert^{2s} \mathcal{F}u),v \rangle_{\rcu{\mathbb{R}^n}} = \langle (\mathcal{F}^{-1}\vert\xi \vert^{2s}) \ast u,v \rangle_{\rcu{\mathbb{R}^n}}\end{equation} extends meromorphically to $s \in \mathbb{C}\setminus \mathcal{P}$. Here, $\ast$ denotes convolution. } \textcolor{black}{For $\mathrm{Re}(s) < 0$ the inverse Fourier transform $\mathcal{F}^{-1}\vert\xi \vert^{2s}$ is locally integrable and the integrand is only weakly singular. Specifically, $\mathcal{F}^{-1}\vert\xi \vert^{2s} = c_{n,s} |x|^{-n-2s}$ for $\mathrm{Re}(s)<0$, $s \not \in \mathcal{P}$ (\cite[Equation 25.25]{classicalbook}). For $s>0$ the relation between \eqref{cextend} and 
 \eqref{e:bilinear} is discussed in \cite[Section 25.4]{classicalbook}. }\\

The  weak formulation of \eqref{e:FL} reads as follows:\\

\noindent Find $u \in \widetilde{H}^s(\Omega)$ such that
\begin{equation}\label{e:FLweak}
\cu{\bc}(u,v) = \int_{\Omega} f v \mathrm{d}x, \qquad \forall v \in \widetilde{
H}^s(\Omega).
\end{equation}

\textcolor{black}{Moreover, the bilinear form} 
$\cu{\bc}$ is continuous and elliptic for $s>-\frac{n}{2}$ real: 
\rcu{t}here exist $C_{\cu{\bc}},\cu{\beta_{\bc}} > 0$ with
\begin{align*}
\cu{\bc}(u,v) \leq C_{\cu{\bc}} \|u\|_{\widetilde{H}^s(\Omega)}\|v\|_{\widetilde{
H}^s(\Omega)},\quad
\cu{\bc}(u,u) \geq \cu{\beta_{\bc}} \|u\|_{\widetilde{H}^s(\Omega)}^2\ .
\end{align*} 
\textcolor{black}{The ellipticity for $s>0$ follows by definition of the $\widetilde{
H}^s(\Omega)$-norm, while the case $s \in (-\frac{n}{2},0)$ is \rcc{a} classical 
\rcc{result} in potential theory \cite[Page 358]{l}.} 

\rcc{Therefore, b}y the Lax-Milgram theorem, the variational problem \eqref{e:FLweak} 
admits a unique solution, and the solution operator $f \mapsto u$ extends to an 
isomorphism from $H^{-s}(\Omega)$ to $\widetilde{H}^s(\Omega)$ for all \textcolor{black}{$s>-\frac{n}{2}$}.

\textcolor{black}{For $s \leq -\frac{n}{2}$ ellipticity 
requires additional assumptions, as in \eqref{orthogonality}. \rcc{Although we 
refrain from discussing these modifications in this article, it is worth pointing 
out that ellipticity is} well known for $n=1$ and $s = -\frac{1}{2}$ in the case 
of the weakly singular integral operator from Remark \ref{reltoBEM} \cite{jn}.}

%%%%%%%%%%%%%%%%%%%%%%%%%%%%%%%%%%%%%%%%%%%%%%%%%%%%%%%%%%%
\subsection{Solution operator in the unit ball}

\cu{Let us write $\calB_1$ for the unit ball $\calB_1(0) \subset \mathbb{R}^n$.
When $\Omega = \calB_1$}, explicit solution 
formulas are available. For $s>0$, the Green's function in this case \cu{is} 
given by
\fj{\begin{equation}\label{e:green}
G_s(x,y) := k_{n,s} \vert x-y\vert^{2s - n} \int_0^{r(x,y)} 
\frac{t^{s-1}}{(t+1)^{n/2}} \dt, \qquad \forall x,y \in \mathbb{R}^n, x\neq y.
\end{equation}}
Here $r(x,y) := \dfrac{(1-\vert x
\vert^2)_+ (1-\vert y \vert^2)_+}{\vert x-y \vert^2}$, \fj{$z_+ := \max\{z,0\}$} and $k_{n,s} := \dfrac{
2^{1-2s}}{|\partial \mathcal{B}_1| \mathsf{\Gamma}(s)^2} $.

\jak{For $s \in \mathbb{N}$, Formula \eqref{e:green} goes back to \cite{tb}, 
while for $s \in (0,1)$ it has long been known in potential theory and L\'{e}vy 
processes \rcu{(see e.g.~\rc{\cite[Chapter 1, Section 3]{l} and \cite[Chapter 5, Equation 3]{riesz})}}}. The extension to arbitrary order $s>0$ is 
more recent and may be found in \cite{ajs}. \\

\textcolor{black}{The following theorem from \cite{ajs} shows that $G_s$ in formula \eqref{e:green} indeed} defines the \textcolor{black}{weakly singular} integral 
kernel of the solution operator to \eqref{e:FL} \textcolor{black}{for $s>0$}. \textcolor{black}{More precisely, we} have the following 
explicit formula for the solution of the Dirichlet 
problem for the fractional Laplace operator in the unit ball $\calB_1$:

\begin{theorem}[\rc{\cite[Theorem 1.4]{ajs}}]\label{thm:explicit}
Let $s,\alpha >0$, $2s+\alpha \not \in \mathbb{N}$, \fj{$m := \lfloor s \rfloor$, 
and $\sigma := s-m$}. For $f \in C^{\alpha}
(\OB_1)$, 
define \fj{\begin{equation*}
u(x) := \begin{cases}
0, &\qquad \text{for } x \in\mathbb{R}^n \setminus \OB_1\\
\int_{\calB_1} G_s(x,y) f(y) \dy, &\qquad \rcu{\text{for }} x \in \calB_1
       \end{cases}.
\end{equation*}} Then $u \in C^{2s+\alpha}(\calB_1)$, $\delta^{1-\sigma} 
u \in C^{m,0}(\OB_1)$ and
\begin{align*}
(-\Delta)^s u = f \;in\; \calB_1,\quad u = 0 \; in \; \mathbb{R}^n \setminus 
\OB_1.
\end{align*}
Here \fj{$\delta(x) := \mathrm{dist}(x, \partial \calB_1)$} for $x$ in a neighborhood 
of $\partial \calB_1$.
\end{theorem}
\textcolor{black}{In particular, $u$ defines a solution to the weak formulation 
\eqref{e:FLweak} relevant for  finite element approximations.}\\

The previous theorem motivates us to 
\begin{itemize}
 \item derive formulas for $G_s(x,y)$ which are easily computable for use as a 
 preconditioner; and
 \item extend the aforementioned formula to negative values of $s$.\\
\end{itemize}

With these purposes in mind, the following Lemma shows that Boggio's formula 
\eqref{e:green} can be implemented efficiently and allows further insight for 
general values of $n$ and $s$: 
\begin{lemma}\label{2f1lemma}
Let $s>0$. Then 
\begin{equation*}
G_{s}(x,y) = s^{-1}{k_{n,s}} \vert x-y\vert^{2s - n} {r(x,y)^s}{\,}_2
F_1\left(\frac{n}{2},s; s+1; -r(x,y)\right),
\end{equation*}
where ${\,}_2 F_1$ is the hypergeometric function.
\end{lemma}
\begin{proof}
We need to prove
\begin{equation*}
 \int_0^{r} \frac{t^{s-1}}{(t+1)^{n/2}} \dt =  \frac{r^s}{s}{\,}_2F_1\left(
 \frac{n}{2}\ ,s; s+1; -r\right).
\end{equation*}
This, however, follows directly from the integral representation of ${\,}_2F_1$ 
\rc{\cite{w2}},
\jak{\begin{equation*}
\begin{split}
{\,}_2F_1\left(\frac{n}{2},s;s+1;-r\right) &= \frac{\Gamma(s+1)}{\Gamma(s)} \int_0^1 t^{
s-1} (1+tr)^{-\frac{n}{2}} \dt\\
&= s r^{-s} \int_0^r \frac{t^{s-1}}{(1+t)^{\frac{n}{2}}} \dt\ .
\end{split}
\end{equation*}
}
\rcc{\qed}
\end{proof}
\begin{remark}\label{rem:1}
\jak{For a generic value of $s$ computational libraries are available to efficiently evaluate the hypergeometric 
function ${\,}_2F_1$, see for example \rc{\cite[Section 4]{pop}}. For specific 
values of $s$, explicit formulas for $G_s$ in terms of elementary functions are 
available and allow for more efficient computations as highlighted in Remark~\ref{rem:3}.}
\end{remark}

%Lemma \ref{2f1lemma} implies that the integral kernel $G_s$ in \eqref{e:green} admits a meromorphic continuation  in the order $2s$ to $\mathbb{C}$, as ${\,}_2F_1$ is a meromorphic function of its parameters. 
The following result provides an explicit formula for the holomorphic 
continuation of the integral kernel $G_s$ from \eqref{e:green}. \textcolor{black}{We restrict ourselves to the case $\mathrm{Re}(s)>-\frac{n}{2}$ relevant for applications.}

\begin{lemma}
The map $(0,\infty)\ni s \mapsto G_s(x,y) \in \mathcal{D}'(\calB_1 \times 
\calB_1)$ extends to  a  holomorphic family of distributions for \textcolor{black}{$s >-\frac{n}{2}$.} For  $N\in \mathbb{N}_0$, the holomorphic continuation of $G_s(x,
y)$ to the half-plane \textcolor{black}{$\mathrm{Re}(s)>\max\{-N-1, -\frac{n}{2}\}$} is given by
\begin{align}
G_s(x,y) &= k_{n,s}\ |x-y|^{2s-n}\left\{\left(\prod_{j=0}^N \frac{\frac{n}
{2}+j}{s+j}\right) \ \int_0^{r(x,y)} \frac{t^{s+N}}{(t+1)^{1+N+n/2}} \dt \right. \nonumber\\ 
& \qquad +\left. \sum_{k=0}^N \left(\prod_{j=0}^{k-1} \frac{\frac{n}{2}+j}{s+j}
\right) \frac{r(x,y)^{s+k}}{(s+k) (r(x,y)+1)^{k+n/2}}\right\}\ . \label{extensionformula}
\end{align}
%Here $p.f.$ denotes the \textit{finite part}. \textcolor{black}{(Omitted statements for $s \in \mathcal{P}$.)}%For $s \in -\mathbb{N}_0$, {$\mathrm{supp} \ G_s \subseteq \{(x,x) : x\in \calB_1\}$.}
%CU: Is this correct??} \sj{JS: Yes. See below.}
\end{lemma}
\begin{proof}
Using integration by parts, for $\mathrm{Re}(s)>0$ \textcolor{black}{we observe the identity}
\begin{equation}\label{e:keyid}
\int_0^{r(x,y)} \frac{t^{s-1}}{(1+t)^{n/2}} \dt =  \frac{n}{2s} \int_0^{r(x,y)} 
\frac{t^s}{(1+t)^{1+n/2}} \dt + \frac{r(x,y)^s}{s(r(x,y)+1)^{n/2}}.
\end{equation}
\textcolor{black}{Together with \eqref{e:green}, we obtain
{
\fj{\begin{flalign}\label{e:ibp1}
&G_{s}(x,y)= k_{n,s} \vert x-y\vert^{2s - n} \left(\frac{n}{2s} \int_0^{r(x,y)} 
\frac{t^s \dt}{(1+t)^{1+n/2}}  + \frac{r(x,y)^s}{s(r(x,y)+1)^{n/2}}\right)
%&\langle G_{s}, u\otimes v\rangle_{\rcc{\calB_1\times\calB_1}} \nonumber \\
%&= k_{n,s}\left\langle p.f.\vert x-y\vert^{2s - n} \left(\frac{n}{2s} \int_0^{r(x,y)} 
%\frac{t^s \dt}{(1+t)^{1+n/2}}  + \frac{r(x,y)^s}{s(r(x,y)+1)^{n/2}}\right), u(x) v(y)\right\rangle_{\calB_1 \times \calB_1},
\end{flalign}}}}
\textcolor{black}{with the right hand side defined for $s \neq 0$, $\mathrm{Re}(s)>\max\{-1, -\frac{n}{2}\}$. }
Because $\mathsf{\Gamma}(s)$ has simple poles for $s\in-\mathbb{N}_0$, but no 
zeros, and $k_{n,s} = \dfrac{2^{1-2s}}{|\partial \mathcal{B}_1| \mathsf{\Gamma}
(s)^2}$, for $x\neq y$ the kernel $G_s(x,y)$ extends holomorphically to $s=0$, 
with a simple zero in $s=0$. %In fact, for $s=0$ the solution operator to \eqref{e:FL} is the identity, with integral kernel given by the Dirac delta 
%distribution $\delta_{x-y}$. 
The asserted formula follows for $N=0$.

\textcolor{black}{The proof for general $N$ follows by induction: We assume that \eqref{extensionformula} holds for $N \in \mathbb{N}_0$. Note that for $\mathrm{Re}(s)>\max\{-N-1, -\frac{n}{2}\}$,}
\textcolor{black}{\begin{align*}%\label{e:keyid2}
\int_0^{r(x,y)} \frac{t^{s+N}}{(1+t)^{1+N+n/2}} \dt &=  \frac{\frac{n}{2}+N+1}{s+N+1} \int_0^{r(x,y)} 
\frac{t^{s+N+1}}{(1+t)^{2+N+n/2}} \dt \nonumber \\ &\qquad + \frac{r(x,y)^{s+N+1}}{(s+N+1)(r(x,y)+1)^{1+N+n/2}}.
\end{align*}}
\textcolor{black}{The right hand side of \eqref{e:keyid} is defined for $s \not \in -\mathbb{N}_0$, $\mathrm{Re}(s)>\max\{-N-2, -\frac{n}{2}\}$. 
We conclude, }
\textcolor{black}{\begin{align*}
&\left(\prod_{j=0}^N \frac{\frac{n}
{2}+j}{s+j}\right) \ \int_0^{r(x,y)} \frac{t^{s+N}}{(t+1)^{1+N+n/2}} \dt 
+ \sum_{k=0}^N \left(\prod_{j=0}^{k-1} \frac{\frac{n}{2}+j}{s+j}
\right) \frac{r(x,y)^{s+k}}{(s+k) (r(x,y)+1)^{k+n/2}}\\
&= \left(\prod_{j=0}^{N} \frac{\frac{n}{2}+j}{s+j}\right) \ \left\{\frac{\frac{n}{2}+N+1}{s+N+1} \int_0^{r(x,y)} 
\frac{t^{s+N+1}}{(1+t)^{2+N+n/2}} \dt \right. \\ &\qquad + \left. \frac{r(x,y)^{s+N+1}}{(s+N+1)(r(x,y)+1)^{1+N+n/2}}\right\} +\sum_{k=0}^N \left(\prod_{j=0}^{k-1} \frac{\frac{n}{2}+j}{s+j}\right) \frac{r(x,y)^{s+k}}{(s+k) (r(x,y)+1)^{k+n/2}}\ .
\end{align*}}
\textcolor{black}{Equation \eqref{extensionformula} for $N+1$ follows. As above, \eqref{extensionformula} extends to $s = -N-1$ because the simple pole in the denominator is cancelled by the zero of the prefactor $k_{n,s}$.}
\rcc{\qed}
% For  $x\neq y$\cu{, the kernel $G_s(x,y)$ vanishes 
%for $s \in -\mathbb{N}_0$ because of the poles of $\mathsf{\Gamma}(s)$}.
\end{proof}

\begin{proposition}Let %$s \in \mathbb{C} \setminus \mathcal{P}$
$\mathrm{Re}(s)>-\frac{n}{2}$ and $f \in C^\infty_0(\calB_1)$.
Then the distribution  \fj{$u_s := \mathrm{op}(G_s) f \in \mathcal{D}'(\calB_1)$} defined by 
\fj{\begin{equation*}
\langle \mathrm{op}(G_s) f, v \rangle_{\mathcal{B}_1} = \langle G_s, f \otimes v 
\rangle_{\mathcal{B}_1 \otimes \mathcal{B}_1}, \qquad \forall v \in C^{{\infty}
}_0(\mathcal{B}_1),
\end{equation*}}
belongs to $\widetilde{H}^{\mathrm{Re}(s)}(\calB_1)$ and satisfies  the weak formulation \eqref{e:FLweak},
%$$(-\Delta)^s u_s = f\quad \text{pointwise in } \calB_1\ .$$
\begin{equation*}
\cu{\bc}(u_s,v) = \int_{\Omega} f v \mathrm{d}x, \qquad \forall v \in \widetilde{
H}^{\mathrm{Re}(s)}(\Omega).
\end{equation*}
Here, $(-\Delta)^s$ is defined by the continuation of \eqref{fourierdef}. 
\end{proposition}
\begin{proof}
\textcolor{black}{From Theorem \ref{thm:explicit}, note that for $s \in (0,1)$ the function $u_s$ satisfies $(-\Delta)^s u_s = (-\Delta)^s \left( \mathrm{op}(G_s) f \right)= f$ in $H^{-s}(\Omega)$, i.e.~$u_s$ satisfies  the weak formulation \eqref{e:FLweak}. As both the operator $(-\Delta)^s$ with Dirichlet exterior conditions and $\mathrm{op}(G_s)$ are holomorphic for $s$ in the connected set %$\mathbb{C}\setminus \mathcal{P}$
$\mathrm{Re}(s)>-\frac{n}{2}$, the identity extends from 
$s\in (0,1)$ to $\mathrm{Re}(s)>-\frac{n}{2}$.%$\mathbb{C}\setminus \mathcal{P}$. %For $s\in \mathcal{P}$, $(-\Delta)^s$ is  only determined apart from a linear combination of derivative operators,  following \rcc{\cite{lh}}\rc{Heiko??}, but fixed e.g.~by being an inverse of $G_s$. By definition, $\mathrm{op}(G_s)$ also respects the homogeneous boundary condition in $\Omega^c$. 
}\rcc{\qed}
\end{proof}

\vspace{0.3cm}

For numerical applications, we require the bilinear form of the solution 
operator $\mathrm{op}(G_s)$. It is defined as
\fj{\begin{equation}
\bb_{{s}}(u,v) := p.f.~\int_{\calB_1} \int_{\calB_1} G_s(x,y) u(y) v(x) \dy \dx,
\end{equation}}
for $u,v \in C^{\infty}(\OB_1)$. 

\textcolor{black}{The continuity and \rcc{ellipticity} of $\bb_{\rcu{s}}$ in $\widetilde{H}^s(\calB_1)$ for all $s>0$ %follows from the appropriate version of  the G\aa rding inequality in $H^s(\mathbb{R}^n)$ by restriction to  $\widetilde{H}^s(\calB_1)$ \cite[Section 2.1]{g2}. 
follow from the continuity and \rcc{ellipticity} of $\rcc{\bc}$, as its inverse bilinear form.}
From the density of $C^{\cu{\infty}}(
\OB_1)$ in $H^{-s}(\calB_1)$, we conclude:
\begin{lemma}\label{blemmaball} \rcu{Let $s>-\frac{n}{2}$. The bilinear form} $\bb_{\rcu{s}}$ extends 
to a continuous and elliptic bilinear form $\bb_{\rcu{s}}: H^{-s}(\calB_1) 
\times H^{-s}(\calB_1) \rightarrow \mathbb{R}$. More precisely, there exist 
$\cu{C_{\bb_{\rcu{s}}}}, \beta_{\cu{\bb_{\rcu{s}}}} >0$, such that
\begin{align*}
\bb_{\rcu{s}}(u,v) \leq \cu{C_{\bb_{\rcu{s}}}} \|u\|_{H^{-s}(\calB_1)}\|v\|_{H^{-s}(\calB_1)}\ , \quad
\bb_{\rcu{s}}(u,u) \geq \beta_{\cu{\bb_{\rcu{s}}}} \|u\|^2_{H^{-s}(\calB_1)}.
\end{align*}
\end{lemma}

\cu{At the time of writing this article, such explicit solution formulas are 
known for very few specific domains} other than $\calB_1$: the full space 
$\mathbb{R}^n$ (from the Fourier transform of $|x|^{-s}$), and the half space 
$\mathbb{R}^n_+$ (by antisymmetrization).
% A corresponding formula for the integral kernel of the 
% solution operator for the weak formulation \eqref{e:FLweak} follows by continuity 
% from this pointwise formula, as indicated below.\\

\begin{remark} \label{reltoBEM}
%\textcolor{black}{(Note my changes compared to Carolina's version)}
%$ CU: I agree to HG's changes to my version, so I removed the color (and added 
%the subscript _D everywhere, as it was inconsistent) :)
\sj{Problem \eqref{e:FL} is closely related to boundary integral formulations.} 
% for the  Laplace equation.} \rcu{Let $(-\Delta_D)^s: \widetilde{H}^s(\Omega)\to 
% H^{-s}(\Omega)$ be the fractional Laplacian (with Dirichlet conditions in the exterior) from 
% \eqref{e:FL}}. 
\rcc{Let us consider the restriction operator $R_s:H^s(\mathbb{R}^n) \to H^s(\Omega)$.}
By identifying $\Omega \subset \mathbb{R}^n$ with the flat screen 
$\Gamma:=\Omega\times \{0\} \subset \mathbb{R}^{n+1}$, the hypersingular 
operator $\OW$ for the Laplace equation in the exterior domain $\mathbb{R}^{n+1} \setminus 
\overline{\Gamma}$ coincides with \rcc{$R_{-1/2}\circ\frac{1}{2}(-\Delta)^{1/2}$}, 
while the weakly singular operator $\OV$ coincides with \rcc{$R_{1/2}\circ\frac{1}{2}(-\Delta)^{-1/2}$}. \rcu{Indeed, $\OK$ and $\OK^\prime$ vanish on $\Gamma$. 
Therefore, $\OW$ is a multiple of the Dirichlet-to-Neumann operator \cite[Section~3.7]{ss} for 
 the Laplace equation in the exterior domain 
$\mathbb{R}^{n+1} \setminus \overline{\Gamma}$ \cite[Section~12.3]{gs}, 
as is \rcc{$R_{-1/2}\circ(-\Delta)^{1/2}$} \cite[Chapter~11, Equation 11.72]{gg}. 
Similarly, $\OV$ and \rcc{$R_{1/2}\circ(-\Delta)^{-1/2}$} 
are both multiples of 
the Neumann-to-Dirichlet operator.} 
In these cases, \eqref{e:green} and \eqref{e:ibp1} recover recent 
formulas %for $G_{1/2}(x,y)$ and $G_{-1/2}(x,y)$, respectively, that are of interest for 
for the inverses of $\OV$ and $\OW$, which have been of interest in boundary 
integral equations. Let us compute these simplifications for the relevant
values of $n$, $s$:

\noindent a) $n=2$, $s = \frac{1}{2}$: In this case $\int_0^{r} \frac{t^{s-1}}
{(t+1)^{n/2}} \dt 
= 2 \arctan(\sqrt{r})$, so that 
$$G_{1/2}(x,y) = \frac{1}{\pi^2} \vert x-y\vert^{-1}\arctan(\sqrt{r(x,y)})\ 
.$$

Note that $G_{1/2}$ coincides, up to a factor $2$, with the kernel of the 
operator $\mOV$ for the flat circular screen in $3d$ \cite{hju1}.

%\cu{CU:If you look closely, \cite{hju1} defined $\mOV$ as the inverse of $-W$  here, so when comparing to the correct sign, the factor is $+2$.}

\noindent b) $n=1$, $s = \frac{1}{2}$: Here $\int_0^{r} \frac{t^{s-1}}{(t+1)^{
n/2}} \dt = 2 \mathrm{arsinh}(\sqrt{r})$, and hence $$G_{1/2}(x,y) = 2 k_{1,1/2}
\mathrm{arsinh}(\sqrt{r(x,y)})=2 k_{1,1/2} \ln\left(\sqrt{r(x,y)} + \sqrt{1+r(x,y)}\right)\ .$$ 

Writing $\omega(x)=\sqrt{
1-x^2}$, one obtains
\begin{align*}
\sqrt{r(x,y)} + \sqrt{1+r(x,y)} &= \frac{\omega(x)\omega(y)}{\vert x - y\vert} 
+ \sqrt{1 +  \frac{\omega(x)^2\omega(y)^2}{\vert x - y\vert^2}}
%&= \frac{\omega(x)\omega(y) + \sqrt{\vert x - y\vert^2 + \omega(x)^2\omega(y)^2
%}}{\vert x - y\vert}\\
= \frac{\omega(x)\omega(y) + 1 - xy}{\vert x - y\vert}\\
&= \frac{\frac{1}{2} \left( (y-x)^2 + (\omega(x) + \omega(y))^2 \right)
}{\vert x - y\vert}.
\end{align*}

This agrees with the kernel of the operator $\mOV$ from \cite{hju0,jn} up to a 
factor $2$. Note that $k_{1,1/2} = \frac{1}{\pi}$, and see \cite{bucur} for a 
detailed discussion of the prefactor $k_{n,s}$ in the degenerate case $n=2s$.

\noindent c) $n=2$, $s=-\frac{1}{2}$: We obtain
$$ G_{-1/2}(x,y) = -\frac{1}{\pi^2}\left( \frac{1}{\sqrt{r(x,y)}|x-y|^3} + 
\frac{\arctan(\sqrt{r(x,y)})}{|x-y|^3}\right).$$

Again, $G_{-1/2}$ recovers, up to a factor $2$, the kernel of the operator 
$\mOW$ for the flat circular screen in $3d$ \cite{hju1}.

\noindent d) $n=1$, $s=-\frac{1}{2}$: In this case $\frac{n}{2s} \int_0^{r} 
\frac{t^s}{(1+t)^{1+n/2}} \dt = - \frac{2 \sqrt{r}}{\sqrt{1+r}}$,
so that $$G_{-1/2}(x,y) =  -  \frac{\sqrt{1+r(x,y)}}{\pi\vert x-y\vert^{2} 
\sqrt{r(x,y)}}=\frac{xy-1}{\pi|x-y|^2\omega(x)\omega(y)}\ .$$

$G_{-1/2}$ matches, up to a factor $-2$, the kernel of the operator $\mOW$ 
for the interval in $2d$, Formula (4.21) in \cite{jn}.\\
\end{remark}

\begin{remark}\label{rem:3}
For the numerical experiments\rcu{,} below the cases when $n=2$ and 
$s =\frac{1}{4}, \frac{7}{10}$, and $s = \frac{3}{4}$, are also relevant. 
There we obtain:
%\begin{equation*}
%\begin{split}
%G_{3/4}(x,y) = &\sqrt{2} k_{2,3/4} |x-y|^{-1/2} \left( \log(\sqrt{r}-\sqrt{2}\sqrt[4]{r}+1) - \log(\sqrt{r}+\sqrt{2}\sqrt[4]{r}+1)\right. \\ & \left. \quad - 2 \arctan(1-\sqrt{2}\sqrt[4]{r}) +2 \arctan(1+\sqrt{2}\sqrt[4]{r}) \right).
%\end{split}
%\end{equation*}
\begin{equation*}
\begin{split}
G_{1/4}(x,y) = & -2 k_{2,1/4} |x-y|^{-3/2} e^{3 i \pi/4} \left(\arctan(\sqrt[4]{
r}e^{i \pi/4})+\mathrm{artanh}(\sqrt[4]{r}e^{i \pi/4})\right),\\
G_{7/10}(x,y) = & -2 k_{2,7/10} |x-y|^{-3/5} \left( \arctan(\sqrt[10]{r}) 
+ e^{3 i \pi/10} \mathrm{artanh}(\sqrt[10]{r} e^{i \pi/10}) \right. \\
&\left. \quad + e^{9 i \pi/10} \mathrm{artanh}(\sqrt[10]{r} e^{3 i \pi/10})
+ e^{i \pi/10} \mathrm{artanh}(\sqrt[10]{r} e^{7 i \pi/10})\right. \\
&\left.\quad + e^{7 i \pi/10} \mathrm{artanh}(\sqrt[10]{r} e^{9i \pi/10}) \right),\\
G_{3/4}(x,y) = & \ 2 k_{2,3/4} |x-y|^{-1/2} e^{i \pi/4} \left( \arctan(\sqrt[4]{r}
e^{i \pi/4})-\mathrm{artanh}(\sqrt[4]{r}e^{i \pi/4})\right).
\end{split}
\end{equation*}
\end{remark}

\begin{remark}
Similar explicit formulas are available for other rational values of $s$, in 
terms of the Lerch Phi function \cite{w} when $n=2$ and in terms of elementary 
functions for special values of $s$.
\end{remark}

%%%%%%%%%%%%%%%%%%%%%%%%%%%%%%%%%%%%%%%%%%%%%%%%%%%%%%%%%%%%%%%%%%%%%%%%%%%%%%%
\section{Pseudodifferential Dirichlet Problems}\label{sec:bvp}

\rcc{In this Section, we introduce the family of problems we aim to solve.}
Let $A : H^s(\Gamma) \to H^{-s}(\Gamma)$ be a continuous operator of order 
$2s$ on an $n$-dimensional $C^{m,\sigma}$-regular Riemannian manifold $\Gamma$, 
$|s| \leq m+\sigma$. Examples include pseudodifferential operators of order $2s$ 
\cite[Chapter 7--8]{gg}, as well as their generalizations like the weakly or 
hypersingular \rcc{boundary integral}
operators on a manifold $\Gamma$ with edges or corners\sj{, or Riesz potentials 
in potential theory.} 

\cu{Recall the} Dirichlet problem for $A$ in a domain $\Omega \subset \Gamma$ 
\cu{from \eqref{mainproblem}}, which is formally given by
\begin{align}\label{mainbvp}
\begin{cases}
 Au = f \qquad & \text{in } \Omega,\\
u = 0 \qquad & \text{in } \Gamma\setminus \overline{\Omega}.
\end{cases}\nonumber
\end{align}

The weak formulation of Problem \eqref{mainproblem} involves the bilinear form 
$\ba_A$ on $C^\infty_0(\Omega)$, 
defined by 
\fj{\begin{equation}\label{e:bilin}
\ba_A(u,v) := \langle Au, v \rangle_\Gamma = \langle Au, v \rangle_\Omega\ .
\end{equation}}

From the mapping properties of $A$ and the fact that $\widetilde{H}^s(\Omega) 
\subset H^{s}(\Gamma)$, we note
\begin{equation*}
|\ba_A(u,v)| \leq C_A \|u\|_{\widetilde{H}^s(\Omega)}\|v\|_{\widetilde{H}^s(\Omega)}\ .
\end{equation*}

Thus, by continuity, $\ba_A$ extends to a bilinear form on $\widetilde{H}^s(\Omega)$. 
Then, for $f \in H^{-s}(\Omega)$, we obtain the following weak formulation of the 
homogeneous Dirichlet problem \eqref{mainproblem}: Find $u \in \widetilde{H}^s(
\Omega)$, such that
\begin{equation}\label{e:FLweakGamma}
\ba_A(u,v) = \langle f, v \rangle\ , \qquad \forall v \in \widetilde{H}^s(\Omega).
\end{equation}
\rcu{For simplicity, w}e assume that $\ba_A$ satisfies the inf-sup condition 
\begin{equation}\label{e:infsupGamma}
\mathrm{sup}_{v\in \widetilde{H}^s(\Omega)} \frac{\ba_A(u,v)}{\|v\|_{\widetilde{
H}^s(\Omega)}} \geq \cu{\beta}_A \|u\|_{\widetilde{H}^s(\Omega)}
\end{equation}
for all $\rcu{u} \in \widetilde{H}^s(\Omega)$, and some $\cu{\beta}_A>0$.%, 
% as well as the compatibility of the right hand side $f$: $\langle f, v \rangle = 0$ 
% for all $v \in K=\{w \in \widetilde{H}^s(\Omega): \ba_A(\cdot,w) =0\}$.
% 
% Under these assumptions, the variational problem \eqref{e:FLweakGamma} admits a 
% unique solution $u \in \widetilde{H}^s(\Omega)$, and the solution operator $f 
% \mapsto u$ is continuous from the subspace $K^0 \subset H^{-s}(\Omega)$\js{, 
% the polar set of $K$,} of compatible right hand sides to $\widetilde{H}^s(\Omega)$.

\begin{remark}
\cu{We remind the reader that} ellipticity of the bilinear form $\ba_A$ is 
sufficient for the inf-sup condition \eqref{e:infsupGamma} \cu{to hold}. 
Ellipticity of nonlocal Dirichlet problems is discussed 
in \cite{fkv}, for example. 

On the other hand, \cu{coercive pseudodifferential boundary problems, as the 
boundary integral formulations of the Helmholtz equation, also 
satisfy the inf-sup condition \eqref{e:infsupGamma}. Indeed,} G\aa rding 
inequalities are easily discussed when $A$ is a pseudodifferential operator of 
order $2s$ on $\Gamma$ with symbol $p_A(x,\xi)$ \cite{g2}. If $A$ satisfies 
$p_A(x,\xi) \geq c |\xi|^{2s}$ with $c>0$, then  for any $\tilde{s}<s$ the 
associated bilinear form satisfies a G\aa rding inequality on $\Gamma$, 
$$\langle Au, u \rangle_\Gamma \geq \tilde{c}_A \|u\|^2_{H^s(\Gamma)} - 
\cc{\tilde{C}_A}\|u\|^2_{H^{\tilde{s}}(\Gamma)}$$  
for some \sj{$\tilde{c}_A,\, \tilde{C}_A>0$}, see \cite[Theorem B.4]{gs}. By restriction to $u 
\in \widetilde{H}^s(\Omega)$, 
a G\aa rding inequality is satisfied by $\ba_A$, and the inf-sup condition 
\eqref{e:infsupGamma} then holds \rcu{on the complement of} a finite dimensional kernel.
\end{remark}

In the following we assume that $\overline{\Omega}$ is diffeomorphic to the 
unit ball $\overline{\calB}_1 \subset \mathbb{R}^n$ under a $C^{m,\sigma}
$-diffeomorphism $\chi: \calB_1 \to \Omega$.  For $|s| \leq m+\sigma$, \js{by 
the chain rule} it induces an isomorphism $\chi^* : H^{-s}(\Omega) 
\xrightarrow{\sim} H^{-s}(\calB_1)$  by composition with 
$\chi$. From $\chi^*$ and the bilinear form $\bb_{\rcu{s}}$ on $\calB_1$ defined by 
Boggio's kernel \eqref{e:green}, we obtain a bilinear form on $\Omega$:
\begin{equation}\label{e:b_mapped}
\bb_{\rcu{s},\chi}(u,v) := \bb_{\rcu{s}}(\chi^* u, \chi^* v).
\end{equation}
The proof of the \rcu{next} Lemma then follows from the continuity and 
\cu{ellipticity} of the bilinear form $\bb_{\rc{s}}$, \rcu{provided} in Lemma \ref{blemmaball}. 
\begin{lemma}\label{blemma}
\textcolor{black}{For $\mathrm{Re}(s)>-\frac{n}{2}$} \rcu{the bilinear form} $\bb_{\rcu{s},\chi}$ \rcu{defined in \eqref{e:b_mapped}} 
extends to a continuous and elliptic bilinear form $\bb_{\rcu{s},\chi}: H^{-s}(\Omega) 
\times H^{-s}(\Omega) \rightarrow \mathbb{R}$. More precisely, 
there exist $C_{\rcu{s},\chi}, \beta_{\rcu{s},\chi}>0$, such that
\begin{align*}
\bb_{\rcu{s}\rc{,}\chi}(u,v) \leq C_{\rcu{s},\chi} \|u\|_{H^{-s}(\Omega)}\|v\|_{H^{-s}
(\Omega)}\ ,\quad
\bb_{\rcu{s}\rc{,}\chi}(u,u) \geq \beta_{\rcu{s},\chi} \|u\|^2_{H^{-s}(\Omega)}.
\end{align*}
\end{lemma}
%The operator $B_\chi : H^{-s}(\Omega) \xrightarrow{\sim} \widetilde{H}^s(\Omega)$ 
%associated to $\bb_{\chi}$ will be used as an approximate solution operator for 
%the homogeneous Dirichlet problem \eqref{e:FLweakGamma}.

Given its mapping and pseudospectral properties, the operator $B_{\rcu{s},\chi} 
: H^{-s}(\Omega)\xrightarrow{\sim} \widetilde{H}^s(\Omega)$ associated to $\bb_{
\rcu{s},\chi}$ 
will be used to build a suitable preconditioner for 
the homogeneous Dirichlet problem \eqref{e:FLweakGamma}.

\rcu{
\begin{remark}
 If \eqref{e:infsupGamma} does not hold on $\widetilde{H}(\Omega)$ but on the 
 complement of a finite dimensional kernel, one can still use the operator 
 $B_{\rcu{s},\chi}$ 
 to build a suitable operator preconditioner. We refer to \cite{sw} for a 
 detailed discussion.
\end{remark}
}
% 
% \rcu{CU: I think here we should add some discussion about regularity of these 
% problems. I.e. how the tilde spaces encode edge singularities. Since this will 
% motivate that we aim for preconditioners that are stable on local refinements 
% and study the stability on adaptive meshes in Section 5.3. }

%%%%%%%%%%%%%%%%%%%%%%%%%%%%%%%%%%%%%%%%%%%%%%%%%%%%%%%%%%%%%%%%%%%%%%%%%%%%%%%
\section{Preconditioning and Discretization}\label{sec:precond}

As we saw in the previous section, the bilinear forms $\ba_A$ and $\bb_{\rcu{s}
,\chi}$ are continuous and \cu{satisfy inf-sup conditions} in their corresponding 
spaces. \cu{Moreover, their} associated operators $\mathcal{A}$ and $B_{\rcu{s},
\chi}$ are isomorphisms which map in opposite directions. Their composition $B_{
\rcu{s},\chi}
\mathcal{A} \, : \widetilde{H}^s(\Omega) \to \widetilde{H}^s(\Omega)$ therefore 
is an endomorphism. %In summary, we have all the requirements on the continuous level  to use $B_{\chi}$ to construct an operator preconditioner for $\mathcal{A}$.

In this section, we discuss the missing piece to properly apply the operator 
preconditioning theory: We look for adequate discretizations such that the 
composition $B_{\rcu{s},\chi}\mathcal{A}$ remains well-conditioned in the 
discrete setting, and thereby defines an optimal operator 
preconditioner. We follow the approach from \cite[Section 1.2.2]{CHR02}, \cite{h}.

Define the bilinear form $\bd : \widetilde{H}^s(\Omega) \times H^{-s}(\Omega)
 \to \mathbb{R}$ as
\fj{$$\bd(v, \varphi) := \dual{v}{\varphi}_{\Omega}, \quad v \in \widetilde{H}^s(\Omega), 
\varphi \in H^{-s}(\Omega),$$}
where $\dual{\cdot}{\cdot}_\Omega$ denotes \textcolor{black}{the duality pairing from Section \ref{sec:notation}.}

\rcu{Let $\lbrace\mathcal{T}_h\rbrace_{h}$ be a family of triangulations of 
$\Omega$, whose members are labelled by their mesh width $h$.}
Let $\widetilde{\mathbb{V}}_{h} \subset \widetilde{H}^s(\Omega)$ % associated to $\mathcal{T}_h$, 
and $\mathbb{W}_{h} \subset H^{-s}(\Omega)$ be conforming finite element spaces 
\rcu{associated to $\mathcal{T}_h$}.
%associated to $\mathcal{T}_h'$, consisting of piecewise 
%polynomial functions of degree $p$ on the respective mesh (continuous for $p\geq 1$). 
%Let $\widetilde{\mathbb{V}}_{h}^p = \mathbb{V}_{h}^p \cap \widetilde{H}^s(\Omega)$.
We assume that the restrictions of the bilinear forms $\ba_A$ and $\bd$ to these 
finite dimensional spaces satisfy an inf-sup condition uniformly in $h$:
\begin{equation}\label{eq:infsupAh}
\underset{v_h\in \widetilde{\mathbb{V}}_{h}}{\mathrm{sup}} \frac{\ba_A(u_h,v_h)
}{\|v_h\|_{\widetilde{H}^s(\Omega)}} \geq \cu{\beta_A} \|u_h\|_{\widetilde{H}^s
(\Omega)}, \qquad\text{for all }u_h \in \widetilde{\mathbb{V}}_{h},
\end{equation}
%\begin{equation}\label{eq:infsupBh}
%\mathrm{sup}_{\varphi_h\in \mathbb{W}_{h}} \frac{\bb_{\xi}(\upsilon_h,\varphi_h)
%}{\|\upsilon_h\|_{H^{-s}(\Omega)}} \geq \beta \|\varphi_h\|_{H^{-s}(\Omega)}, 
%\qquad \text{for all } \upsilon_h \in {\mathbb{W}}_{h},
%\end{equation}
\begin{equation}\label{eq:infsupdh}
\underset{\varphi_h\in\mathbb{W}_{h}}{\mathrm{sup}} \dfrac{\bd(v_h, \varphi_h)}
{\|\varphi_h\|_{H^{-s}(\Omega)}} \geq \cu{\beta_{\bd}} \|v_h\|_{\widetilde{H}^s
(\Omega)}, \qquad \text{for all } v_h \in \widetilde{\mathbb{V}}_{h},
\end{equation}
with $\cu{\beta_A, \beta_{\bd}}>0$ independent of $h$. Due to ellipticity, an 
analogous inf-sup condition for $\bb_{\rcu{s},\chi}$ holds by Lemma \ref{blemma}.

Then, for any sets of bases 
$$\widetilde{\mathbb{V}}_{h} = \text{span }\lbrace \psi_i \rbrace_{i=1}^N \:
\text{ and } \: \mathbb{W}_{h} = \text{span }\lbrace \phi_j \rbrace_{j=1}^M$$ 
such that 
\begin{equation}
 \label{eq:dim}
N := \dim \widetilde{\mathbb{V}}_{h} = \dim \mathbb{W}_h =:M,
\end{equation}
the Galerkin matrices
\begin{equation*}
\mathbf{A}_{i,j} := \ba_A(\psi_i,\psi_j), \qquad
\mathbf{B}_{i,j} := \bb_{\rcu{s},\chi}(\phi_i,\phi_j), \qquad
\mathbf{D}_{i,j} := \bd(\psi_i,\phi_j),
\end{equation*}
satisfy the following bound for the spectral condition number
\begin{equation}\label{mainresult}
 \kappa\left( \mathbf{D}^{-1} \mathbf{B} \mathbf{D}^{-T} \mathbf{A}\right) \leq
 \dfrac{C_{\rcu{s},\chi} C_A \Vert \bd \Vert^2}{\beta_A \beta_{\rcu{s},\chi} \beta_{\bd}^2}.
\end{equation}
Here $\Vert \bd \Vert$ is the operator norm of $\bd$ \cite[Theorem~2.1]{h}.

\textcolor{black}{We propose the preconditioner
\begin{equation}\label{precondformula}
\mathbf{P}:= \mathbf{D}^{-1} \mathbf{B} 
\mathbf{D}^{-T}\ ,
\end{equation} 
%As the bilinear forms $\ba_A$ and $\bb_{\chi}$ are elliptic in their associated 
%Sobolev spaces, they do satisfy their inf-sup conditions (\eqref{eq:infsupAh} 
%and \eqref{eq:infsupBh}, respectively).
% For operators like the fractional Laplacian the bilinear form $\ba_A$ not only 
% satisfies the inf-sup condition \eqref{e:infsupGamma}, but it is elliptic in 
% its associated Sobolev space. It therefore satisfies the inf-sup condition 
% \eqref{eq:infsupAh} for any conforming choice of $\widetilde{\mathbb{V}}_{h}$. 
%and \eqref{eq:infsupBh}, respectively).
% Therefore, in this case, 
\noindent \cu{and point out that  }
we only need to choose $\widetilde{\mathbb{V}}_{h}$ and $\mathbb{W}_{h}$ such that 
\eqref{eq:infsupdh} and \eqref{eq:dim} are guaranteed.\\
As a consequence of the general framework for operator preconditioning 
\cite[\rcu{Theorem~2.1}]{h} we obtain:
\begin{theorem}\label{thm:main}
 Let $\mathbf{A}$ be the Galerkin matrix of $A$ and
$\mathbf{P}$ the preconditioner in \eqref{precondformula}. Then there exists a
constant \js{$C>0$} independent of $h$ and such that for any discretization 
satisfying \eqref{eq:infsupAh},
\eqref{eq:infsupdh} and \eqref{eq:dim} the spectral condition number
$\kappa\left(  \mathbf{P}  \mathbf{A}\right)$ is bounded by \js{$C$}.
\end{theorem}
In the following, we illustrate how these assumptions can be met on common 
discretizations by \textit{triangular} meshes. }

%%%%%%%%%%%%%%%%%%%%%%%%%%%%%%%%%%%%%%%%%%%%%%%%%%%%%%%%%%%
\subsection{Discretization}
\label{ssec:discretization}

\rcu{
Let us begin by motivating the dicussion and reminding the reader that solutions 
to \eqref{mainproblem} feature edge singularities, and can also have corner 
singularities when {$\Omega$} is not smooth. \textcolor{black}{These singularities are analogous to 
those for the fractional Laplacian for $s\in (0,1)$: Even when $\partial \Omega$ is smooth, \cite[Theorem 4]{g3} shows that the solution $u$ to \eqref{mainproblem} behaves like $u(x) \sim \mathrm{dist}(x,\partial \Omega)^s$ in a neighborhood of $\partial \Omega$. Similarly, near a corner $C$ of a polygon $u(x) \sim \mathrm{dist}(x,C)^\lambda$, where the exponent $\lambda$ depends on $s$ and the geometry of the corner \cite{hgepsjs}.
When discretizing with low-order finite 
elements, these singularities are often resolved by local refinements} to recover optimal convergence rates.

Consequently, it makes sense that preconditioners devised for these kind of 
problems are required to work on meshes which are not quasi-uniform. While other preconditioners 
have been extensively studied on locally refined meshes \cite{am1,fmpp,dirk,mm}, 
this analysis is still incomplete for operator preconditioning. 

Usually, local refinements are implemented via two strategies: 
\begin{enumerate}
 \item Using \textit{a priori} error convergence knowledge to choose suitable 
 algebraically or geometrically graded meshes;
\item Employing \textit{a posteriori} error estimates to implement adaptively 
refined meshes. 
\end{enumerate}

We remark that both approaches are broadly used in the numerical solution of 
PDEs and relevant for the problems this article is interested in. Moreover, 
from a practical point of view, and since we aim for a general preconditioner 
that can be used for a large range of problems of the form \eqref{mainproblem}, 
we take pains to ensure that the proposed preconditioner works for both 
refinement strategies. In order to achieve this, we exploit that 
Theorem~\ref{thm:main} tells us exactly which conditions we need to verify to 
make this happen.
}

\fj{In the following we restrict to the 2-dimensional case, $n=2$.} For simplicity of notation,  assume that $\Gamma$ is a polyhedral surface and 
$\Omega$ has a polygonal boundary. Let \rcu{$\lbrace\mathcal{T}_h\rbrace_{h}$} 
be a family of triangulations of $\Omega$, %, whose members are labelled by their mesh width $h$. 
and let $\mathbb{S}^p(\mathcal{T}_h)$ the finite element spaces 
consisting of piecewise polynomial functions of degree $p$ on a mesh $\mathcal{
T}_h$ (continuous for $p\geq 1$). We choose $\widetilde{\mathbb{V}}_{h}= 
\mathbb{S}^p(\mathcal{T}_h)\cap\widetilde{H}^{s}(\Omega)$.

When $\vert s \vert \leq 1$, the requirements \eqref{eq:infsupdh} and 
\eqref{eq:dim} are known to be satisfied for a wide class of discretizations 
based on dual meshes $\rcu{\check{\mathcal{T}}}_h$ of $\mathcal{T}_h$, with 
$\mathbb{W}_{h}=\mathbb{S}^q(\rcu{\check{\mathcal{T}}}_h)$ \rcu{and $q$ suitably 
chosen depending on $p$ \rc{\cite[Chapter 2]{s2}}. A typical example of the possible 
combinations of degrees $p$ and $q$ would be $(p,q)=(1,0)$ for $0<s \leq 1$.} 
We note that \rcu{the results for such \textit{primal-dual discretizations}
include quasi-uniform meshes and a broad family of non-uniform meshes 
generated via the first local refinement strategy described above.} 
\rcu{Indeed, when $\vert s \vert\leq 1$ and $n=1$, one can prove that \eqref{eq:infsupdh} holds on shape regular algebraically $2$-graded 
meshes, and shape regular geometrically $r$-graded meshes with some 
conditions on the grading parameter $r$ following the arguments
from \cite[Section~4.3]{hju0}. For higher dimensions, one typically verifies 
this numerically.}.
% Unlike for other preconditioners \cite{am1,dirk,fmpp}, a
\rcu{However, the stability requirement \eqref{eq:infsupdh} has not been 
shown for meshes generated via the second local refinement strategy. We 
dedicate the next subsection to address this question.}

On the other hand, recent work by \cite{svv,svv2} offers an alternative 
yet suitable construction for $\widetilde{\mathbb{V}}_{h}$ and $\mathbb{W}_h$ 
which avoids the dual mesh approach. It works for $p=0,1$ and also higher order
polynomials. 
Furthermore, it can also tackle non-uniform meshes with the advantage that it 
requires no mesh conditions besides the so-called K-mesh property.

For $s>1$, there have been no results to the best of the 
authors' knowledge.

%%%%%%%%%%%%%%%%%%%%%%%%%%%%%%%%%%%%%%%%%%%%%%%%%%%%%%%%%%%
\subsection{\rcu{Stability of primal-dual discretization on locally refined meshes}}
\label{sec:adaptive}

\rcu{
In this section, we prove for the first time that operator preconditioning with 
standard primal-dual finite element discretization also leads to bounded 
condition numbers for adaptively refined meshes. 

We believe this is an interesting result on its own account. On the one hand, 
one may argue that adaptive refinements are particularly relevant when thinking 
on a general preconditioning strategy, as they can be implemented with the same 
generality as the preconditioner itself (i.e., no a priori information about the 
geometry, like smoothness or symmetry, is needed to deliver an optimal output). 
On the other hand, this result also implies the $H^1$--stability of a generalized 
$L^2$--projection, a fundamental question of independent interest \cite{by,cc,s}.

As an extensive presentation of adaptivity is outside the purposes of this 
article, %(and we refer the reader to \sj{\cite{nv12}} for \sj{an overview}),
we focus this section on key ideas and keep presentation as concise as possible to 
communicate this novel and relevant extension to a general audience. Nevertheless, 
the interested reader may find the technical details and proofs in 
Appendix~\ref{app:Adaptivity}.

As a proof of concept, we address this question for operator preconditioning 
using classical primal-dual discretization $\widetilde{\mathbb{V}}_{h}= \mathbb{
S}^1(\mathcal{T}_h)\cap\widetilde{H}^{s}(\Omega)\subset\widetilde{H}^{s}(\Omega)$ 
and $\mathbb{W}_{h}=\mathbb{S}^0(\check{\mathcal{T}}_h)\subset {H}^{-s}(\Omega)$ 
for $0\leq s\leq 1$, as introduced in subsection~\ref{ssec:discretization}
\footnote{It is worth pointing out that the same arguments apply to show 
stability for the case $\mathbb{V}_{h}= \mathbb{S}^1(\mathcal{T}_h)\subset{H}^{
s}(\Omega)$ and $\widetilde{\mathbb{W}}_{h}=\mathbb{S}^0(\check{\mathcal{T}}_h
)\subset \widetilde{H}^{-s}(\Omega)$ for $0<s\leq 1$. By duality arguments, this 
will also imply \eqref{eq:infsupdh} for the combination $\widetilde{\mathbb{V}
}_{h}= \mathbb{S}^0(\mathcal{T}_h)\cap\widetilde{H}^{s}(\Omega)\subset\widetilde{
H}^{s}(\Omega)$ and $\mathbb{W}_{h}=\mathbb{S}^1(\check{\mathcal{T}}_h)\subset
H^{-s}(\Omega)$ for $-1\leq s\leq0$.}.
By construction of $\widetilde{\mathbb{V}}_{h}$ and $\mathbb{W}_{h}$, 
\eqref{eq:infsupAh} and \eqref{eq:dim} hold. Therefore, we only need to show 
that the stability requirement \eqref{eq:infsupdh} is satisfied with the chosen 
discretizations to be able to use Theorem~\ref{thm:main} on adaptively refined 
meshes. In order to do this, we briefly introduce some 
general notions about adaptivity.
}

Given an initial triangulation $\mathcal{T}^{(0)}_h$, the adaptive algorithm 
generates a sequence $\mathcal{T}_h^{(\ell)}$ of 
triangulations based \cu{on} error indicators $\eta^{(\ell)}(\tau),\; \tau \in 
\mathcal{T}_h^{(\ell)}$, a refinement criterion and a refinement rule, by 
following the established sequence of steps:
\begin{equation*}
\mathrm{SOLVE}\to\mathrm{ESTIMATE}\to\mathrm{MARK}\to\mathrm{REFINE}.
\end{equation*}
\rcu{This procedure is summarized in the following algorithm}:
\begin{alg}\label{alg:Adaptive}
${\;}$\\
Inputs: Triangulation $\mathcal{T}_h^{(0)}$, refinement parameter $\theta \in (0,
1)$, tolerance $\varepsilon >0$, data $f$.\\
For $\ell = 0, 1, 2 , \dots$
\begin{enumerate}
\item Solve problem \eqref{mainproblem}, for $u_h$ on $\mathcal{T}_h^{(\ell)}$.
\item Compute error indicators $\eta^{(\ell)}(\tau)$ in each triangle $\tau \in 
\mathcal{T}_h^{(\ell)}$.
\item Stop if $\sum_{k} \eta^{(\ell)}(\tau_k) \leq \varepsilon$.
\item Find $\eta^{(\ell)}_{max} = \max_{\tau} \eta^{(\ell)}(\tau).$
\item Mark all $\tau$ with $\eta^{(\ell)}(\tau) > \theta \eta^{(\ell)}_{max}$.
\item Refine each marked triangle to obtain new mesh $\mathcal{T}_h^{(\ell+1)}$.
\end{enumerate}
end\\
Output: Solution $u_{h}$.
\end{alg}

\rcu{Let us assume that we start with an initial triangulation $\mathcal{T}_h^{(0
)}$ such that \eqref{eq:infsupdh} holds for our choice of $\widetilde{\mathbb{V
}}_{h}$ and $\mathbb{W}_{h}$. Clearly, step 6 is the only stage in 
Algorithm~\ref{alg:Adaptive} where one could alter \eqref{eq:infsupdh} for 
subsequent refinements. Therefore, this is the part one has to consider carefully. For the sake of illustration, in this paper we will show how to 
do with this for the \textit{red-green refinement} (see Appendix~\ref{app:Adaptivity} for details).  }\\

\rcu{\begin{lemma}\label{lem:Adaptive}
Let $\mathcal{T}_h^{(0)}$ be a shape regular \rcu{and locally quasi-uniform} 
initial triangulation of $\Omega$. We consider a family of meshes $\Xi:=
\lbrace \mathcal{T}_h^{(\ell)}\rbrace_{\ell \in \mathbb{N}}$ generated from 
$\mathcal{T}_h^{(0)}$ by the adaptive refinement described in 
Algorithm~\ref{alg:Adaptive} \jak{using red-green refinement}. Let $0\leq s\leq 
1$. On each level $\ell \in \mathbb{N}$, we choose $\widetilde{\mathbb{V}}_{
\ell}= \mathbb{S}^1(\mathcal{T}^{(\ell)}_h)\cap\widetilde{H}^{s}(\Omega)$ and 
$\mathbb{W}_{\ell}=\mathbb{S}^0(\check{\mathcal{T}_h}{(\ell)})$. 

Then, \textit{under some mild conditions}\footnote{condition \eqref{eq:c0Cond} in Appendix~\ref{app:Adaptivity}.} on the local quasi-uniformity constant of 
$\mathcal{T}^{(0)}_h$, the following inf-sup constant holds 
\begin{equation}
\underset{\varphi_h\in\mathbb{W}_{\ell}}{\mathrm{sup}} \dfrac{\bd(v_h, \varphi_h)
}{\|\varphi_h\|_{H^{-s}(\Omega)}} \geq \beta_{\bd} \|v_h\|_{\widetilde{H}^s
(\Omega)}, \qquad \text{for all } v_h \in \widetilde{\mathbb{V}}_{\ell},
\end{equation}

for all $\mathcal{T}_h^{(\ell)}\in \Xi$, and with $\beta_{\bd}$ independent of $\ell \in 
\mathbb{N}$. 
\end{lemma}
The proof of this Lemma, together with the incumbent definitions of shape regularity, 
local quasi-uniformity and the mild conditions on $\mathcal{T}_h^{(0)}$ can be 
found in the Appendix~\ref{app:Adaptivity}.

Let us now discuss the result for the generalized $L^{2}$--projection. Let 
$\mathcal{T}_h$ be a mesh of $\Omega$. As before, we consider finite dimensional 
spaces $\widetilde{\mathbb{V}}_{h}\subset \widetilde{H}^s(\Omega)$ and ${\mathbb{
W}}_{h}\subset {H}^{-s}(\Omega)$ for $\,0\leq s \leq1$. We define a generalized 
$L^{2}$--projection $\widetilde{Q}_{h, \mathcal{T}_h}: L^2(\Omega) \rightarrow 
\widetilde{\mathbb{V}}_{h} \subset$ $L^{2}(\Gamma)$ by a Galerkin--Petrov 
variational problem,
\begin{equation}\label{eq:wQh}
\quad\left\langle\widetilde{Q}_{h,\mathcal{T}_h} u, w_{h}\right\rangle_{\Omega}
=\left\langle u, w_{h}\right\rangle_{\Omega} \quad \mathrm{for \; all} \;w_{h} 
\in \mathbb{W}_{h}\ .
\end{equation}
As a direct consequence of Lemma \ref{lem:Adaptive} and \jak{\cite[Theorem 2.2]{s2}}, 
we obtain:
\begin{cor}\label{cor:stab1}
Consider a shape regular triangulation $\mathcal{T}_h^{(0)}$ under the same assumptions as in 
Lemma~\ref{lem:Adaptive}.
Then $\widetilde{Q}_{h,\mathcal{T}_h^{(\ell)}}$ is bounded on $\widetilde{H}^s(\Omega)$, 
with operator norm $\|\widetilde{Q}_{h,\mathcal{T}_h^{(\ell)}}\|_{\mathcal{L}
(H^s(\Omega))} < C$  for all $\ell \in \mathbb{N}$ with a constant $C$ independent of $\ell$.
\end{cor}
Related results for the orthogonal $L^2$--projection have been of interest, e.g.~in the 
analysis of adaptive mesh refinement procedures. \\

\begin{remark}
By \cite{by} the stability condition \eqref{eq:infsupdh} is satisfied when 
$\mathbb{W}_{h} = \widetilde{\mathbb{V}}_{h}$ for $s>0$. Therefore, 
Theorem~\ref{thm:main} also holds in that case.
\end{remark}
}

%%%%%%%%%%%%%%%%%%%%%%%%%%%%%%%%%%%%%%%%%%%%%%%%%%%%%%%%%%%
\subsection{Opposite order preconditioning}
As an alternative to our preconditioner, if $A$ is of order $2s$, one may 
consider to use the bilinear form $\bc_{-s}$ arising from the Dirichlet problem 
\eqref{e:bilin} for the operator $(-\Delta)^{-s}$ to build a preconditioner for 
$\ba_{A}$. In the case of boundary integral equations this approach is 
well-established as Calder\'{o}n preconditioning, specially on closed surfaces. 
For the boundary problems here, we note that the resulting spectral condition 
number \rcc{may} not be $h$-independent, due to the mismatch of the mapping 
properties of the operators. Indeed, \rcu{we obtain the following condition number 
bound in terms of $h$.}

\begin{proposition}\label{prop:blowup}
Let $\vert s \vert \in (1/2,1]$ and set $\widetilde{\mathbb{V}}_{h}=
\mathbb{S}^p(\mathcal{T}_h)\cap\widetilde{H}^{s}(\Omega), \, p=0,1$.
 Let $\tilde{\mathbf{C}}_s$ be the Galerkin matrix induced by $\bc_{-s}$
 \rcu{in $\widetilde{\mathbb{V}}_{h}$}. 
 Then, the following bound on the spectral condition number is 
 satisfied when $h$ is sufficiently small:
 \begin{equation}\label{eq:condBound}
 \kappa\left( \mathbf{D}^{-1} \tilde{\mathbf{C}}_s \mathbf{D}^{-T} 
 \mathbf{A}\right) \leq \mathcal{O}(h^{-2\vert s \vert + 1})\dfrac{C_{\gamma} 
 C_A \Vert \bd \Vert^2}{\beta_{A} \beta_{\gamma} \beta_{\bd}^2},
\end{equation}
where $C_{\gamma}$ and $\beta_\gamma$ are the continuity and \cu{ellipticity} 
constants of $\bc_{-s}$.
\end{proposition}
The proof follows similar arguments to those in \cite{CHR02} and is provided 
in \rcu{Appendix~\ref{app:blowup}. For $s=\pm \frac{1}{2}$ a logarithmic 
 growth of the condition number in $h$ is well-known for Calder\'{o}n preconditioning 
 on screens \cite{hju0,hju}.}

% \rcu{
% It is worth pointing out that the estimate \eqref{eq:condBound} is known to be 
% sharp for $s=1$ \textcolor{black}{[REF!] (HG: Do we really know a reference?)}. 
% Moreover, in the limit case $s=\pm \frac{1}{2}$ a logarithmic 
% growth of the condition number in $h$ is well-known for Calder\'{o}n preconditioning 
% on screens. Therefore, although Proposition~\ref{prop:blowup} only shows an 
% upper bound, there are reasons to expect the condition number to grow like 
% $h^{-2\vert s \vert+ 1}$ for $\vert s \vert 
% \in (1/2,1]$.
% }

%%%%%%%%%%%%%%%%%%%%%%%%%%%%%%%%%%%%%%%%%%%%%%%%%%%%%%%%%%%%%%%%%%%%%%%%%%%%%%%
\section{Numerical Experiments}\label{sec:exp}

\cu{In order to test our preconditioner, we study different pseudodifferential
operators $A$ and implement their bilinear forms $\ba_A$ in}
% We implement the bilinear form $\ba$ associated with the fractional Laplacian in 
$\widetilde{\mathbb{V}}_h=\mathbb{S}^1(\mathcal{T}_h)\cap \widetilde{H}^{s}(\Omega
)$ as described in \cite{ag,hgs}. The bilinear form $\bb_{\rcu{s},\chi}$ is implemented 
in $\mathbb{W}_h = \mathbb{S}^0(\mathcal{T}_h^{\prime})$ on the corresponding 
(barycentric) dual mesh \rc{\cite[Section 3]{hut}}. 

\cu{When operators have singular kernels, as it is the case for $\bb_{\rcu{s},\chi}$, the}
implementations of the bilinear forms 
split the integral into a singular part near $x=y$ and a regular complement. The 
singular integral is evaluated using a composite graded quadrature rule, which 
converts the integral over two elements into an integral over $[0,1]^4$ and 
resolves the singular integral with a geometrically graded composite quadrature 
rule.
The regular part is evaluated using a standard composite quadrature rule. This 
approach is standard in boundary element methods\jak{, see} \cite[Chapter 5]{ss}
\footnote{\jak{\rcc{\texttt{MATLAB}} code for the assembly of the preconditioner 
\rcc{for $s>0$} is available on 
\textit{github.com/nc09jsto/preconditionercode}}. \rcc{The case $s=-1/2$ was 
assembled using \texttt{BETL2} \cite{HIK12}, which currently cannot handle adaptive 
refinements. }}.

\jak{For the specific values of $s$ used in the experiments we employ formulas from 
Remark~\ref{rem:3}. For general values of $s$, due to Lemma~\ref{2f1lemma}, one 
can make use of computational libraries such as \cite[Section 4]{pop} for the 
hypergeometric function ${}_2F_1$.}

Numerical results for the weakly singular and hypersingular operators on 
open curves and surfaces, where $s=\pm \frac{1}{2}$, may be found in \cite{hju0,hju}. 

Here we perform numerical experiments for pseudodifferential operators related 
to the fractional Laplacian on quasi-uniform meshes; on graded triangular meshes, which 
lead to quasi-optimal convergence rates \rc{\cite{ab,hgepsjs}}; and on adaptively 
generated triangular meshes obtained using Algorithm \ref{alg:Adaptive}. In all 
cases we report the achieved spectral condition numbers (denoted as $\kappa$) 
\rcu{and} iterations needed to solve the linear system (labeled \textit{It.}). \cu{We use 
conjugate gradient (CG) when $A$ is symmetric, and GMRES when it is not.}
$N$ denotes the number of degrees of freedom (dofs).
The CG/GMRES iterations were counted until the relative Euclidean norm of the 
residual was $10 ^{-10}$.

\textcolor{black}{Note that we report condition numbers and iteration counts to measure the performance of our preconditioner. A theoretical discussion of runtime complexity is beyond the scope of this work. We mention, however, that implementations which avoid the barycentric dual mesh have been investigated in \cite{svv,svv2} and multilevel preconditioners for negative order operators with linear complexity have been addressed in \cite{svv3}.}

\begin{remark}
For the numerical experiments below, we follow Algorithm~\ref{alg:Adaptive} 
with the following considerations:
\begin{itemize}
\item In step 2, we use the residual error indicators introduced in \cite{ag,hgs}. 
This means: For $\alpha>0$, we approximate the dual norms 
$\| v_h \|_{H^{-\alpha}\rcc{(\rc{\Omega})}}$ and $\|v_h\|_{H^{\alpha}\rcc{(\rc{\Omega})}}$ by the scaled $L^2$-norms 
$h^{\alpha} \|v_h\|_{L^2\rcc{(\rc{\Omega})}}$ and $h^{-\alpha}\|v_h\|_{L^2\rcc{(\rc{\Omega})}}$ , respectively. We 
define the local error indicators $\eta^{(\ell)}(\tau_k)$ for all elements 
$\tau_k \in \mathcal{T}_h^{(\ell)}$:
\fj{\begin{equation*}
\eta^{(\ell)}(\tau_k)^2 := \sum_{i \in \mathcal{N}_h } h_i^{2s} \|(r_h-\bar{
r}_{h})\varphi_i\|^2_{L^2(\omega_i)}\ ,
\end{equation*}}
\cu{wh}ere, $\mathcal{N}_h$ is the set of all vertices \cu{in $\mathcal{T}_h^{(
\ell)}$}, \fj{$r_h := f - (-\Delta)^s u_h$, and $\bar{r}_{h} := \frac{\int_{\omega_i}
r_h\varphi_i}{\int_{\js{\omega_i}}\varphi_i}$} for the interior vertices $i \in 
\mathcal{N}_h$, and $\bar{r}_h=0$ otherwise. Here, $\varphi_i$ is a piecewise linear basis function \rcu{in the span of 
$\widetilde{\mathbb{V}}_{h}$} and $\omega_i := \operatorname{supp}  \, 
\varphi_i $. All integrals are evaluated using a Gauss-Legendre 
quadrature.
 \item In step 6, we use red-green refinement subject to the $1$--irregularity and 
$2$--neighbour rules (see \rcc{Definitions~\ref{def:1-reg}--\ref{def:red-green} in 
Appendix~\ref{app:Adaptivity}} or \cite{bsw} for further details). 
\end{itemize}

\end{remark}

%\jak{???}\\
%\jak{JS: Complexity: Stevenson, van Venetie (2020): Multilevel preconditioner with linear complexity for negative operators. Also earlier work (S,vV) assembly on same mesh + linear complexity for bubbles.}

\begin{figure}[!h]
\centering
\subfloat[][Quasi-uniform]{\includegraphics[width = 0.3\textwidth]{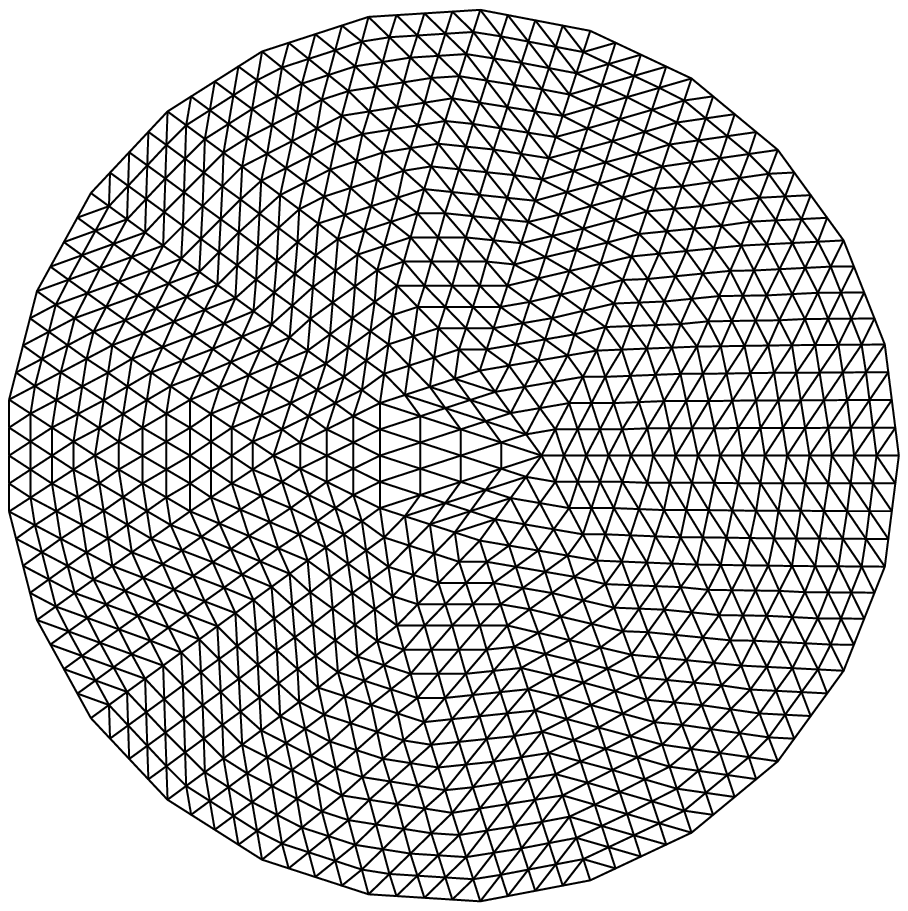}
\label{fig:mesha}}
\subfloat[][$2$-graded]{\includegraphics[width = 0.3\textwidth]{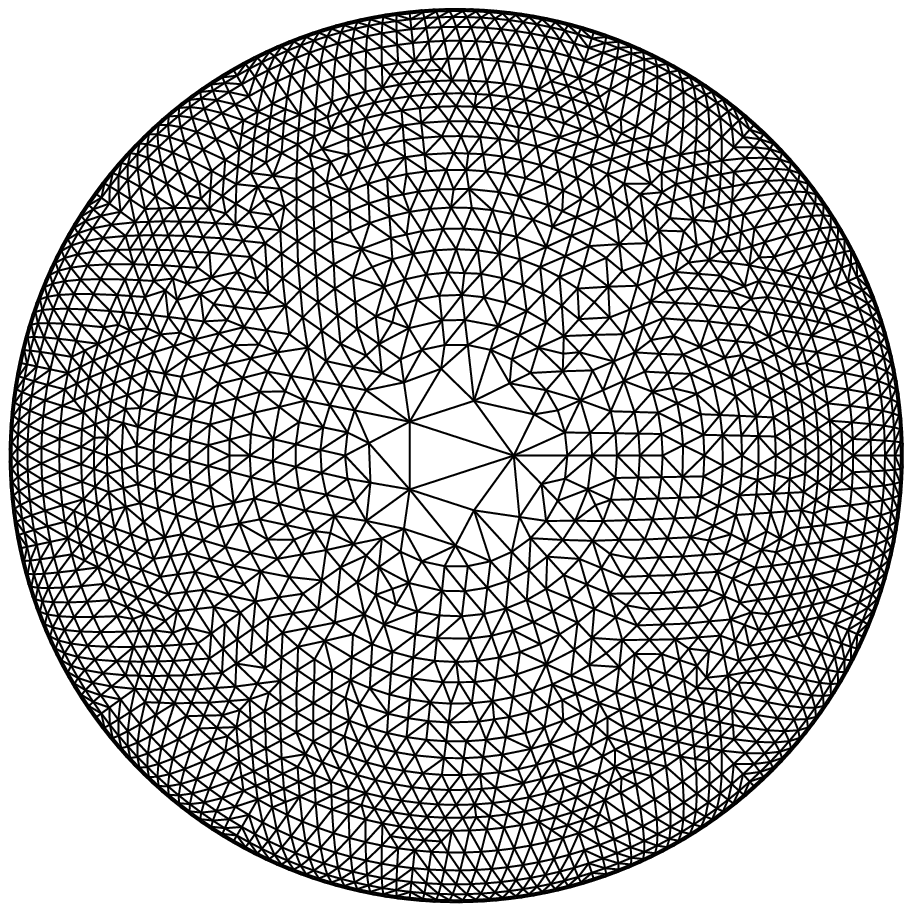}
\label{fig:meshb}}
\subfloat[][adaptively generated]{\includegraphics[width = 0.3\textwidth]{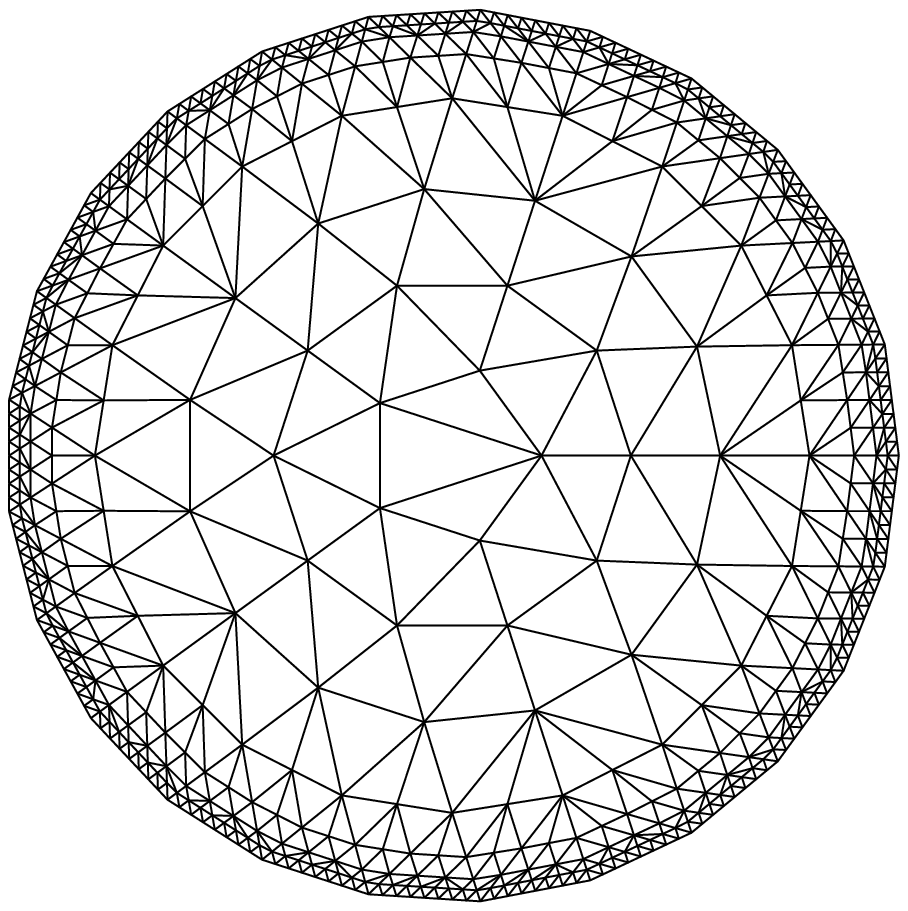}
\label{fig:meshc}}
\caption{Meshes for $\calB_1$.}\label{meshes}
\end{figure}

%------------------------------------------------------------------------------
\begin{example}\label{example1}
We consider the discretization of the Dirichlet problem \eqref{e:FLweakGamma} 
with $A=(-\Delta)^{s}$ and $f = 1$ in the unit disk $\calB_1 \subset\mathbb{R}^2$. 
The exact solution for this problem is given by $u(x) = a_{n,s} (1-|x|^2)^{s}$, 
where $a_{n,s} := \dfrac{\mathsf{\Gamma}(n/2)}{2^{2s} \mathsf{\Gamma}(1+s) 
\mathsf{\Gamma}(s+n/2)}$. $\calB_1$ is approximated by three meshes: 
quasi-uniform, $2$-graded, and adaptively generated triangular meshes as 
depicted in Figure~\ref{meshes}. We consider fractional exponents $s=-\frac{1}{2}, \frac{1}{4}, 
\frac{7}{10}, \frac{3}{4}$, to indicate the general applicability of our methods.
\end{example}

Tables~\ref{t:Uniform_ex1}--\ref{t:Adaptive_ex1} show the results of the Galerkin 
matrix $\mathbf{A}$ and its preconditioned form $\mathbf{P}\mathbf{A}$ for the 
different fractional exponents on the three families of meshes under 
consideration\rcu{.} %\rcu{(see Figure~\ref{meshes}).}

On all three classes of meshes, the condition number and the number of solver 
iterations for $\mathbf{A}$ show the expected strong growth when increasing $N$, 
while they are small and bounded for $\mathbf{P}\mathbf{A}$. We remark that the
reduction of CG iterations achieved by our preconditioner is significant, with 
a higher reduction for larger \jak{$|s|$}. Furthermore, $\kappa(\mathbf{P}\mathbf{A})$ 
remains almost constant across the refinement levels when $s=\frac{1}{4}$. We 
note, however, a very slow growth for $s=\frac{7}{10}$ and $s=\frac{3}{4}$ for 
the considered dofs. \jak{For $s=-\frac{1}{2}$ we obtain larger condition numbers consistent with previous observations \cite{hju}. We note the larger condition number for the last data point on $2$-graded meshes of the preconditioned problem. We attribute this to a discretization error of the particular implementation.}  \sj{Even though we use the exact inverse on the unit disk 
\cu{to build our} preconditioner, \cu{it is worth noticing that in this case 
$\mathbf{P}\mathbf{A}$ achieves only an approximate identity after
discretization. This approximation error, together with the tolerance of $10^{-10}$ 
for the residual, explain why condition numbers and CG iteration counts are larger 
than 1.}}

% \begin{table}[h!]\centering \small
% \caption{Condition numbers and GMRES/CG iterations
% on quasi-uniform mesh \cu{(Figure~\ref{fig:mesha})}, Example~\ref{example1}.}\label{t:Uniform_ex1}
%  \centering
%  \resizebox{\columnwidth}{!}{
% \begin{tabular}{|c|cc|cc|cc|cc|cc|cc|}\hline
%  & \multicolumn{4}{c|}{$s=1/4$} &
% \multicolumn{4}{c|}{$s=7/10$} & \multicolumn{4}{c|}{$s=3/4$} \\ \cline{2-13}
% \rule{0pt}{11pt} N &
% \multicolumn{2}{c|}{$\mathbf{A}$} & \multicolumn{2}{c|}{$\mathbf{P}\mathbf{A}$}
% &\multicolumn{2}{c|}{$\mathbf{A}$} & \multicolumn{2}{c|}{$\mathbf{P}\mathbf{A}$}
% &\multicolumn{2}{c|}{$\mathbf{A}$} & \multicolumn{2}{c|}{$\mathbf{P}\mathbf{A}$}
% \\ \cline{2-13}
% \rule{0pt}{11pt}& $\kappa$ & It. & $\kappa$ & It.
% & $\kappa$ & It. & $\kappa$ & It. & $\kappa$ & It. & $\kappa$ & It. \\ \hline
% 123 & 2.0 & 12/12 & 1.2 & 6/6 & 6.9 & 15/15 & 1.5 & 9/9 & 8.2 & 16/16 & 1.5 & 9/10 \\
% 492 & 2.7 & 13/13 & 1.2 & 7/7 & 20.9 & 33/28 & 1.5 & 9/10 & 27.0 & 32/30 & 1.5 & 10/10\\
% 1968 & 4.1 & 16/16 & 1.3 & 7/7 & 62.1 & 46/47 & 1.6 & 10/10 & 87.2 & 51/51 & 1.7 & 10/11\\
% 7872 & 6.3 & 21/21 & 1.3 & 7/7 & 176.2 & 78/79 & 1.8 & 10/11 & 268.0 & 91/92 & 2.1 & 10/12 \\
% 31488 & 9.4 & 27/27 & 1.3 & 7/7 & 478.8 & 134/135 & 1.9 & 11/11 & 784.2 & 160/160 & 2.6 & 11/12 \\ \hline
%  \end{tabular}
%  }
% \end{table}
\begin{table}[h!]\centering \footnotesize
\caption{Condition numbers and CG iterations on quasi-uniform mesh (Figure~\ref{fig:mesha}), 
Example~\ref{example1}.\\}\label{t:Uniform_ex1}
 \centering \setlength{\tabcolsep}{3.3pt}
\begin{tabular}{|c|cc|cc|cc|cc|cc|cc|cc|cc|}\hline
 & \multicolumn{4}{c|}{$s=-1/2$} & \multicolumn{4}{c|}{$s=1/4$} & 
\multicolumn{4}{c|}{$s=7/10$} & \multicolumn{4}{c|}{$s=3/4$} \\ \cline{2-17}
\rule{0pt}{11pt} N & 
\multicolumn{2}{c|}{$\mathbf{A}$} & \multicolumn{2}{c|}{$\mathbf{P}\mathbf{A}$} 
&
\multicolumn{2}{c|}{$\mathbf{A}$} & \multicolumn{2}{c|}{$\mathbf{P}\mathbf{A}$} 
&\multicolumn{2}{c|}{$\mathbf{A}$} & \multicolumn{2}{c|}{$\mathbf{P}\mathbf{A}$}
&\multicolumn{2}{c|}{$\mathbf{A}$} & \multicolumn{2}{c|}{$\mathbf{P}\mathbf{A}$}
\\ \cline{2-17}
\rule{0pt}{11pt}& $\kappa$ & It. & $\kappa$ & It. & $\kappa$ & It. & $\kappa$ & It. & $\kappa$ & It. & $\kappa$ & It. & $\kappa$ & It. & $\kappa$ & It. \\ \hline
  123 & 35.63 & 27 & 2.61 & 12 & 1.98 & 12 & 1.16 & 6 &   6.85 &  15 & 1.50 & 9 &8.24 & 16 & 1.54 & 10 \\
  492 & 73.58 & 40 & 2.69 & 12 & 2.65 & 13 & 1.20 & 7 &  20.87 &  28 & 1.52 & 10 &  26.99 &  30 & 1.54 & 10\\
 1968 & 153.95 & 56 & 2.74 & 13 & 4.11 & 16 & 1.25 & 7 &  62.10 &  47 & 1.56 & 10 &  87.24 &  51 & 1.72 & 11\\
 7872 & 316.74 & 78 & 2.78 & 13 & 6.34 & 21 & 1.26 & 7 & 176.19 &  79 & 1.76 & 11 & 268.02 &  92 & 2.14 & 12\\ 
31488 & 643.01 & 131 & 2.83 & 14 & 9.36 & 27 & 1.28 & 7 & 478.78 & 135 & 1.93 & 11 & 784.22 & 160 & 2.57 & 12\\ \hline
 \end{tabular}
\end{table}

% \begin{table}[h!]\centering \small
% \caption{Condition numbers and GMRES/CG iterations
% on $2$-graded mesh \cu{(Figure~\ref{fig:meshb})}, Example~\ref{example1}.}\label{t:Graded_ex1}
%  \centering
%  \resizebox{\columnwidth}{!}{
% \begin{tabular}{|c|c c|c c|c c|c c|c c|c c|}\hline
%  & \multicolumn{4}{c|}{$s=1/4$}  &
% \multicolumn{4}{c|}{$s=7/10$} & \multicolumn{4}{c|}{$s=3/4$} \\ \cline{2-13}
% \rule{0pt}{11pt} N &
% \multicolumn{2}{c|}{$\mathbf{A}$} & \multicolumn{2}{c|}{$\mathbf{P}\mathbf{A}$}
% &\multicolumn{2}{c|}{$\mathbf{A}$} & \multicolumn{2}{c|}{$\mathbf{P}\mathbf{A}$}
% &\multicolumn{2}{c|}{$\mathbf{A}$} & \multicolumn{2}{c|}{$\mathbf{P}\mathbf{A}$}
% \\ \cline{2 - 13}
% \rule{0pt}{11pt}&  $\kappa$ & It. & $\kappa$ & It.
% & $\kappa$ & It. & $\kappa$ & It. & $\kappa$ & It. & $\kappa$ & It. \\ \hline
% 123 & 8.4 & 22/20 & 1.1 & 6/6  & 4.5 & 16/16 & 1.7 & 10/11 & 5.2 & 16/16 & 1.9 & 11/12 \\
% 1068 & 23.3 & 34/36 & 1.2 & 7/7 & 28.3 & 32/32 & 2.4 & 13/14 & 33.6 & 33/34 & 2.9 & 14/14 \\
% 4645 & 41.6 & 43/44 & 1.3 & 7/7 & 106.5 & 69/70 & 2.9 & 13/14 & 133.3 & 77/75 & 3.7 & 15/15 \\
% 13680 & 63.5 & 49/48 & 1.3 & 7/7 & 282.6 & 100/99 & 3.0 & 13/14 & 364.1 & 115/116 & 3.9 & 15/16 \\ \hline
%  \end{tabular}}
% \end{table}
\begin{table}[h!]\centering \footnotesize
\caption{Condition numbers and CG iterations on $2$-graded mesh 
(Figure~\ref{fig:meshb}), Example~\ref{example1}.\\}\label{t:Graded_ex1}
 \centering \setlength{\tabcolsep}{3.0pt}
\begin{tabular}{|c|c c|c c|c c|c c|c c|c c|c c|c c|}\hline
 & \multicolumn{4}{c|}{$s=-1/2$} & \multicolumn{4}{c|}{$s=1/4$}  & 
\multicolumn{4}{c|}{$s=7/10$} & \multicolumn{4}{c|}{$s=3/4$} \\ \cline{2-17}
\rule{0pt}{11pt} N & 
\multicolumn{2}{c|}{$\mathbf{A}$} & \multicolumn{2}{c|}{$\mathbf{P}\mathbf{A}$} 
&\multicolumn{2}{c|}{$\mathbf{A}$} & \multicolumn{2}{c|}{$\mathbf{P}\mathbf{A}$} 
&\multicolumn{2}{c|}{$\mathbf{A}$} & \multicolumn{2}{c|}{$\mathbf{P}\mathbf{A}$}
&\multicolumn{2}{c|}{$\mathbf{A}$} & \multicolumn{2}{c|}{$\mathbf{P}\mathbf{A}$}
\\ \cline{2 - 17}
\rule{0pt}{11pt}&  $\kappa$ & It. & $\kappa$ & It. &  $\kappa$ & It. & $\kappa$ & It. 
& $\kappa$ & It. & $\kappa$ & It. & $\kappa$ & It. & $\kappa$ & It. \\ \hline
  123 & 35.63 &  27 & 2.61 & 12 & 8.41 & 20 & 1.14 & 6 &   4.53 & 16 & 1.72 & 11 &   5.17 &  16 & 1.94 & 12 \\
 1068 & 8190.98 & 255 & 4.92 & 20 & 23.33 & 36 & 1.21 & 7 &  28.33 & 32 & 2.42 & 14 &  33.57 &  34 & 2.92 & 14 \\
 4645 & 24657.62 & 431 & 6.17 & 22 & 41.63 & 44 & 1.25 & 7 & 106.53 & 70 & 2.85 & 14 & 133.26 &  75 & 3.65 & 15 \\
13680 & 58165.89 & 620 & 9.25 & 26 & 63.52 & 48 & 1.27 & 7 & 282.57 & 99 & 2.97 & 14 & 364.14 & 116 & 3.87 & 16 \\ \hline
 \end{tabular}
\end{table}

% \begin{table}[h!]\centering \small
% \caption{Condition numbers and GMRES/CG iterations
% on adaptively generated meshes \cu{(Figure~\ref{fig:meshc})}, Example~\ref{example1}.}\label{t:Adaptive_ex1}
%  \centering
%  \resizebox{\columnwidth}{!}{
% \begin{tabular}{|c|c c|c c|c c|c c|c c|c c|}\hline
%  & \multicolumn{4}{c|}{$s=1/4$}  &
% \multicolumn{4}{c|}{$s=7/10$} & \multicolumn{4}{c|}{$s=3/4$} \\ \cline{2-13}
% \rule{0pt}{11pt} N &
% \multicolumn{2}{c|}{$\mathbf{A}$} & \multicolumn{2}{c|}{$\mathbf{P}\mathbf{A}$}
% &\multicolumn{2}{c|}{$\mathbf{A}$} & \multicolumn{2}{c|}{$\mathbf{P}\mathbf{A}$}
% &\multicolumn{2}{c|}{$\mathbf{A}$} & \multicolumn{2}{c|}{$\mathbf{P}\mathbf{A}$}
% \\ \cline{2 - 13}
% \rule{0pt}{11pt}&  $\kappa$ & It. & $\kappa$ & It.
% & $\kappa$ & It. & $\kappa$ & It. & $\kappa$ & It. & $\kappa$ & It. \\ \hline
% 123 & 2.0 & 12/12 & 1.2 & 6/6  & 6.9 & 15/15 & 1.5 & 9/10 & 8.2 & 16/16 & 1.5 & 9/9 \\
% 238 & 5.4 & 22/22 & 1.2 & 6/6 & 7.8 & 21/21 & 1.6 & 10/10 & 9.2 & 21/21 & 1.7 & 11/11 \\
% 518 & 15.5 & 37/37 & 1.2 & 7/7 & 11.3 & 28/28 & 1.8 & 11/11 & 12.6 & 28/29 & 1.9 & 12/12 \\
% 1098 & 45.3 & 58/58 & 1.2 & 7/7 & 17.5 & 37/37 & 1.8 & 11/11 & 18.2 & 37/38 & 2.0 & 12/12 \\
% 2278 & 131.8 & 85/85 & 1.2 & 7/7 & 28.3 & 47/48 & 1.9 & 12/12 & 27.2 & 47/48 & 2.2 & 13/13 \\
% 4658 & 387.0 & 121/121 & 1.3 & 8/8 & 46.7 & 64/65 & 2.0 & 12/12 & 41.5 & 60/61 & 2.4 & 14/14 \\
% 9438 & 1138.7 & 165/165 & 1.3 & 8/8 & 78.4 & 84/85 & 2.1 & 13/13 & 64.3 & 75/77 & 2.5 & 14/14 \\ \hline
% \end{tabular}}
% \end{table}
\begin{table}[h!]\centering \footnotesize
\caption{Condition numbers and CG iterations on adaptively generated meshes
(Figure~\ref{fig:meshc}), Example~\ref{example1}.\\}\label{t:Adaptive_ex1}
 \centering 
\begin{tabular}{|c|c c|c c|c c|c c|c c|c c|}\hline
 & \multicolumn{4}{c|}{$s=1/4$}  & 
\multicolumn{4}{c|}{$s=7/10$} & \multicolumn{4}{c|}{$s=3/4$} \\ \cline{2-13}
\rule{0pt}{11pt} N & 
\multicolumn{2}{c|}{$\mathbf{A}$} & \multicolumn{2}{c|}{$\mathbf{P}\mathbf{A}$} 
&\multicolumn{2}{c|}{$\mathbf{A}$} & \multicolumn{2}{c|}{$\mathbf{P}\mathbf{A}$}
&\multicolumn{2}{c|}{$\mathbf{A}$} & \multicolumn{2}{c|}{$\mathbf{P}\mathbf{A}$}
\\ \cline{2 - 13}
\rule{0pt}{11pt}&  $\kappa$ & It. & $\kappa$ & It. 
& $\kappa$ & It. & $\kappa$ & It. & $\kappa$ & It. & $\kappa$ & It. \\ \hline
 123 &    1.98 &  12 & 1.16 & 6 &  6.85 & 15 & 1.50 & 10 &  8.24 & 16 & 1.54 &  9 \\
 238 &    5.39 &  22 & 1.17 & 6 &  7.82 & 21 & 1.60 & 10 &  9.22 & 21 & 1.67 & 11 \\
 518 &   15.46 &  37 & 1.20 & 7 & 11.27 & 28 & 1.76 & 11 & 12.55 & 29 & 1.89 & 12 \\
1098 &   45.30 &  58 & 1.21 & 7 & 17.53 & 37 & 1.83 & 11 & 18.15 & 38 & 2.01 & 12 \\ 
2278 &  131.77 &  85 & 1.23 & 7 & 28.28 & 48 & 1.91 & 12 & 27.17 & 48 & 2.16 & 13 \\
4658 &  386.95 & 121 & 1.26 & 8 & 46.65 & 65 & 2.00 & 12 & 41.48 & 61 & 2.35 & 14 \\
9438 & 1138.72 & 165 & 1.27 & 8 & 78.41 & 85 & 2.08 & 13 & 64.30 & 77 & 2.50 & 14 \\ \hline
\end{tabular}
\end{table}

To gain further insight about this small growth in $\kappa(\mathbf{P}\mathbf{A
})$, we also inspect the eigenvalues of $\mathbf{A}$ and $\mathbf{P}\mathbf{A}$ 
for the two families of meshes where this behaviour is more notorious. These are 
displayed in Figure~\ref{f:eigvals_ex1}. We see in plots $(a), (c), (e)$ that 
the spectra on quasi-uniform meshes are as expected, while on graded meshes, 
plots $(b), (d), (f)$ reveal that the clustering of eigenvalues for the 
preconditioned matrix still increases slowly with the dofs. As the slope of 
this small growth tends to $0$ when augmenting the number of dofs, we attribute 
it to the preasymptotic regime. 

%\cc{CU: I removed the comment regarding GMRES because 
%it was actually not 100\% true and I prefer to avoid potential complains...
%
%On the other hand (related to the comment on the previous paragraph regarding approximate 
%identitity): I suspect that the preasymptotic phase is bigger for larger 
%$s$ due to the stronger singularity and (I presume) the consequent larger quadrature 
%error. A simple test for this hypothesis would be to run with more quadrature points 
%(say for the first three sets of meshes). Not sure how much effort this would be...
%Do you think it would be worth it?
%}
%
%\sj{JS: Let's do this if a referee requests it... }

\begin{figure}[!ht]
\centering
\subfloat[][Uniform mesh, $s=\frac{1}{4}$]
{\includegraphics[width = 0.43\textwidth]{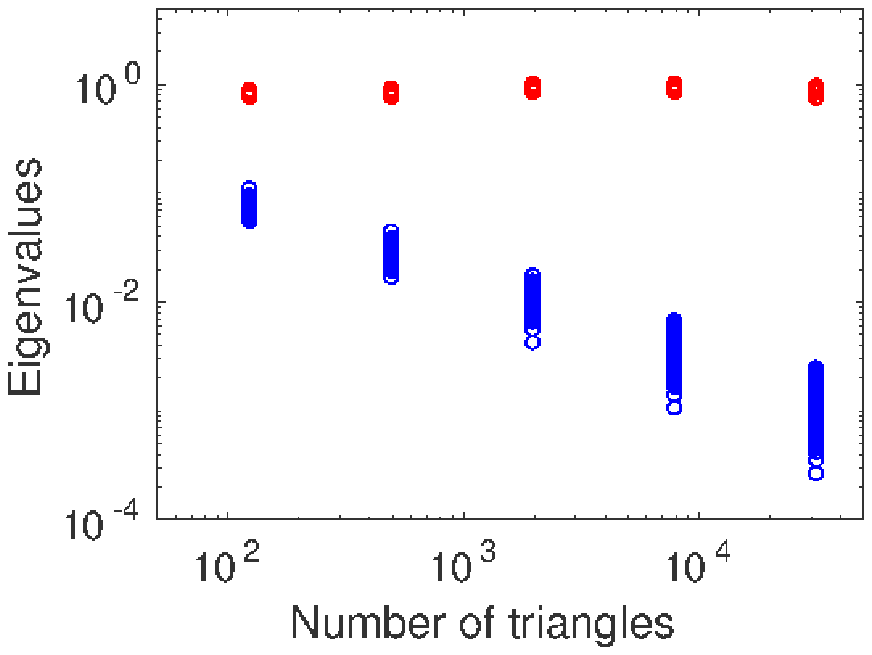}}
\subfloat[][$2$--graded mesh, $s = \frac{1}{4}$]
{\includegraphics[width = 0.43\textwidth]{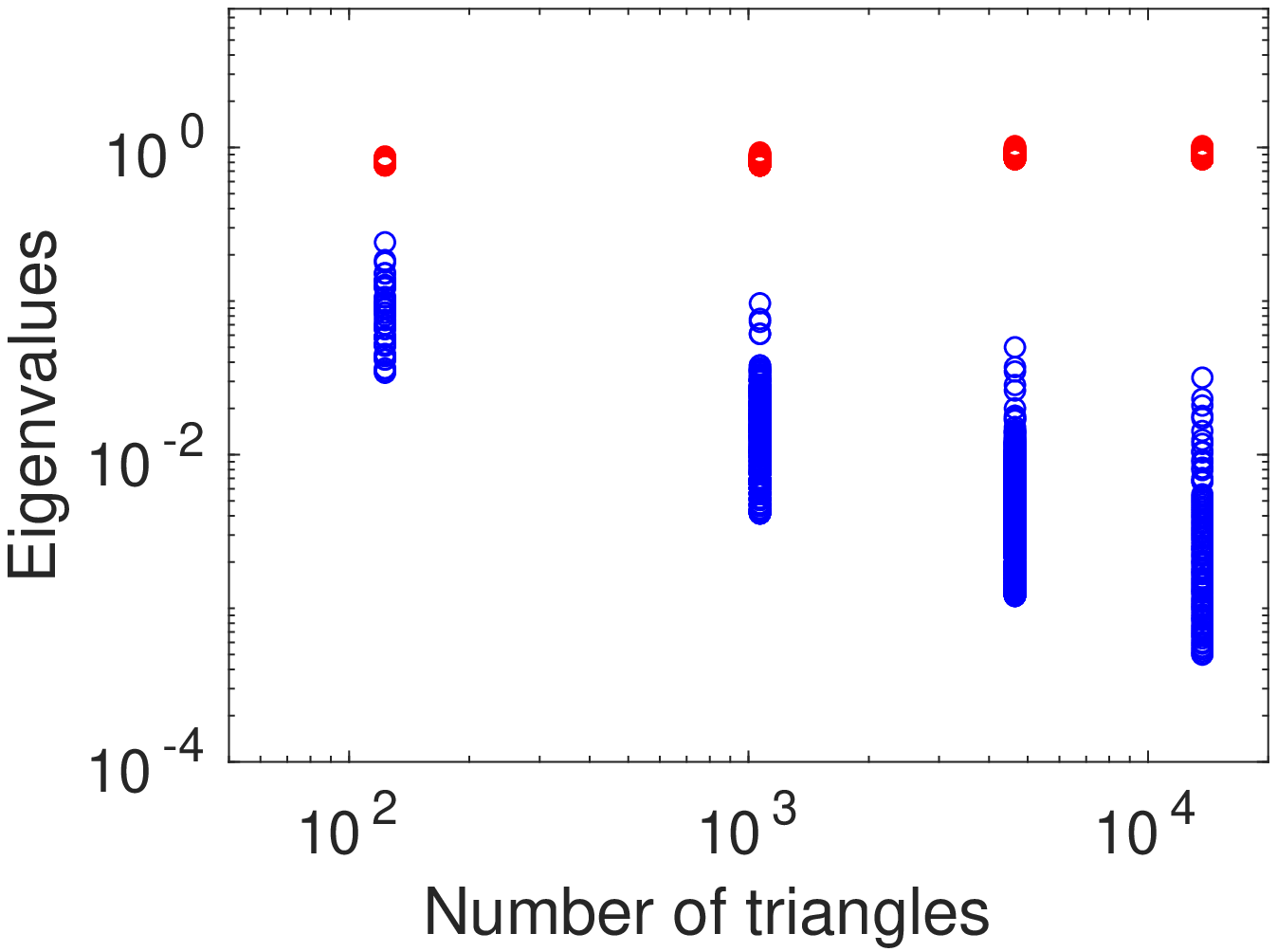}}\\

\subfloat[][Uniform mesh, $s=\frac{7}{10}$]
{\includegraphics[width = 0.43\textwidth]{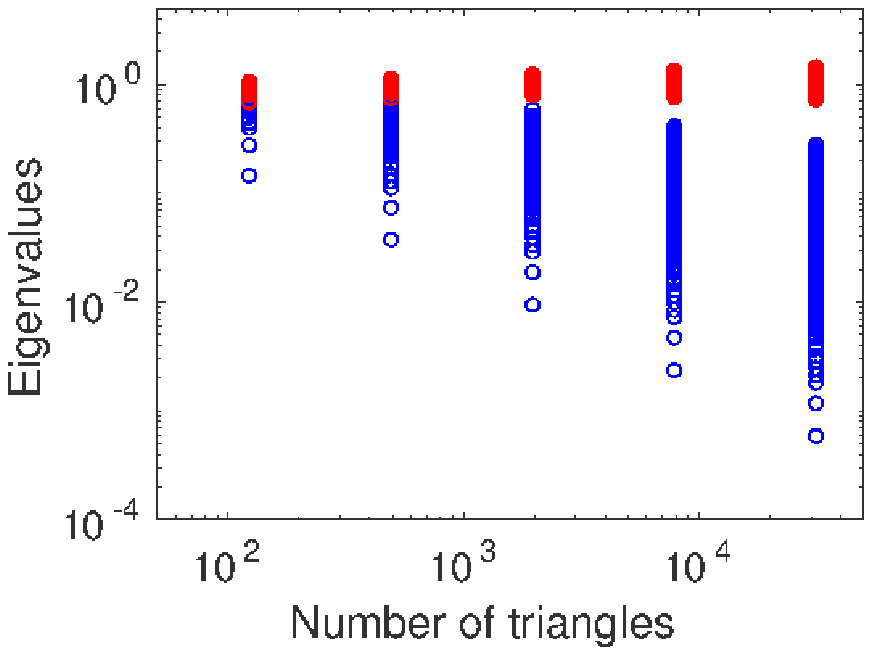}}
\subfloat[][$2$--graded mesh, $s = \frac{7}{10}$]
{\includegraphics[width = 0.43\textwidth]{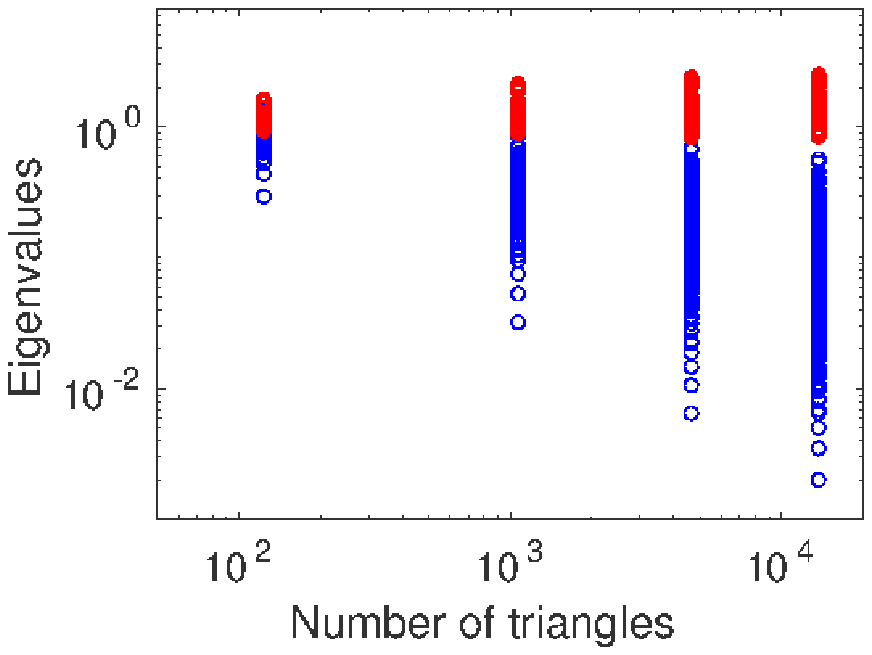}}\\

\subfloat[][Uniform mesh, $s=\frac{3}{4}$]
{\includegraphics[width = 0.43\textwidth]{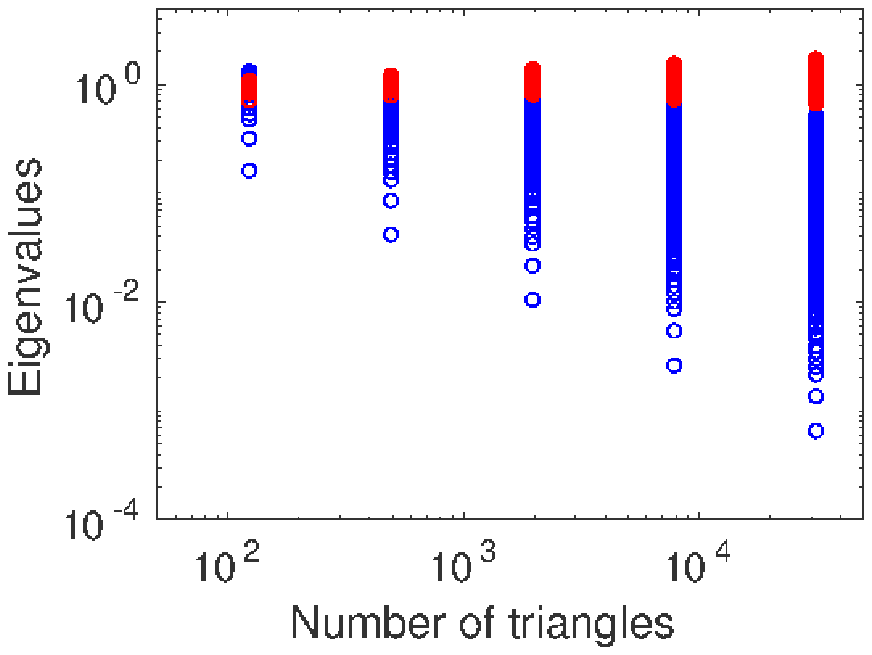}}
\subfloat[][$2$--graded mesh, $s = \frac{3}{4}$]
{\includegraphics[width = 0.43\textwidth]{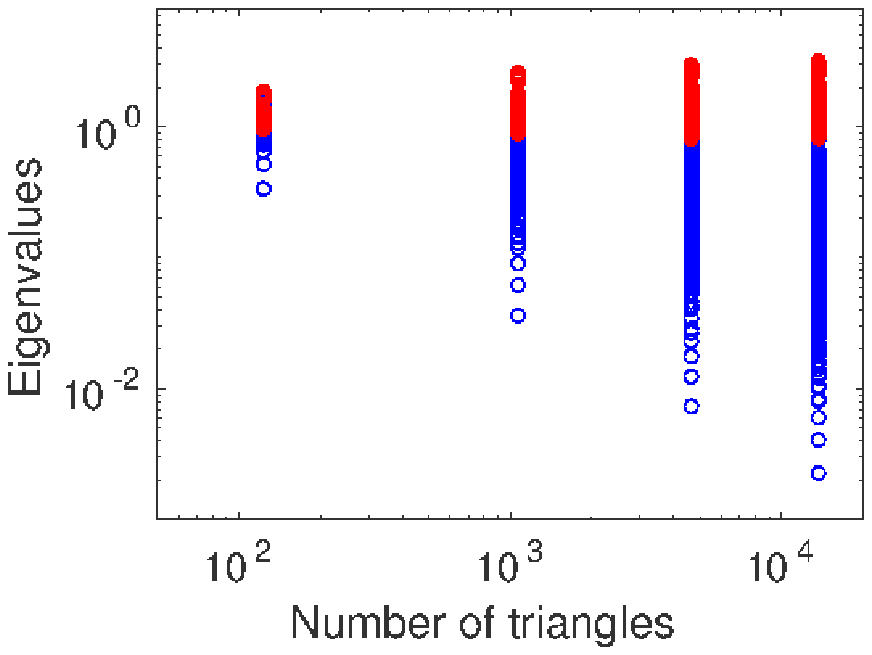}}
\caption{Eigenvalues of $\mathbf{A}$ (blue), resp.~$\mathbf{P}\mathbf{A}$ (red), 
Example \ref{example1}. }\label{f:eigvals_ex1}
\end{figure}

%------------------------------------------------------------------------------

The next example illustrates the performance of the preconditioner defined by 
the bilinear form \eqref{e:b_mapped} on a domain bi-Lipschitz to $\calB_1$.

\begin{example}\label{example1a} We consider the discretization of the Dirichlet 
problem \eqref{e:FLweakGamma} with $A=(-\Delta)^{s}$ and $f = 1$ in the L-shaped 
domain $\Omega = [-1,3]^2 \setminus [1,3]^2 \subset \mathbb{R}^2$ depicted in 
Figure~\ref{f:Lshape}a. We examine fractional exponents $s = \frac{1}{4}, \frac{1}
{2}, \frac{3}{4}$ on quasi-uniform, geometrically and algebraically graded meshes, 
see Figure~\ref{f:graded_meshes} for an illustration. A numerical solution on a 
mesh with $3968$ elements is shown in Figure~\ref{f:Lshape}b. The preconditioner 
is computed using the radial projection $\chi$ from the L-shaped domain to 
$\calB_1$. Here,
\begin{align*}
\chi: \Omega \to \mathcal{B}_1, \quad \chi(x) = \frac{1}{r(x)}\frac{x}{|x|},
\end{align*}
where $r(x) := \sup \lbrace \lambda \in [1,\infty): \lambda x \in \Omega \rbrace$.

%\cc{CU: Now that we have space: shall we write the transformation for reproducibility?}

\end{example}

Tables~\ref{tab:UN}--\ref{tab:SRAG} display the results of the Galerkin matrix 
$\mathbf{A}$ and its preconditioned form $\mathbf{P}\mathbf{A}$  on a sequence 
of corresponding meshes. As in the unit disk $\calB_1$ in Example \ref{example1},
the condition number and the number of solver iterations for $\mathbf{A}$ show 
a strong increase with augmenting the dofs $N$, while the growth is small and 
of slope tending to $0$ for $\mathbf{P}\mathbf{A}$. We also note that the size 
of the condition numbers is slightly bigger than those from 
Example~ \ref{example1}. This is a consequence of the fact that the 
preconditioner is no longer defined from an exact solution operator to the 
continuous problem, and thus the bound on the condition number is $h$-independent, 
yet larger than in the previous example. Indeed, as predicted by the theory, we 
see that the condition numbers and CG iterations obtained with the preconditioner 
remain small and bounded on quasi-uniform and geometrically graded meshes. 
However, the condition numbers of $\mathbf{P}\mathbf{A}$ for the 
algebraically graded meshes (Figure~\ref{fig:agLshno}) do not remain bounded. 
This is consistent with the theory, which applies to shape regular meshes, a 
condition not satisfied here. In order to \js{illustrate this further}, we 
also study a shape regular variant of the algebraically graded meshes 
(Figure~\ref{fig:agLshyes}). The obtained results are reported in 
Table~\ref{tab:SRAG}, which reveals that the condition numbers are bounded 
again. \textcolor{black}{We point out that the assumptions of Theorem \ref{thm:main} are satisfied under certain mesh conditions introduced in Appendix A.2. The algebraically graded meshes from 
Figure~\ref{fig:agLshno}) violate the shape regularity condition $(C1)$ (and also condition $(C3)$ for $s=\frac{3}{4}$), 
while all other meshes considered verify all mesh conditions}.%\rcc{[CU: Update?? Mmm]}\rc{JS: \ref{app:conditions} ???} 
\\
\\

%\js{(JS: Just a quick remark:\\
% Quasi-uniform meshes: $c_R \approx 0.56$ and $c_L \approx 1.29$, according to the $c_R$ dependent estimate, everything holds for $|s| \leq 1$. \\
%Algebraically graded meshes: $c_R$ degenerates: $0.26, 0.18, 0.14, 0.10, \dots$ and $c_L \approx 3$, according to the $c_R$ dependent estimate, doesn't hold for $s =3/4$, for any of the meshes, and is ok for $s \leq 1/2$ for the meshes here. \\
%Geometrically graded meshes: $c_R =0.5$ and $c_L\approx 2.83$, according to the $c_R$ dependent estimate, everything holds for $|s| \leq 3/4$\\
%Algebraically graded, shape regular meshes: $c_R \approx 0.38$ and $c_L \approx 3$, according to the $c_R$ dependent estimate, everything holds for $|s| \leq 3/4$.\\
%For $c_R$ dependent estimate, see \textit{MeshConditions.pdf} in the Dropbox folder :). )}

%\js{Figures~\ref{f:lshape_GG_UN} and \ref{f:lshape_AG} show the spectra of the Galerkin and preconditioned Galerkin matrices for this problem on quasi-uniform, geometrically and algebraically graded meshes. Note that the spectrum of the preconditioned Galerkin matrices does not stay bounded in the case of algebraically graded meshes.}

\begin{figure}[h!]
\centering
\subfloat[][Quasi-uniform\\]{\includegraphics[width = 0.24\textwidth]{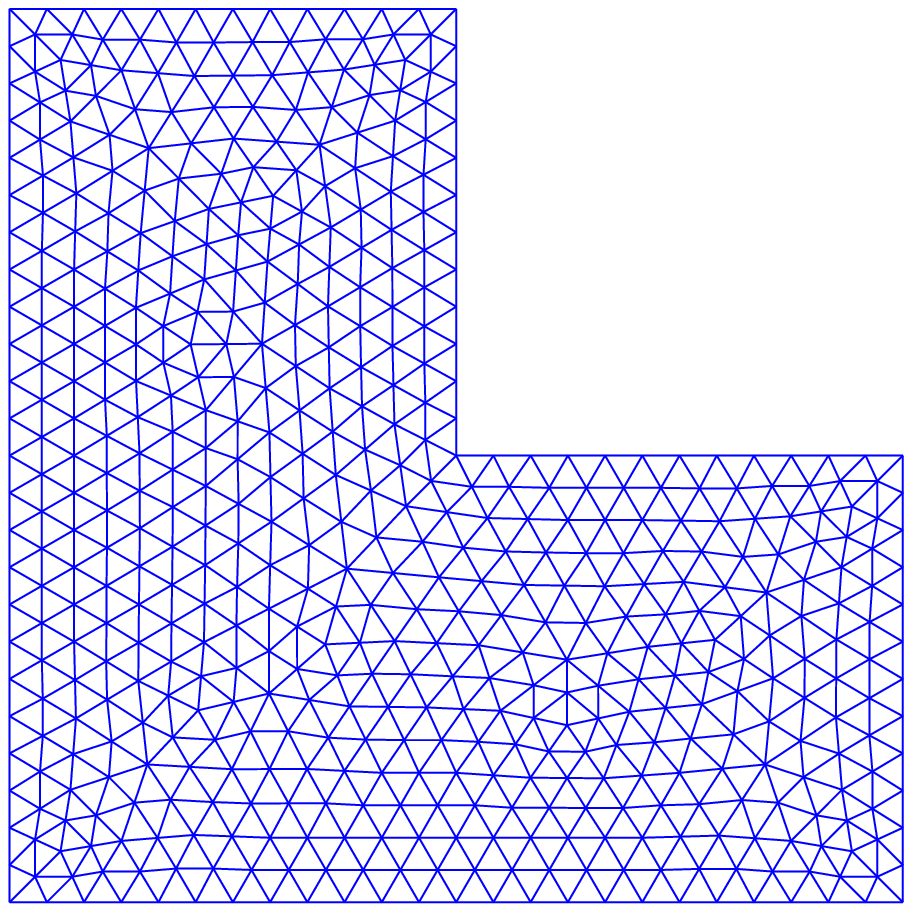}
\label{fig:unLsh}}
\subfloat[][Geometrically\\ \hphantom{(b)} graded]{\includegraphics[width = 0.24\textwidth]{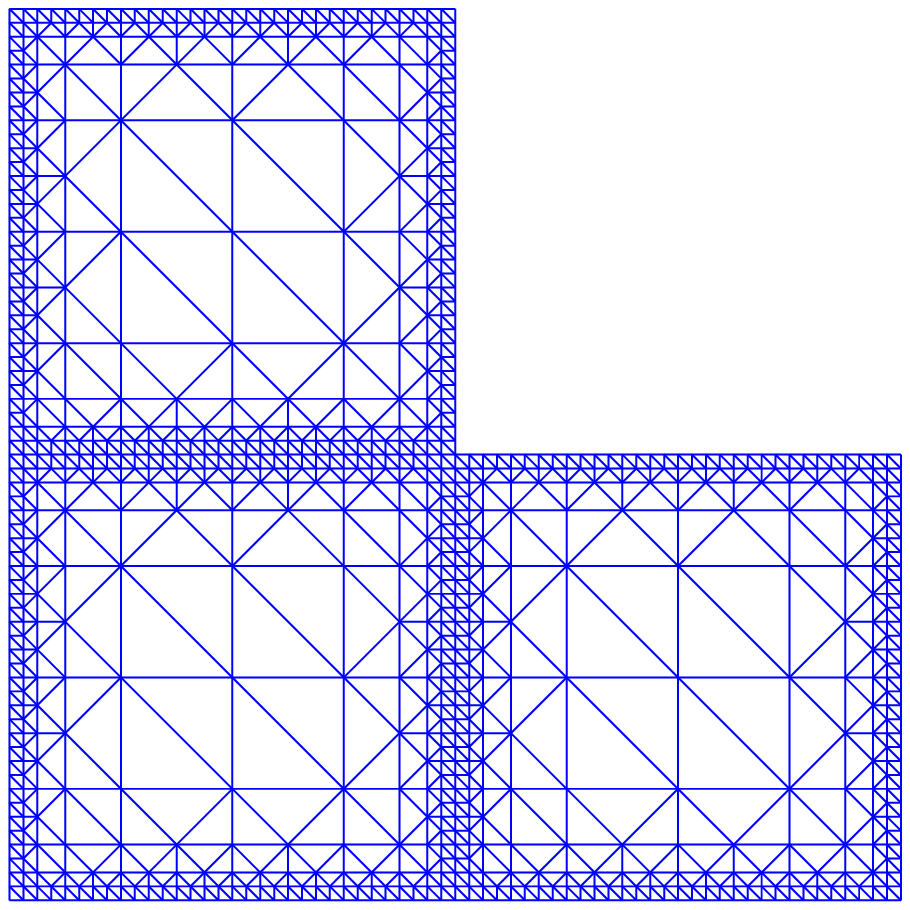}
\label{fig:ggLsh}}
\subfloat[][Algebraically \\\hphantom{(c)} $2$-graded]{\includegraphics[width = 0.24\textwidth]{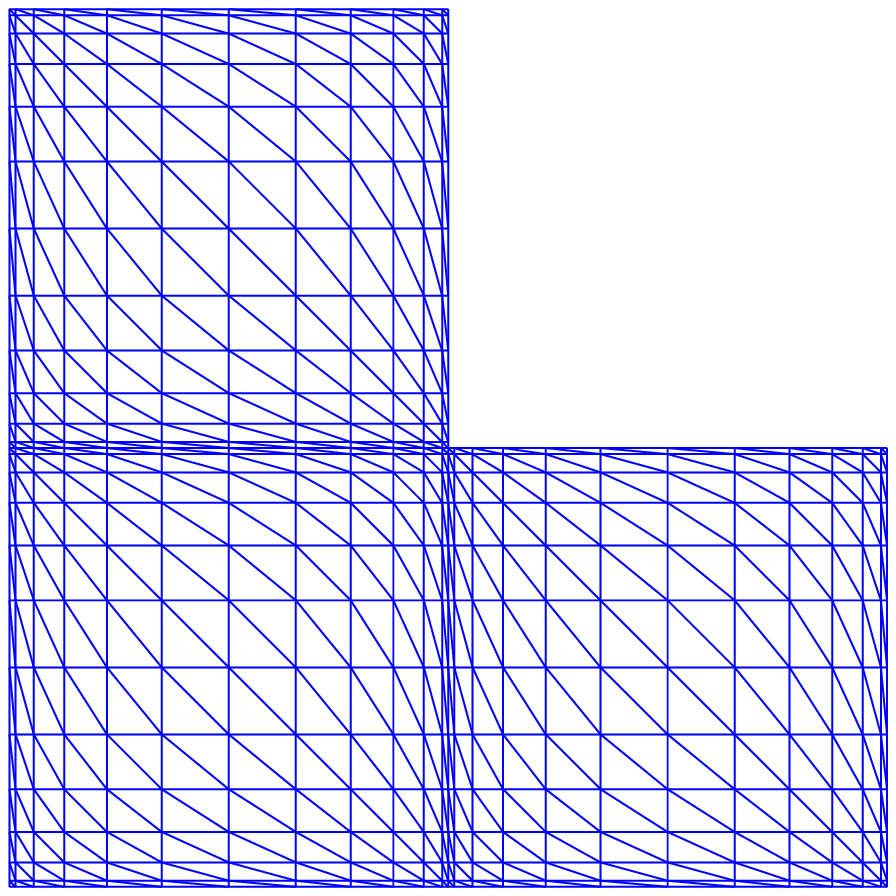}
\label{fig:agLshno}}
\subfloat[][Algebraically\\$2$-graded shape regular]{\includegraphics[width = 0.24\textwidth]{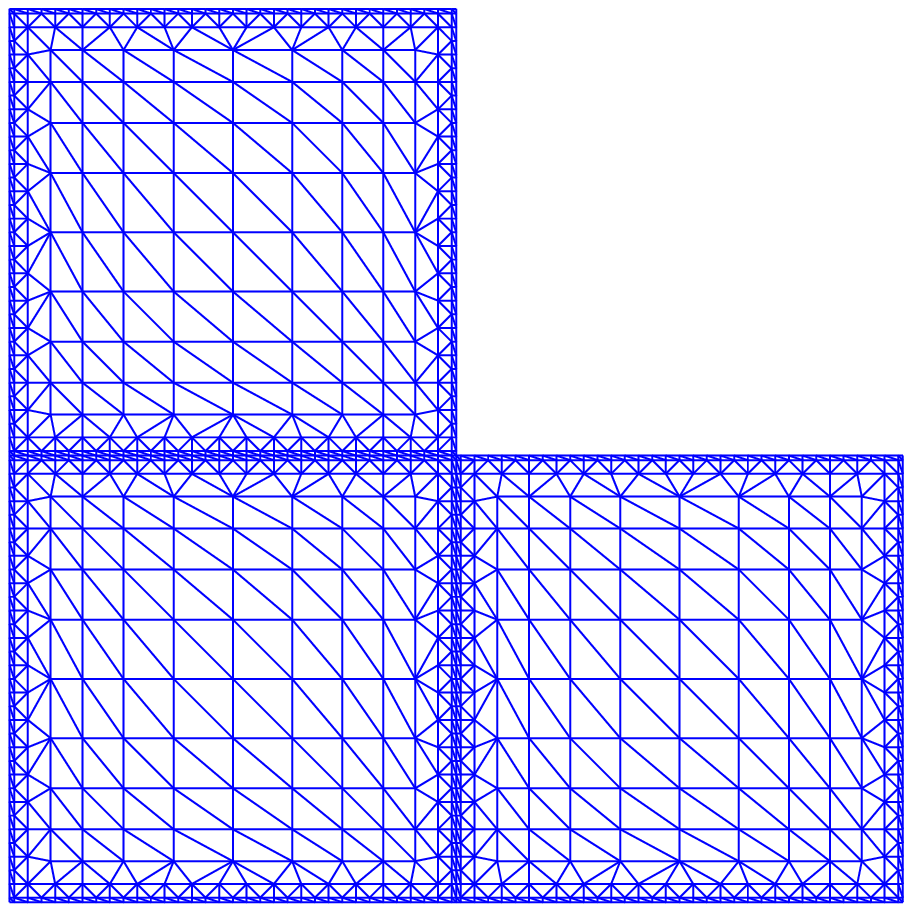}
\label{fig:agLshyes}}
\caption{Meshes used for L-shaped domain, Example~\ref{example1a}.}\label{f:graded_meshes}
\end{figure}

% \begin{table}[h!]\centering \small
% \caption{Condition numbers and GMRES/CG iterations
% on quasi-uniform meshes for L-shape \cu{(Figure~\ref{fig:unLsh})}, Example~\ref{example1a}.}
% \label{tab:UN} \centering
% \resizebox{\columnwidth}{!}{
% \begin{tabular}{|c|c c|c c|c c|c c|c c|c c|}\hline
%  & \multicolumn{4}{c|}{$s=1/4$} & \multicolumn{4}{c|}{$s=1/2$} &\multicolumn{4}{c|}{$s=3/4$} \\ \cline{2-13}
% \rule{0pt}{11pt} N &
% \multicolumn{2}{c|}{$\mathbf{A}$} & \multicolumn{2}{c|}{$\mathbf{P}\mathbf{A}$}
% &\multicolumn{2}{c|}{$\mathbf{A}$} & \multicolumn{2}{c|}{$\mathbf{P}\mathbf{A}$}
% &\multicolumn{2}{c|}{$\mathbf{A}$} & \multicolumn{2}{c|}{$\mathbf{P}\mathbf{A}$}
% \\ \cline{2-13}
% \rule{0pt}{11pt}& $\kappa$ & It. & $\kappa$ & It. & $\kappa$ & It. & $\kappa$ & It.& $\kappa$ & It. & $\kappa$ & It. \\ \hline
% 248 & 2.4 & 15/15 & 1.2 & 8/8 & 4.0 & 16/16 & 1.5 & 9/9 & 8.9 & 23/23 & 2.4 & 12/12 \\
% 992 & 2.9 & 16/16 & 1.3 & 8/8 &8.2 & 24/24 & 1.6 & 9/9 & 26.2 & 40/40 & 2.7 & 12/13 \\
% 3968 & 4.3 & 19/19 & 1.3 & 8/8 & 17.0 & 36/36 & 1.7 & 10/10 & 77.4 & 69/70 & 2.9 & 13/13 \\
% 15872 & 6.7 & 24/24 & 1.3 & 8/8 & 35.0 & 52/52 & 1.7 & 10/10 & 226.6 & 118/118 & 3.1 & 13/14 \\ \hline
%  \end{tabular}
%  }
% \end{table}
\vspace{-0.75cm}
\begin{table}[h!]\centering \footnotesize
\caption{Condition numbers and CG iterations on quasi-uniform meshes for L-shape 
(Figure~\ref{fig:unLsh}), Example~\ref{example1a}.\\[-0.6cm]} 
\label{tab:UN} \centering 
\begin{tabular}{|c|c c|c c|c c|c c|c c|c c|}\hline
 & \multicolumn{4}{c|}{$s=1/4$} & \multicolumn{4}{c|}{$s=1/2$} &\multicolumn{4}{c|}{$s=3/4$} \\ \cline{2-13}
\rule{0pt}{11pt} N & 
\multicolumn{2}{c|}{$\mathbf{A}$} & \multicolumn{2}{c|}{$\mathbf{P}\mathbf{A}$}
&\multicolumn{2}{c|}{$\mathbf{A}$} & \multicolumn{2}{c|}{$\mathbf{P}\mathbf{A}$}
&\multicolumn{2}{c|}{$\mathbf{A}$} & \multicolumn{2}{c|}{$\mathbf{P}\mathbf{A}$}
\\ \cline{2-13}
\rule{0pt}{11pt}& $\kappa$ & It. & $\kappa$ & It. & $\kappa$ & It. & $\kappa$ & It.& $\kappa$ & It. & $\kappa$ & It. \\ \hline
  248 & 2.35 & 15 & 1.24 & 8 &  4.00 & 16 & 1.48 &  9 &   8.90 &  23 & 2.35 & 12 \\
  992 & 2.86 & 16 & 1.27 & 8 &  8.22 & 24 & 1.58 &  9 &  26.22 &  40 & 2.68 & 13 \\
 3968 & 4.25 & 19 & 1.30 & 8 & 17.02 & 36 & 1.65 & 10 &  77.35 &  70 & 2.92 & 13 \\
15872 & 6.73 & 24 & 1.32 & 8 & 35.00 & 52 & 1.69 & 10 & 226.56 & 118 & 3.11 & 14 \\ \hline
 \end{tabular}
\end{table}

% \begin{table}[h!]\centering \small
% \caption{Condition numbers and GMRES/CG iterations
% on geometrically graded meshes for L-shape \cu{(Figure~\ref{fig:ggLsh})}, Example~\ref{example1a}.}
% \label{tab:GG} \centering
% \resizebox{\columnwidth}{!}{
% \begin{tabular}{|c|c c|c c|c c|c c|c c|c c|}\hline
%  & \multicolumn{4}{c|}{$s=1/4$} & \multicolumn{4}{c|}{$s=1/2$} &\multicolumn{4}{c|}{$s=3/4$} \\ \cline{2-13}
% \rule{0pt}{11pt} N &
% \multicolumn{2}{c|}{$\mathbf{A}$} & \multicolumn{2}{c|}{$\mathbf{P}\mathbf{A}$}
% &\multicolumn{2}{c|}{$\mathbf{A}$} & \multicolumn{2}{c|}{$\mathbf{P}\mathbf{A}$}
% &\multicolumn{2}{c|}{$\mathbf{A}$} & \multicolumn{2}{c|}{$\mathbf{P}\mathbf{A}$}
% \\ \cline{2-13}
% \rule{0pt}{11pt}& $\kappa$ & It. & $\kappa$ & It. & $\kappa$ & It. & $\kappa$ & It.& $\kappa$ & It. & $\kappa$ & It. \\ \hline
% 288 & 4.3 & 20/20 & 1.2 & 8/8 & 7.1 & 21/21 & 1.5 & 9/9 & 14.1 & 26/26 & 2.4 & 11/13 \\
% 720 & 12.5 & 34/34 & 1.3 & 8/8 & 18.7 & 34/34 & 1.6 & 10/10 & 35.0 & 38/38 & 2.5 & 12/14 \\
% 1632 & 36.4 & 53/53 & 1.3 & 9/9 & 47.0 & 50/50 & 1.7 & 11/11 & 82.3 & 57/57 & 2.6 & 14/15\\
% 3504 & 105.3 & 76/76 & 1.4 & 9/9 & 114.5 & 76/76 & 1.8 & 11/11 & 185.3 & 82/83 & 2.7 & 14/15 \\
% 7296 & 302.2 & 111/111 & 1.4 & 10/10 & 271.2 & 109/109 & 1.8 & 12/12 & 403.9 & 121/122 & 2.8 & 15/15 \\
% 14928 & 862.9 & 162/162 & 1.4 & 10/10 & 628.3 & 155/155 & 1.8 & 11/11 & 859.5 & 172/172 & 2.8 &  15/15\\ \hline
%  \end{tabular}
%  }
% \end{table}
\begin{table}[h!]\centering \footnotesize
\caption{Condition numbers and CG iterations
on $2$--graded (geometrically) meshes for L-shape (Figure~\ref{fig:ggLsh}), 
Example~\ref{example1a}.\\} 
\label{tab:GG} \centering 
\begin{tabular}{|c|c c|c c|c c|c c|c c|c c|}\hline
 & \multicolumn{4}{c|}{$s=1/4$} & \multicolumn{4}{c|}{$s=1/2$} &\multicolumn{4}{c|}{$s=3/4$} \\ \cline{2-13}
\rule{0pt}{11pt} N & 
\multicolumn{2}{c|}{$\mathbf{A}$} & \multicolumn{2}{c|}{$\mathbf{P}\mathbf{A}$}
&\multicolumn{2}{c|}{$\mathbf{A}$} & \multicolumn{2}{c|}{$\mathbf{P}\mathbf{A}$}
&\multicolumn{2}{c|}{$\mathbf{A}$} & \multicolumn{2}{c|}{$\mathbf{P}\mathbf{A}$}
\\ \cline{2-13}
\rule{0pt}{11pt}& $\kappa$ & It. & $\kappa$ & It. & $\kappa$ & It. & $\kappa$ & It.& $\kappa$ 
& It. & $\kappa$ & It. \\ \hline
  288 &   4.28 &  20 & 1.24 &  8 &   7.08 &  21 & 1.51 &  9 &  14.06 &  26 & 2.36 & 13 \\
  720 &  12.53 &  34 & 1.29 &  8 &  18.65 &  34 & 1.60 & 10 &  35.02 &  38 & 2.46 & 14 \\
 1632 &  36.44 &  53 & 1.33 &  9 &  47.03 &  50 & 1.68 & 11 &  82.34 &  57 & 2.56 & 15 \\
 3504 & 105.28 &  76 & 1.37 &  9 & 114.49 &  76 & 1.75 & 11 & 185.29 &  83 & 2.67 & 15 \\
 7296 & 302.23 & 111 & 1.39 & 10 & 271.20 & 109 & 1.79 & 12 & 403.92 & 122 & 2.75 & 15 \\ 
14928 & 862.91 & 162 & 1.39 & 10 & 628.32 & 155 & 1.76 & 11 & 859.51 & 172 & 2.84 & 15\\ \hline
 \end{tabular}
\end{table}

% \begin{table}[h!]\centering \small
% \caption{Condition numbers and GMRES/CG iterations
% on $2$--graded (algebraically) meshes for L-shape \cu{(Figure~\ref{fig:agLshno})}, Example~\ref{example1a}.}
% \label{tab:AG} \centering
% \resizebox{\columnwidth}{!}{
% \begin{tabular}{|c|c c|c c|c c|c c|c c|c c|}\hline
%  & \multicolumn{4}{c|}{$s=1/4$} & \multicolumn{4}{c|}{$s=1/2$} &\multicolumn{4}{c|}{$s=3/4$} \\ \cline{2-13}
% \rule{0pt}{11pt} N &
% \multicolumn{2}{c|}{$\mathbf{A}$} & \multicolumn{2}{c|}{$\mathbf{P}\mathbf{A}$}
% &\multicolumn{2}{c|}{$\mathbf{A}$} & \multicolumn{2}{c|}{$\mathbf{P}\mathbf{A}$}
% &\multicolumn{2}{c|}{$\mathbf{A}$} & \multicolumn{2}{c|}{$\mathbf{P}\mathbf{A}$}
% \\ \cline{2-13}
% \rule{0pt}{11pt}& $\kappa$ & It. & $\kappa$ & It. & $\kappa$ & It. & $\kappa$ & It.& $\kappa$ & It. & $\kappa$ & It. \\ \hline
% 384 & 12.5 & 34/34 & 1.4 & 9/9 & 8.9 & 28/28 & 1.8 & 12/12 & 28.7 & 37/37 & 4.3 & 20/22 \\
% 1536 & 41.9 & 60/61 & 1.6 & 10/10 & 21.5 & 46/46 & 2.8 & 15/16 & 146.7 & 77/82 & 26.5 & 38/46 \\
% 4704 & 105.8 & 93/94 & 1.9 & 12/12 & 47.3 & 65/67 & 3.8 & 16/18 & 559.5 & 138/159 & 91.3 & 136/161 \\
% 16224 & 283.5 & 149/153 & 2.7 & 13/13 & 124.6 & 102/104 & 5.2 & 18/19 & 2726.6 & 357/486 & 695.9 & 280/443 \\ \hline
%  \end{tabular}
%  }
% \end{table}
\begin{table}[h!]\centering \footnotesize
\caption{Condition numbers and CG iterations on $2$--graded (algebraically) meshes 
for L-shape (Figure~\ref{fig:agLshno}), Example~\ref{example1a}.\\} 
\label{tab:AG} \centering 
\begin{tabular}{|@{ }c@{ }|c c@{ }|c c@{ }|c c@{ }|c c@{ }|c c@{ }|c c@{ }|}\hline
 & \multicolumn{4}{c|}{$s=1/4$} & \multicolumn{4}{c|}{$s=1/2$} &\multicolumn{4}{c|}{$s=3/4$} \\ \cline{2-13}
\rule{0pt}{11pt} N & 
\multicolumn{2}{c@{ }|}{$\mathbf{A}$} & \multicolumn{2}{c@{ }|}{$\mathbf{P}\mathbf{A}$}
&\multicolumn{2}{c@{ }|}{$\mathbf{A}$} & \multicolumn{2}{c@{ }|}{$\mathbf{P}\mathbf{A}$}
&\multicolumn{2}{c@{ }|}{$\mathbf{A}$} & \multicolumn{2}{c@{ }|}{$\mathbf{P}\mathbf{A}$}
\\ \cline{2-13}
\rule{0pt}{11pt}& $\kappa$ & It. & $\kappa$ & It. & $\kappa$ & It. & $\kappa$ & It.& $\kappa$ & It. & $\kappa$ & It. \\ \hline
  384 &  12.49 &  34 & 1.36 &  9 &   8.91 &  28 & 1.78 & 12 &   28.72 &  37 &   4.30 &  22 \\
 1536 &  41.86 &  61 & 1.64 & 10 &  21.51 &  46 & 2.81 & 16 &  146.66 &  82 &  26.52 &  46 \\
 4704 & 105.76 &  94 & 1.94 & 12 &  47.29 &  67 & 3.76 & 18 &  559.48 & 159 &  91.34 & 161 \\
16224 & 283.50 & 153 & 2.65 & 13 & 124.63 & 104 & 5.17 & 19 & 2726.63 & 486 & 695.92 & 443 \\ \hline
 \end{tabular}

\end{table}

% \begin{table}[h!]\centering \small
% \caption{Condition numbers and GMRES/CG iterations
% on algebraically $2$--graded (shape regular) meshes for L-shape \cu{(Figure~\ref{fig:agLshyes})}, Example~\ref{example1a}.}
% \label{tab:SRAG} \centering
% \resizebox{\columnwidth}{!}{
% \begin{tabular}{|c|c c|c c|c c|c c|c c|c c|}\hline
%  & \multicolumn{4}{c|}{$s=1/4$} & \multicolumn{4}{c|}{$s=1/2$} &\multicolumn{4}{c|}{$s=3/4$} \\ \cline{2-13}
% \rule{0pt}{11pt} N &
% \multicolumn{2}{c|}{$\mathbf{A}$} & \multicolumn{2}{c|}{$\mathbf{P}\mathbf{A}$}
% &\multicolumn{2}{c|}{$\mathbf{A}$} & \multicolumn{2}{c|}{$\mathbf{P}\mathbf{A}$}
% &\multicolumn{2}{c|}{$\mathbf{A}$} & \multicolumn{2}{c|}{$\mathbf{P}\mathbf{A}$}
% \\ \cline{2-13}
% \rule{0pt}{11pt}& $\kappa$ & It. & $\kappa$ & It. & $\kappa$ & It. & $\kappa$ & It.& $\kappa$ & It. & $\kappa$ & It. \\ \hline
% 528 & 13.1 & 36/36 & 1.3 & 8/8 & 13.0 & 30/31 & 1.7 & 11/11 & 25.1 & 32/33 & 2.6 & 15/15 \\
% 912 & 19.2 & 44/44 & 1.3 & 8/8 & 19.8 & 37/37 & 1.7 & 11/11 & 42.3 & 42/43 & 2.9 & 16/16 \\
% 2736 & 43.9 & 66/66 & 1.3 & 9/9 & 44.5 & 57/58 & 1.8 & 12/12 & 111.2 & 75/76 & 4.0 & 19/19\\
% 4920 & 63.8 & 79/79 & 1.4 & 9/9 & 67.1 & 72/73 & 1.8 & 12/12 & 183.7 & 97/99 & 4.2 & 19/19 \\
% 9072 & 97.2  & 96/96 & 1.4 & 9/9 & 102.5 & 89/91 & 1.8 & 12/12 & 306.1 & 127/129 & 4.4 & 20/20 \\
% 14784 & 140.1 & 114/114 & 1.4 & 9/9 & 142.7 & 107/108 & 1.7 & 11/11 & 458.3 & 159/161 & 4.5 &  20/20\\ \hline
%  \end{tabular}
%  }
% \end{table}
\begin{table}[h!]\centering \footnotesize
\caption{Condition numbers and CG iterations on $2$--graded (algebraically shape regular) 
meshes for L-shape (Figure~\ref{fig:agLshyes}), Example~\ref{example1a}.\\} 
\label{tab:SRAG} \centering 
\begin{tabular}{|c|c c|c c|c c|c c|c c|c c|}\hline
 & \multicolumn{4}{c|}{$s=1/4$} & \multicolumn{4}{c|}{$s=1/2$} &\multicolumn{4}{c|}{$s=3/4$} \\ \cline{2-13}
\rule{0pt}{11pt} N & 
\multicolumn{2}{c|}{$\mathbf{A}$} & \multicolumn{2}{c|}{$\mathbf{P}\mathbf{A}$}
&\multicolumn{2}{c|}{$\mathbf{A}$} & \multicolumn{2}{c|}{$\mathbf{P}\mathbf{A}$}
&\multicolumn{2}{c|}{$\mathbf{A}$} & \multicolumn{2}{c|}{$\mathbf{P}\mathbf{A}$}
\\ \cline{2-13}
\rule{0pt}{11pt}& $\kappa$ & It. & $\kappa$ & It. & $\kappa$ & It. & $\kappa$ & It.& $\kappa$ & It. & $\kappa$ & It. \\ \hline
  528 &  13.12 &  36 & 1.28 & 8 &  12.99 &  31 & 1.67 & 11 &  25.12 &  33 & 2.64 & 15 \\
  912 &  19.15 &  44 & 1.30 & 8 &  19.78 &  37 & 1.71 & 11 &  42.33 &  43 & 2.87 & 16 \\
 2736 &  43.93 &  66 & 1.34 & 9 &  44.51 &  58 & 1.78 & 12 & 111.22 &  76 & 4.01 & 19 \\
 4920 &  63.79 &  79 & 1.36 & 9 &  67.06 &  73 & 1.79 & 12 & 183.65 &  99 & 4.22 & 19 \\
 9072 &  97.20 &  96 & 1.37 & 9 & 102.45 &  91 & 1.76 & 12 & 306.14 & 129 & 4.39 & 20 \\ 
14784 & 140.13 & 114 & 1.38 & 9 & 142.72 & 108 & 1.73 & 11 & 458.32 & 161 & 4.49 & 20 \\ \hline
 \end{tabular}
\end{table}

\begin{figure}[!ht]
\centering
%\subfloat[][]{\includegraphics[width = 0.4\textwidth]{UNLshape.eps}}\quad
\subfloat[][]{\includegraphics[width = 0.43\textwidth]{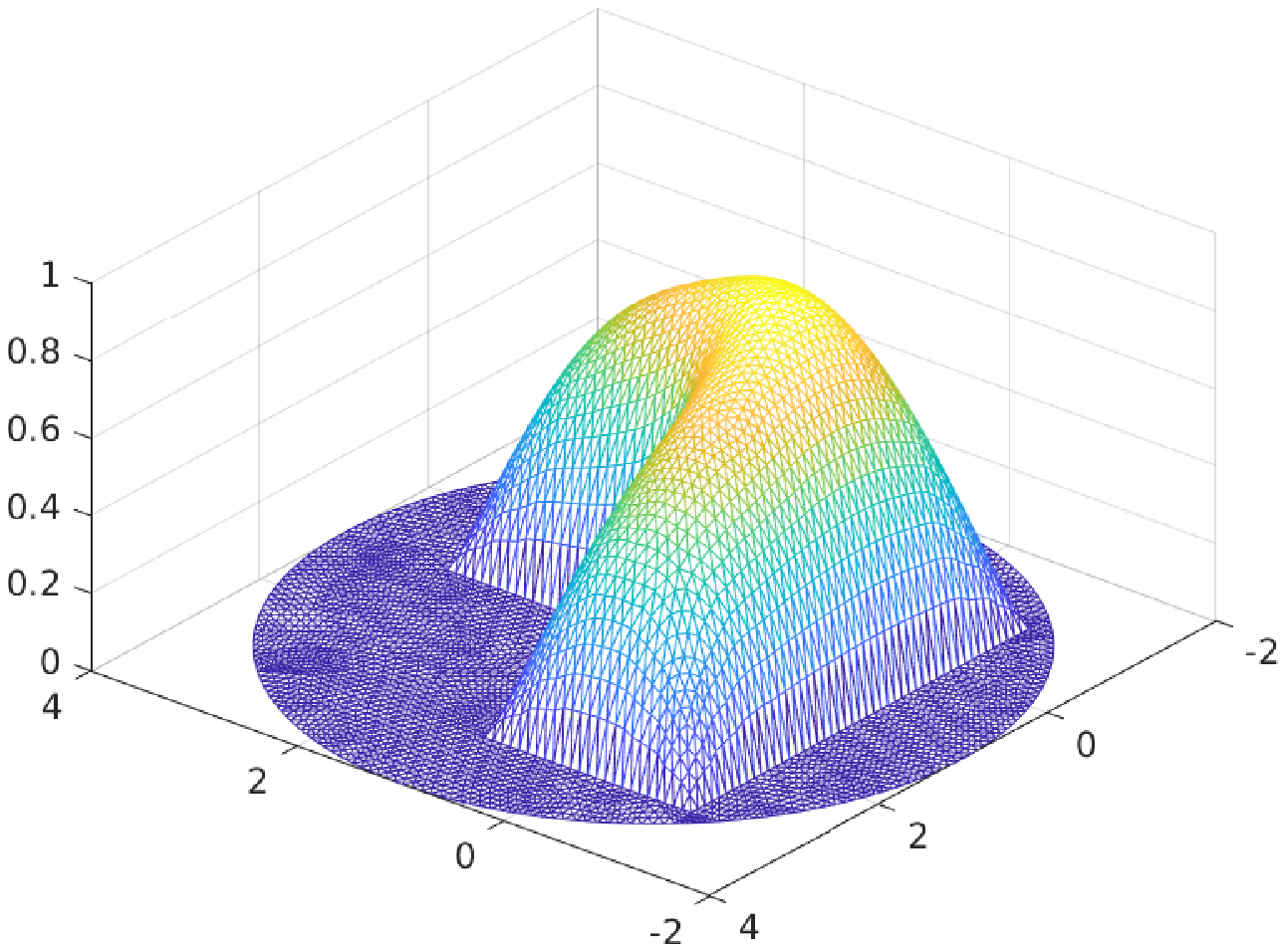}} \quad
\subfloat[][]{\includegraphics[width = 0.43\textwidth]{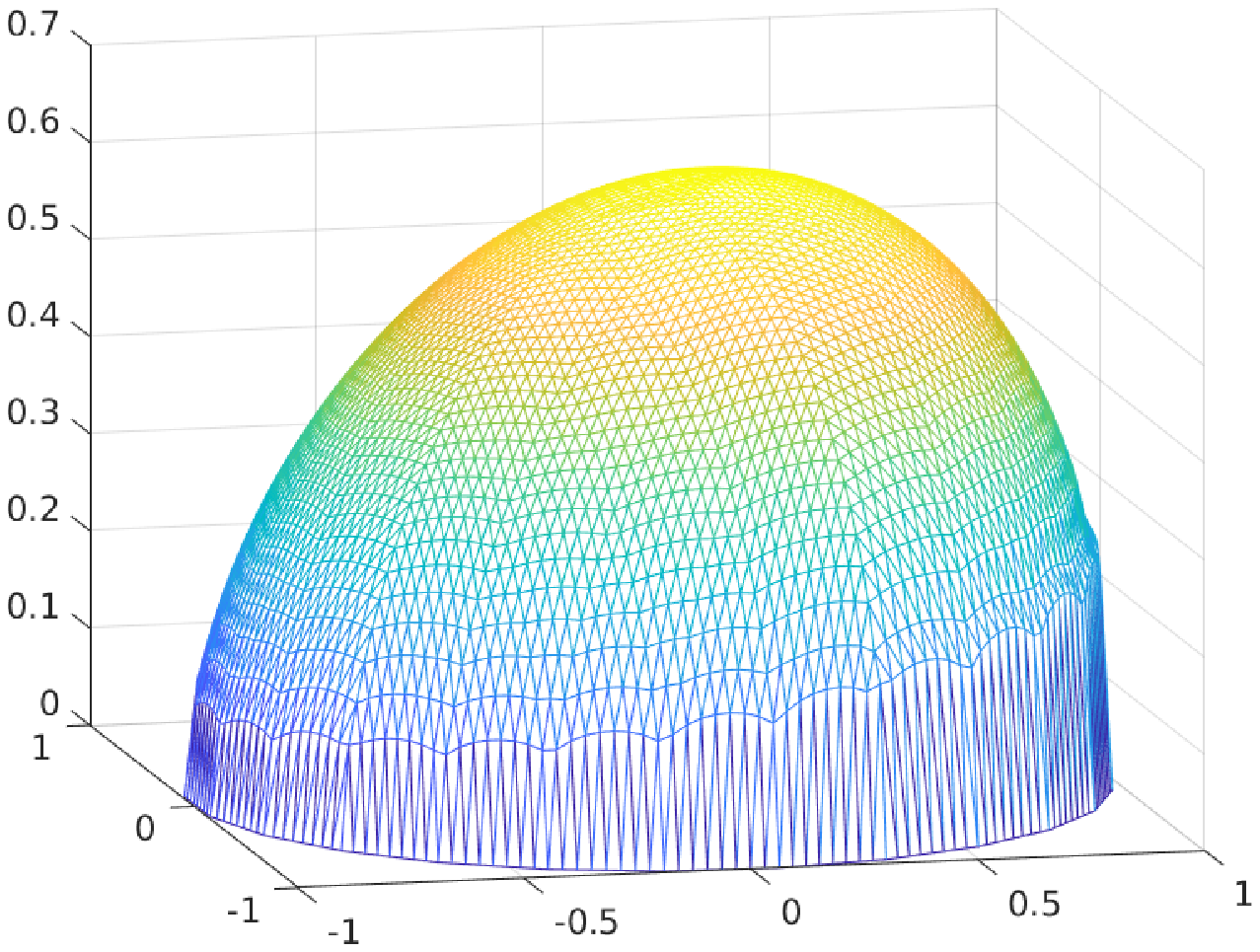}}
\caption{Numerical solutions for Example~\ref{example1a} $(a)$ and 
Example~\ref{example3} $(b)$ with $s = \frac{3}{4}$.}\label{f:Lshape}
\end{figure}

%------------------------------------------------------------------------------
\begin{example}\label{example_rectangles}
We consider the discretization of the Dirichlet problem \eqref{e:FLweakGamma} 
with $A=(-\Delta)^{s}$ and $f = 1$ in rectangular domains $[-a,a]\times[-1,1] 
\subset \mathbb{R}^2$ with varying aspect ratio \cu{$a:1$}. We examine fractional 
exponents $s = \frac{1}{4}, \frac{1}{2}, \frac{3}{4}$ on quasi-uniform meshes, see Figure~\ref{f:mesh_rect} for illustration.
The preconditioner is computed using the radial projection $\chi$ from the 
rectangular domain to $\calB_1(0)$.
\end{example}

\begin{figure}[!h]
\centering
%\subfloat[][$1:1$]{\includegraphics[width = 0.3\textwidth]{MeshRect_1_1.eps}
%\label{fig:mesha_rect}}
\subfloat[][$2:1$]{\includegraphics[width = 0.3\textwidth]{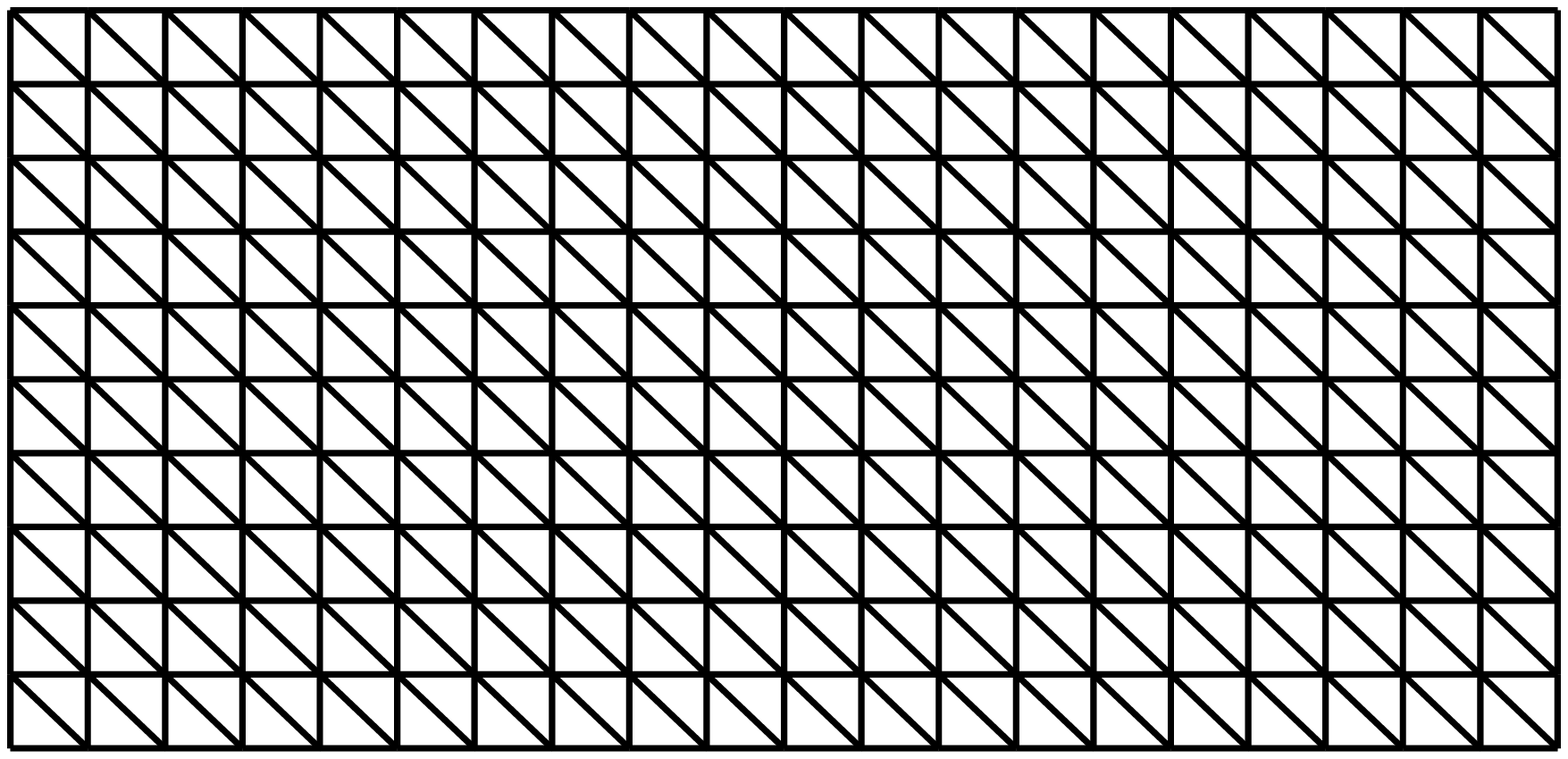}
\label{fig:meshb_rect}}
\subfloat[][$4:1$]{\includegraphics[width = 0.57\textwidth]{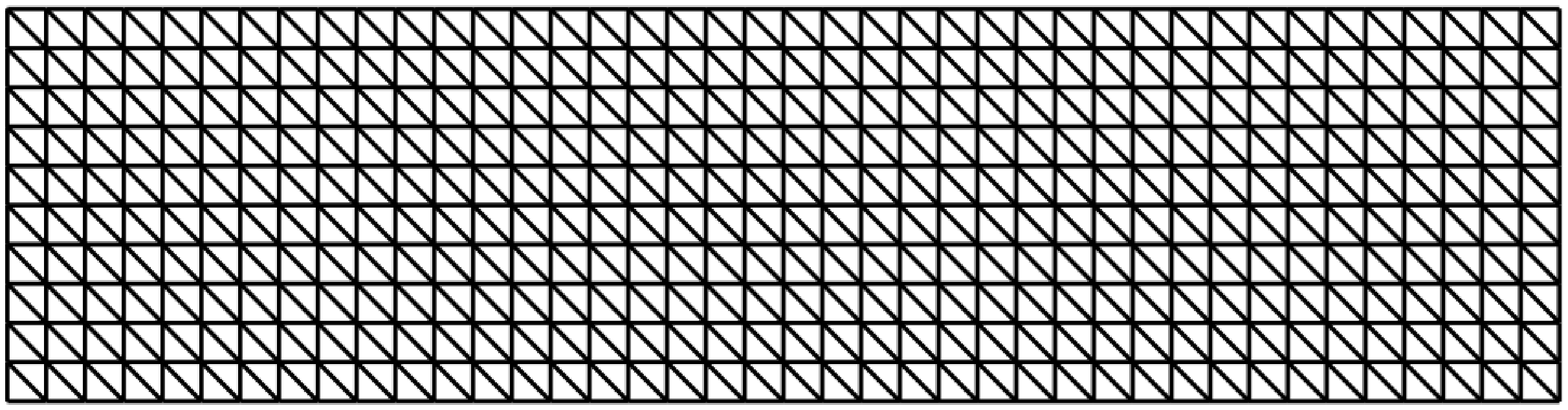}
\label{fig:meshc_rect}}
\caption{Meshes for rectangles with varied aspect ratio.}\label{f:mesh_rect}
\end{figure}

Tables~\ref{tab:s025}--\ref{tab:s075} display the results of the Galerkin matrix 
$\mathbf{A}$ and its preconditioned form $\mathbf{P}\mathbf{A}$  on a sequence 
of corresponding meshes. \cu{In most cases, the preconditioner performs 
qualitatively the same as we already observed for Example \ref{example1a}:
the condition numbers and the number of solver iterations for $\mathbf{P}
\mathbf{A}$ tend to remain constant with respect to $h$. 

The novelty here is how results change when the aspect ratio $a:1$ increases. 
Indeed, as expected from the theory, condition numbers, and consequently CG 
iteration counts, grow when the ``distortion'' from the unit disk 
is more significant,  i.e. for bigger aspect ratios. Moreover, how the 
transformation impacts condition numbers depends on the related Sobolev norms, 
reason why they are actually $s$-dependent. This is clearly reflected in our 
experiments where the difference between the results when the aspect ratio is 
$1:1$ and $16:1$ is relatively small for $s=0.25$, but notorious for $s=0.75$.
Nevertheless, as the original system for these distorted geometries are more 
ill-conditioned, $\mathbf{P}\mathbf{A}$ still reduces the number of 
iteration counts in a meaningful manner. 

Although it is hard to draw general conclusions, with these results we expect to 
convey two messages: On the one hand, we highlight the robustness of this 
preconditioning approach. On the other hand, we warn the reader that there may 
be geometries for which, \sj{despite of the quasi-uniform mesh on the original geometry}, ``the mapping trick'' from \eqref{e:b_mapped} can lead to large, yet bounded, condition numbers	
	and thereby may no longer be a practical strategy to construct a preconditioner.}

\begin{table}[h!]\centering \footnotesize
\caption{Condition numbers and CG iterations on quasi-uniform mesh with $s = 0.25$ 
and varied aspect ratio \cu{$a:1$, Example~\ref{example_rectangles}}.} 
\label{tab:s025} \centering  \setlength{\tabcolsep}{2.3pt}
\begin{tabular}{|c|c c|c c|c c|c c|c c |c c|c c|c c|c c|c c| }\hline
 &\multicolumn{4}{c|}{$1:1$} & \multicolumn{4}{c|}{$2:1$} & \multicolumn{4}{c|}{$4:1$} & \multicolumn{4}{c|}{$8:1$} & \multicolumn{4}{c|}{$16:1$} \\ \cline{2-21}
\rule{0pt}{11pt} $h$ & 
\multicolumn{2}{c|}{$\mathbf{A}$} & \multicolumn{2}{c|}{$\mathbf{P}\mathbf{A}$}
&\multicolumn{2}{c|}{$\mathbf{A}$} & \multicolumn{2}{c|}{$\mathbf{P}\mathbf{A}$}
&\multicolumn{2}{c|}{$\mathbf{A}$} & \multicolumn{2}{c|}{$\mathbf{P}\mathbf{A}$}
&\multicolumn{2}{c|}{$\mathbf{A}$} & \multicolumn{2}{c|}{$\mathbf{P}\mathbf{A}$}
&\multicolumn{2}{c|}{$\mathbf{A}$} & \multicolumn{2}{c|}{$\mathbf{P}\mathbf{A}$}
\\ \cline{2-21}
\rule{0pt}{11pt}& $\kappa$ & It. & $\kappa$ & It. & $\kappa$ & It. & $\kappa$ & It. & $\kappa$ & It. & $\kappa$ & It. & $\kappa$ & It. & $\kappa$ & It. & $\kappa$ & It. & $\kappa$ & It. \\ \hline
  $\frac{1}{5\sqrt{2}}$ &  1.90 & 12 & 1.98 & 11 & 2.06 & 12 & 2.05 & 11 & 2.22 & 12 & 2.16 & 12 & 2.30 & 13 & 2.29 & 12 & 2.34 & 14 & 2.43 & 13\\
  $\frac{1}{10\sqrt{2}}$ & 2.54 & 12 & 2.05 & 11 & 2.91 & 13 & 2.12 & 12 & 3.14 & 14 & 2.25 & 12 & 3.26 & 16 & 2.41 & 13 & 3.32 & 18 & 2.56 & 14\\
  $\frac{1}{15\sqrt{2}}$ & 3.11 & 13 & 2.09 & 12 & 3.57 & 15 & 2.15 & 12 & 3.85 & 16 & 2.28 & 12 & 4.00 & 19 & 2.45 & 13 & 4.15 & 20 & 2.62 & 14\\
  $\frac{1}{20\sqrt{2}}$ & 3.59 & 14 & 2.10 & 12 & 4.12 & 16 & 2.17 & 12 & 4.45 & 18 & 2.31 & 12 & 4.62 & 20 & 2.49 & 14 & 4.78 & 21 & 2.67 & 14 \\
 \hline
 \end{tabular}
 \\[-0.2cm]
\end{table}

\begin{table}[h!]\centering \footnotesize
\caption{Condition numbers and CG iterations on quasi-uniform mesh with 
$s = 0.50$ and varied aspect ratio \cu{$a:1$, Example~\ref{example_rectangles}}.} 
\label{tab:s050} \centering \setlength{\tabcolsep}{1.8pt}
\begin{tabular}{|c|c c|c c|c c|c c|c c |c c|c c|c c|c c|c c| }\hline
 &\multicolumn{4}{c|}{$1:1$} & \multicolumn{4}{c|}{$2:1$} & \multicolumn{4}{c|}{$4:1$} & \multicolumn{4}{c|}{$8:1$} & \multicolumn{4}{c|}{$16:1$} \\ \cline{2-21}
\rule{0pt}{11pt} $h$ & 
\multicolumn{2}{c|}{$\mathbf{A}$} & \multicolumn{2}{c|}{$\mathbf{P}\mathbf{A}$}
&\multicolumn{2}{c|}{$\mathbf{A}$} & \multicolumn{2}{c|}{$\mathbf{P}\mathbf{A}$}
&\multicolumn{2}{c|}{$\mathbf{A}$} & \multicolumn{2}{c|}{$\mathbf{P}\mathbf{A}$}
&\multicolumn{2}{c|}{$\mathbf{A}$} & \multicolumn{2}{c|}{$\mathbf{P}\mathbf{A}$}
&\multicolumn{2}{c|}{$\mathbf{A}$} & \multicolumn{2}{c|}{$\mathbf{P}\mathbf{A}$}
\\ \cline{2-21}
\rule{0pt}{11pt}& $\kappa$ & It. & $\kappa$ & It. & $\kappa$ & It. & $\kappa$ & It. & $\kappa$ & It. & $\kappa$ & It. & $\kappa$ & It. & $\kappa$ & It. & $\kappa$ & It. & $\kappa$ & It. \\ \hline
 $\frac{1}{5\sqrt{2}}$ &  4.35 & 14 & 1.81 & 11 & 5.64 & 17 & 1.92 & 11 & 6.37 & 20 & 2.25 & 12 & 6.68 & 23 & 2.74 & 13 & 6.80 & 27 & 3.41 & 15\\
 $\frac{1}{10\sqrt{2}}$ & 8.70 & 20 & 1.83 & 11 & 11.30 & 25 & 1.96 & 12 & 12.79 & 30 & 2.32 & 13 & 13.42 & 35 & 2.91 & 15 & 13.72 & 40 & 3.64 & 17\\
 $\frac{1}{15\sqrt{2}}$ & 13.07 & 25 & 1.84 & 11 & 16.98 & 31 & 1.97 & 11 & 19.22 & 37 & 2.36 & 13 & 20.19 & 45 & 2.99 & 15 & 20.63 & 49 & 3.77 & 17\\
 $\frac{1}{20\sqrt{2}}$ & 17.44 & 29 & 1.85 & 11 & 22.66 & 36 & 1.99 & 12 & 25.66 & 42 & 2.39 & 14 &  26.97 & 52 & 3.05 & 15 & 27.48 & 54 & 3.86 & 17 \\
 \hline
 \end{tabular}
 \\[-0.2cm]
\end{table}

\begin{table}[h!]\centering \footnotesize
\caption{Condition numbers and CG iterations on quasi-uniform mesh with 
$s = 0.75$ and varied aspect ratio \cu{$a:1$, Example~\ref{example_rectangles}}.} 
\label{tab:s075} \centering \setlength{\tabcolsep}{1.0pt}
\begin{tabular}{|c|c c|c c|c c|c c|c c |c c|c c|c c|c c|c c| }\hline
 &\multicolumn{4}{c|}{$1:1$} & \multicolumn{4}{c|}{$2:1$} & \multicolumn{4}{c|}{$4:1$} & \multicolumn{4}{c|}{$8:1$} & \multicolumn{4}{c|}{$16:1$} \\ \cline{2-21}
\rule{0pt}{11pt} $h$ & 
\multicolumn{2}{c|}{$\mathbf{A}$} & \multicolumn{2}{c|}{$\mathbf{P}\mathbf{A}$}
&\multicolumn{2}{c|}{$\mathbf{A}$} & \multicolumn{2}{c|}{$\mathbf{P}\mathbf{A}$}
&\multicolumn{2}{c|}{$\mathbf{A}$} & \multicolumn{2}{c|}{$\mathbf{P}\mathbf{A}$}
&\multicolumn{2}{c|}{$\mathbf{A}$} & \multicolumn{2}{c|}{$\mathbf{P}\mathbf{A}$}
&\multicolumn{2}{c|}{$\mathbf{A}$} & \multicolumn{2}{c|}{$\mathbf{P}\mathbf{A}$}
\\ \cline{2-21}
\rule{0pt}{11pt}& $\kappa$ & It. & $\kappa$ & It. & $\kappa$ & It. & $\kappa$ & It. & $\kappa$ & It. & $\kappa$ & It. & $\kappa$ & It. & $\kappa$ & It. & $\kappa$ & It. & $\kappa$ & It. \\ \hline
 $\frac{1}{5\sqrt{2}}$ &  13.16 & 22 & 1.70 & 11 & 19.06 & 30 & 1.99 & 12 & 22.13 & 37 & 3.62 & 14 & 23.23 & 44 & 11.56 & 17 & 23.60 & 53 & 39.03 & 25\\
  $\frac{1}{10\sqrt{2}}$ &  37.19 & 37 & 1.75 & 12 & 53.94 & 51 & 2.05 & 14 & 62.68 & 62 & 3.99 & 16 & 65.87 & 76 & 17.33 & 20 & 67.36 & 91 & 18.12 & 30\\
 $\frac{1}{15\sqrt{2}}$ & 68.36 & 50 & 1.81 & 13 & 99.16 & 68 & 2.11 & 15 & 115.27 & 83 & 4.21 & 17 & 121.21 & 103 & 21.58 & 22 & 123.93 & 124 & 20.09 & 34\\
   $\frac{1}{20\sqrt{2}}$ & 105.28 & 62 & 1.87 & 14 & 152.74 & 85 & 2.15 & 15 & 177.61 & 107 & 4.35 & 18 & 186.92 & 129 & 24.94 & 24 & 189.17 & 147 & 21.84 & 36\\
 \hline
 \end{tabular}
 \\[-0.2cm]
\end{table}

As a final example, we apply the  preconditioner to a non-symmetric model problem motivated by the fractional Patlak-Keller-Segel equation for chemotaxis \cite{egp}.

\begin{example}\label{example3}
We consider the discretization of the Dirichlet problem
\eqref{e:FLweakGamma} with $A=(-\Delta)^{s} + c\cdot \nabla$, $c= (0.3,0)^T$ and $f = 1$ on the unit disk $\calB_1
\subset \mathbb{R}^2$ with $s = \frac{1}{2}$, $s = \frac{7}{10}$ and $s = \frac{3}{4}$. Quasi-uniform and algebraically $2$-graded meshes are considered. A numerical solution on a uniform mesh with 7872 elements is depicted in Figure~\ref{f:Lshape}.
\end{example}

Tables~\ref{tab:1} and \ref{tab:2} display the condition numbers of the Galerkin matrix $\mathbf{A}$ and its preconditioned form $\mathbf{P}\mathbf{A}$ for the different fractional exponents on sequences of quasi-uniform meshes, and on algebraically graded meshes. The number of GMRES iterations is given for 
this non-symmetric problem.

As in the earlier examples, on both quasi-uniform and graded meshes the condition number and the number of solver iterations for $\mathbf{A}$ show a strong increase with $N$. For $\mathbf{P}\mathbf{A}$ they are bounded with a slight growth, with numbers very close to those in Example \ref{example1} for $s = \frac{7}{10}, \frac{3}{4}$. Note that for $s=\frac{1}{2}$ the gradient term is of the same order as $(-\Delta)^{s}$.

%Figure \ref{f:eigvals_adv} shows the spectra of the Galerkin and preconditioned Galerkin matrices for this problem.

%\begin{figure}[h!]
%\centering
%\includegraphics[width = 0.5\textwidth]{Advection_soln.eps}
%\caption{Numerical solution for $s = \frac{3}{4}$ with $c = (0.3, 0)^T$, Example~\ref{example3}.}\label{f:adv_soln}
%\end{figure}

% \begin{table}[h!]\centering \small
% \caption{Condition numbers and GMRES iterations
% on quasi-uniform mesh,  Example~\ref{example3}.}
% \label{tab:1} \centering
% \resizebox{\columnwidth}{!}{
% \begin{tabular}{|c|c c|c c|c c|c c|c c |c c|}\hline
%  &\multicolumn{4}{c|}{$s=1/2$} & \multicolumn{4}{c|}{$s=7/10$} & \multicolumn{4}{c|}{$s=3/4$} \\ \cline{2-13}
% \rule{0pt}{11pt} N &
% \multicolumn{2}{c|}{$\mathbf{A}$} & \multicolumn{2}{c|}{$\mathbf{P}\mathbf{A}$}
% &\multicolumn{2}{c|}{$\mathbf{A}$} & \multicolumn{2}{c|}{$\mathbf{P}\mathbf{A}$}
% &\multicolumn{2}{c|}{$\mathbf{A}$} & \multicolumn{2}{c|}{$\mathbf{P}\mathbf{A}$}
% \\ \cline{2-13}
% \rule{0pt}{11pt}& $\kappa$ & It. & $\kappa$ & It. & $\kappa$ & It. & $\kappa$ & It. & $\kappa$ & It. & $\kappa$ & It. \\ \hline
% 123 & 3.1 & 14 & 1.1 & 12 & 6.7 & 17 & 1.5 & 11 & 8.1 & 18 & 1.5 & 11\\
% 492 & 7.0 & 22 & 1.2 & 12 & 20.4 & 29 & 1.5 & 11 & 26.6 & 32 & 1.5 & 11\\
% 1968 & 15.1 & 35 & 1.2 & 12 & 60.9 & 48 & 1.5 & 11 & 85.9 & 55 & 1.7 & 11\\
% 7872 & 31.9 & 54 & 1.2 & 13 & 172.7 & 83 & 1.7 & 11 & 264.0 & 95 & 2.2 & 12\\ 
% 31488 & 65.9 & 81 & 1.4 & 14 & 469.4 & 142 & 2.0 & 11 & 772.5 & 169 & 2.8 & 13 \\
% \hline
%  \end{tabular}
%  }
% \end{table}
\begin{table}[h!]\centering \footnotesize
\caption{Condition numbers and GMRES iterations 
on quasi-uniform mesh,  Example~\ref{example3}.\\} 
\label{tab:1} \centering 
\begin{tabular}{|c|c c|c c|c c|c c|c c |c c|}\hline
 &\multicolumn{4}{c|}{$s=1/2$} & \multicolumn{4}{c|}{$s=7/10$} & \multicolumn{4}{c|}{$s=3/4$} \\ \cline{2-13}
\rule{0pt}{11pt} N & 
\multicolumn{2}{c|}{$\mathbf{A}$} & \multicolumn{2}{c|}{$\mathbf{P}\mathbf{A}$}
&\multicolumn{2}{c|}{$\mathbf{A}$} & \multicolumn{2}{c|}{$\mathbf{P}\mathbf{A}$}
&\multicolumn{2}{c|}{$\mathbf{A}$} & \multicolumn{2}{c|}{$\mathbf{P}\mathbf{A}$}
\\ \cline{2-13}
\rule{0pt}{11pt}& $\kappa$ & It. & $\kappa$ & It. & $\kappa$ & It. & $\kappa$ & It. & $\kappa$ & It. & $\kappa$ & It. \\ \hline
 123 &  3.11 & 14 & 1.08 & 12 &   6.69 & 17 & 1.48 & 11 &   8.11 & 18 & 1.49 & 11\\
 492 &  7.02 & 22 & 1.15 & 12 &  20.39 & 29 & 1.50 & 11 &  26.59 & 32 & 1.53 & 11\\
1968 & 15.08 & 35 & 1.19 & 12 &  60.87 & 48 & 1.54 & 11 &  85.93 & 55 & 1.71 & 11\\
7872 & 31.85 & 54 & 1.22 & 13 & 172.73 & 83 & 1.77 & 11 & 264.01 & 95 & 2.15 & 12\\ \hline
 \end{tabular}
 \\[-0.2cm]
\end{table}

\begin{table}[h!]\centering \footnotesize
\caption{Condition numbers and GMRES iterations 
on graded mesh,  Example~\ref{example3}.} 
\label{tab:2} \centering 
\begin{tabular}{|c|c c|c c|c c|c c|c c |c c|}\hline
 & \multicolumn{4}{c|}{$s=1/2$} & \multicolumn{4}{c|}{$s=7/10$} & \multicolumn{4}{c|}{$s=3/4$} \\ \cline{2-13}
\rule{0pt}{11pt} N & 
\multicolumn{2}{c|}{$\mathbf{A}$} & \multicolumn{2}{c|}{$\mathbf{P}\mathbf{A}$}
&\multicolumn{2}{c|}{$\mathbf{A}$} & \multicolumn{2}{c|}{$\mathbf{P}\mathbf{A}$}
&\multicolumn{2}{c|}{$\mathbf{A}$} & \multicolumn{2}{c|}{$\mathbf{P}\mathbf{A}$}
\\ \cline{2-13}
\rule{0pt}{11pt}& $\kappa$ & It. & $\kappa$ & It. & $\kappa$ & It. & $\kappa$ & It. & $\kappa$ & It. & $\kappa$ & It. \\ \hline
  123 &   3.31 & 19 & 1.17 & 12 &   4.42 &  17 & 1.70 & 12 &   5.07 &  18 & 1.93 & 12\\
 1068 &  14.24 & 31 & 1.26 & 12 &  27.78 &  36 & 2.39 & 14 &  33.07 &  38 & 2.91 & 15\\
 4645 &  44.15 & 54 & 1.34 & 12 & 104.49 &  69 & 2.84 & 15 & 131.43 &  79 & 3.64 & 16\\
13680 & 101.41 & 73 & 1.37 & 12 & 277.05 & 103 & 2.96 & 15 & 358.78 & 117 & 3.87 & 16\\ \hline
 \end{tabular}
\end{table}

% \begin{table}[h!]\centering \small
% \caption{Condition numbers and GMRES iterations
% on $2$-graded mesh,  Example~\ref{example3}.}
% \label{tab:2} \centering
% \resizebox{\columnwidth}{!}{
% \begin{tabular}{|c|c c|c c|c c|c c|c c |c c|}\hline
%  & \multicolumn{4}{c|}{$s=1/2$} & \multicolumn{4}{c|}{$s=7/10$} & \multicolumn{4}{c|}{$s=3/4$} \\ \cline{2-13}
% \rule{0pt}{11pt} N &
% \multicolumn{2}{c|}{$\mathbf{A}$} & \multicolumn{2}{c|}{$\mathbf{P}\mathbf{A}$}
% &\multicolumn{2}{c|}{$\mathbf{A}$} & \multicolumn{2}{c|}{$\mathbf{P}\mathbf{A}$}
% &\multicolumn{2}{c|}{$\mathbf{A}$} & \multicolumn{2}{c|}{$\mathbf{P}\mathbf{A}$}
% \\ \cline{2-13}
% \rule{0pt}{11pt}& $\kappa$ & It. & $\kappa$ & It. & $\kappa$ & It. & $\kappa$ & It. & $\kappa$ & It. & $\kappa$ & It. \\ \hline
% 123 & 3.3 & 19 & 1.2 & 12 & 4.4 & 17 & 1.7 & 12 & 5.1 & 18 & 1.9 & 12\\
% 1068 & 14.2 & 31 & 1.3 & 12 & 27.8 & 36 & 2.4 & 14 & 33.1 & 38 & 2.9 & 15\\
% 4645 & 44.2 & 54 & 1.3 & 12 & 104.5 & 69 & 2.8 & 15 & 131.4 & 79 & 3.6 & 16\\
% 13680 & 101.4 & 73 & 1.4 & 12 & 277.1 & 103 & 3.0 & 15 & 358.8 & 117 & 3.9 & 16\\ \hline
%  \end{tabular}
%  }
% \end{table}

%%%%%%%%%%%%%%%%%%%%%%%%%%%%%%%%%%%%%%%%%%%%%%%%%%%%%%%%%%%%%%%%%%%%%%%%%%%%%%%
%%%%%%%%%%%%%%%%%%%%%%%%%%%%%%%%%%%%%%%%%%%%%%%%%%%%%%%%%%%%%%%%%%%%%%%%%%%%%%%
\pagebreak

%%%%%%%%%%%%%%%%%%%%%%%%%%%%%%%%%%%%%%%%%%%%%%%%%%%%%%%%%%%%%%%%%%%%%%%%%%%%%%%
%%%%%%%%%%%%%%%%%%%%%%%%%%%%%%%%%%%%%%%%%%%%%%%%%%%%%%%%%%%%%%%%%%%%%%%%%%%%%%%

%%%%%%%%%%%%%%%%%%%%%%%%%%%%%%%%%%%%%%%%%%%%%%%%%%%%%%%%%%%%%%%%%%%%%%%%%%%%%%%
%%%%%%%%%%%%%%%%%%%%%%%%%%%%%%%%%%%%%%%%%%%%%%%%%%%%%%%%%%%%%%%%%%%%%%%%%%%%%%%
\appendix

%%%%%%%%%%%%%%%%%%%%%%%%%%%%%%%%%%%%%%%%%%%%%%%%%%%%%%%%%%%%%%%%%%%%%%%%%%%%%%%
\section{\rcu{Proof of Results for Operator Preconditioning on Adaptive Meshes}}
\label{app:Adaptivity}

For the sake of presentation, we dedicate the next two subsections to briefly 
summarize some key concepts about adaptivity and the mesh conditions we 
need to fulfill for stability. Finally, we combine these preliminaries to state 
and and prove the new results on operator preconditioning in adaptively refined 
meshes.

%%%%%%%%%%%%%%%%%%%%%%%%%%%%%%%%%%%%%%%%%%%%%%%%%%%%%%%%%%%
\subsection{Adaptivity preliminaries}
We begin by reminding the reader of some of the concepts introduced in 
Section~\ref{sec:adaptive}. 
Given an initial triangulation $\mathcal{T}_h^{(0)}$, the adaptive 
Algorithm~\ref{alg:Adaptive} generates a sequence $\mathcal{T}_h^{(\ell)}$ of 
triangulations based on error indicators $\eta^{(\ell)}(\tau),\; \tau \in 
\mathcal{T}_h^{(\ell)}$, a refinement criterion and a refinement rule, by 
following the established sequence of steps:
\begin{equation*}
\mathrm{SOLVE}\to\mathrm{ESTIMATE}\to\mathrm{MARK}\to\mathrm{REFINE}.
\end{equation*}

There are different refinement rules that one can choose for the step REFINE. 
We now present some of the most common ones: red refinement, green 
refinement, and red-green refinement. 

\begin{definition}
\label{def:red-ref}
Let $\mathcal{T}_h^{(\ell)}$ be a triangulation. A triangle $\tau \in \mathcal{T
}_h^{(\ell)}$ is \textbf{red refined} by connecting edge midpoints of $\tau$, 
thus splitting $\tau$ into 4 similar triangles.
\end{definition}

\begin{definition}
\label{def:green-ref}
Let $\mathcal{T}_h^{(\ell)}$ be a triangulation. A triangle $\tau \in \mathcal{T
}_h^{(\ell)}$ is \textbf{green refined} by connecting an edge midpoint 
with the opposite vertex of $\tau$, thus splitting $\tau$ into 2 triangles.
\end{definition}

Next, in order to define a red-green refinement, we introduce two related 
properties.

\begin{definition}
\label{def:1-reg}
\begin{enumerate}[label=\alph*)]
 \item A triangulation $\mathcal{T}_h^{(\ell)}$ is called \textbf{$1$--irregular} 
 if the property
\begin{equation*}
\vert \operatorname{lev}(\tau_k) - \operatorname{lev}(\tau_m) \vert \leq 1,
\end{equation*}
holds for any pair of triangles $\tau_k, \tau_m \in \mathcal{T}_h^{(\ell)}$ such 
that $\tau_k \cap \tau_m \neq \emptyset$. 

Here $\operatorname{lev}(\tau_k)$ 
corresponds to the number of refinement steps required to generate $\tau_k$ from 
the initial triangulation $\mathcal{T}_h^{(0)}$.\\

\item The \textbf{$2$--neighbour rule}: Red refine any triangle $\tau_k$ with $2$ 
neighbours that have been red refined. Two triangles are neighbours, if they have a common edge.
\end{enumerate}
\end{definition}

\begin{definition}
 \label{def:red-green}
 A \textbf{Red-green refinement} for a triangulation $\mathcal{T}_h^{(\ell)}$ proceeds as follows:
\begin{enumerate}
\item Remove edges from any triangles that have been green refined.
\item All marked triangles are red refined.
\item Any triangles with 2 or more red refined neighbours are red refined, by 
$2$--neighbour rule.
\item Any triangles that do not fulfil $1$--irregularity rule are further refined.
\item Any triangles with hanging nodes generated during the refinement are green 
refined.
\end{enumerate}
\end{definition}

For further description of the refinement rules, we refer to \cite{bsw,jg}.

\begin{figure}
\center
\includegraphics[width = .95\textwidth]{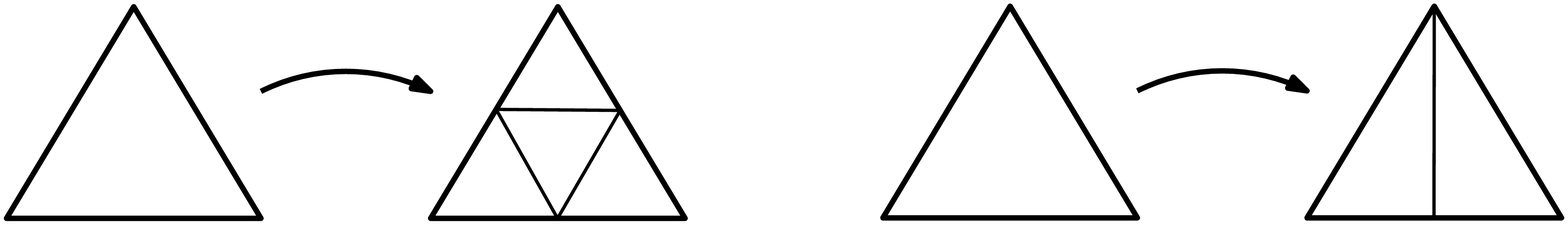}
\caption{Example of red refined triangle (left) and green refined triangle (right).}
\end{figure}

\begin{figure}
\center
\includegraphics[width = .6\textwidth]{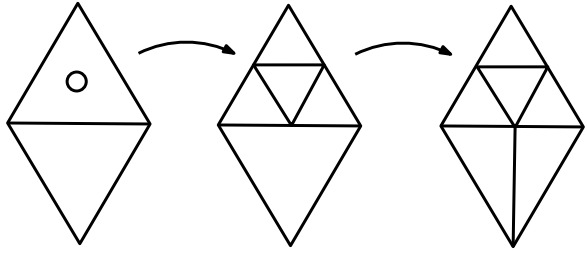}
\caption{Example of a sequence of red-green refinement. Top element is marked by $(o)$ and therefore is red refined in the first step. Bottom triangle then has a hanging node and is green refined in the consequent step.}
\end{figure}

%%%%%%%%%%%%%%%%%%%%%%%%%%%%%%%%%%%%%%%%%%%%%%%%%%%%%%%%%%%
\subsection{Mesh conditions}\label{app:conditions}
We recall that we aim to show \eqref{eq:infsupdh}, i.e.
\begin{equation*}
\underset{\varphi_h\in\mathbb{W}_{h}}{\mathrm{sup}} \dfrac{\bd(v_h, \varphi_h)}
{\|\varphi_h\|_{H^{-s}(\Omega)}} \geq \cu{\beta_{\bd}} \|v_h\|_{\widetilde{H}^s
(\Omega)}, \qquad \text{for all } v_h \in \widetilde{\mathbb{V}}_{h},
\end{equation*}
(see Section~\ref{sec:precond} for notation).

In the case of the discretizations based on dual meshes, this inf-sup stability 
is a consequence of three regularity conditions on the triangulation $\mathcal{T}_h$,
see \cite[Chapters 1--2]{s2}. We now proceed to introduce some notation to 
properly summarize this result. 

\rcu{Let $\mathcal{T}_h$ be a triangulation of $\Omega \subset \mathbb{R}^n$}.
For each triangle $\tau_k \in\mathcal{T}_h$ we define its \textbf{area} $\Delta_{
k} :=\int_{\tau_{k}} dx$; its \textbf{local element size} $h_{k} :=\Delta_{k}^{
1/ n}$; and its \textbf{diameter} $d_{k} :=\sup _{x, y \in \tau_{k}}|x-y|$. 

Let $\varphi_j$ be a piecewise linear basis function \rcu{in the span of 
$\widetilde{\mathbb{V}}_{h}$}. We write $\omega_j := \operatorname{supp}  \, 
\varphi_j $ and define its associated local mesh size $\hat{h}_j$ as
\begin{align*}
\hat{h}_j &:= \frac{1}{\# I(j)} \sum_{m \in I(j)} h_{m}.
\end{align*} 
Here,
$I(j) :=\left\{m \in\{1, \ldots, \#\mathcal{T}_h\} : \tau_{m} \cap \omega_{j} \neq 
\emptyset\right\}, \text{ for } j=1, \ldots, N,$ is the index set of triangles 
$\tau_m\in \mathcal{T}_h$ where the basis function $\varphi_j$ is not identically 
zero.\\

\begin{definition}\label{def:MC}
 For a triangulation $\mathcal{T}_h$, we define the following mesh conditions
\begin{itemize}
\item[$(C1)$] Shape regularity: \rcu{t}here exists $c_R > 0$ such that for all 
$\tau_k \in \mathcal{T}_h$
\begin{equation*}
0<c_R<\frac{h_{k}}{d_{k}}<1.
\end{equation*}
\item[$(C2)$] Local quasi-uniformity\rcu{: f}or all $\tau_k, \tau_m \in \mathcal{
T}$ with $\tau_k \cap \tau_m \neq \emptyset$
\begin{equation*}
\frac{h_k}{h_m} \leq c_L,
\end{equation*}
\rcu{with $C_L$ a (uniform) positive constant.}
\item[$(C3)$] Local $s$-dependent condition: \rcu{there exists $c_0>0$ such 
that f}or all $\tau \in \mathcal{T}_h$
\begin{equation*}
\frac{51}{7}-\sqrt{\sum_{j \in J(m)} \hat{h}_{j}^{2 s} \sum_{j \in J(m)} \hat{h
}_{j}^{-2 s}} \geq c_{0}>0,
\end{equation*}
with $J(m) :=\left\{i \in\{1, \ldots, N\} :  \omega_{i}\cap \tau_{m} \neq 
\emptyset\right\}$ for $m=1,\ldots, \#\mathcal{T}_h$, the index set of basis 
functions $\varphi_i$ which are \rcu{not identically zero} on triangle $\tau_m$.
\end{itemize} 
\end{definition}

\begin{thm}[{\cite[{Theorems~2.1 and 2.2}]{s}}]
\label{thm:dinfsupOlaf}
Let $\mathcal{T}_h$ be a triangulation of $\Omega$ such that $(C1), (C2)$ and 
$(C3)$ are satisfied. Consider the primal-dual discretization 
$\widetilde{\mathbb{V}}_{h}= \mathbb{S}^1(\mathcal{T}_h)\cap\widetilde{H}^{s}
(\Omega)$ and $\mathbb{W}_{h}=\mathbb{S}^0(\check{\mathcal{T}}_h)$ for $0\leq s\leq 1$ 
(see subsection~\ref{ssec:discretization}).

Then, the discrete inf-sup condition \eqref{eq:infsupdh} holds with a positive 
constant $\beta_{\bd}$ independent of $h$.
\end{thm}%

%%%%%%%%%%%%%%%%%%%%%%%%%%%%%%%%%%%%%%%%%%%%%%%%%%%%%%%%%%%
\subsection{Results on adaptively refined meshes}
\rcu{Now we turn our attention to study these conditions for a sequence of 
adaptive triangulations generated by Algorithm~\ref{alg:Adaptive}. For this, 
we write the constants from conditions (C1), (C2) and (C3) associated to a 
triangulation $\mathcal{T}_h^{(\ell)}$ as $c_R^{(\ell)}, c_L^{(\ell)}$ and 
$c_0^{(\ell)}$, respectively.

The next Lemma is the complete version of Lemma~\ref{lem:Adaptive} introduced 
in Section~\ref{sec:adaptive}.
}

\begin{lemma}\label{l:Adaptive}
Consider an \rcu{initial} triangulation $\mathcal{T}_h^{(0)}$ \rcu{satisfying 
the mesh conditions from Definition~\ref{def:MC}, and such that its local 
quasi-uniformity constant \rcu{$c_L^{(0)}$} verifies}
\begin{equation}\label{eq:c0Cond}
 c_L^{(0)} \leq \frac{1}{2}\sqrt[4 \vert s \vert]{\frac{1129}{49}} \approx 
\frac{2.19^{1/\vert s \vert}}{2}.
\end{equation}

\rcu{Let $\Xi:= \lbrace \mathcal{T}_h^{(\ell)}\rbrace_{\ell\in \N}$ be a family of meshes
generated from $\mathcal{T}_h^{(0)}$ by the adaptive refinement described in 
Algorithm~\ref{alg:Adaptive}, using red-green refinements.}
Then (C1), (C2) and (C3) hold for all \rcu{$\mathcal{T}_h^{(\ell)} \in 
\Xi$} for some constants 
$c_R, c_L, c_0>0$, which are independent of $\ell \in \mathbb{N}$. 

In particular, the inf-sup condition \eqref{eq:infsupdh} holds for 
$\vert s \vert \leq 1$, $\widetilde{\mathbb{V}}_{h}= \mathbb{S}^p(\mathcal{T}_h)
\cap\widetilde{H}^{s}(\Omega)$, $\mathbb{W}_{h}=\mathbb{S}^q(\mathcal{T}_h^{\prime})$, and for all $\mathcal{T}_h^{(\ell)}$ independent of $\ell$.
\end{lemma}

\begin{proof}
The proof proceeds by induction on $\ell$. By hypothesis, the initial 
triangulation $\mathcal{T}_h^{(0)}$ satisfies  $(C1)$ and $(C2)$. %\cc{It may be shown that $(C1)$ implies $(C2)$}. 
Therefore, for the initial triangulation 
$\mathcal{T}_h^{(0)}$ we only need to check $(C3)$. 

For the sake of convenience, let us re-label the basis functions
$j \in J(m)$ by $m_i$, with $i=1,\ldots, \#J(m)$. We note that $\max_{m} 
\#J(m)=3$ and that this is our worst case scenario. Therefore, it suffices 
to verify  $(C3)$ in this case:
\begin{equation*}
\frac{51}{7}-\sqrt{\sum_{i = 1}^3 \hat{h}_{m_i}^{2 s} \sum_{i=1}^3 \hat{h}_{
m_i}^{-2 s}} \geq c_{0}>0.
\end{equation*}
Without loss of generality, let $\hat{h}_{m_1} \geq \hat{h}_{m_2} \geq 
\hat{h}_{m_3}$.
Then 
\begin{equation*}
\begin{split}
\sum_{i = 1}^3 \hat{h}_{m_i}^{2 s} \sum_{i=1}^3 \hat{h}_{m_i}^{-2 s}  &= 
\textstyle{3+ \left(\frac{\hat{h}_{m_1}}{\hat{h}_{m_2}} \right)^{2|s|} + 
\left(\frac{\hat{h}_{m_2}}{\hat{h}_{m_3}} \right)^{2|s|} + \left(\frac{\hat{h}_{
m_3}}{\hat{h}_{m_1}} \right)^{2|s|} }\\ & \qquad+\textstyle{ \left(\frac{\hat{h
}_{m_1}}{\hat{h}_{m_3}} \right)^{2|s|} + \left(\frac{\hat{h}_{m_2}}{\hat{h}_{m_1}} 
\right)^{2|s|} + \left(\frac{\hat{h}_{m_3}}{\hat{h}_{m_2}} \right)^{2|s|} }\\
&\leq 3+\textstyle{2\left(  \left(\frac{\hat{h}_{m_1}}{\hat{h}_{m_3}} \right)^{
2|s|} + \left(\frac{\hat{h}_{m_2}}{\hat{h}_{m_2}} \right)^{2|s|} + \left(\frac{
\hat{h}_{m_3}}{\hat{h}_{m_1}} \right)^{2|s|} \right)  \leq 7+2\left(\frac{\hat{
h}_{m_1}}{\hat{h}_{m_3}} \right)^{2|s|},}\\
%&\leq 3+\textstyle{2\left(  \left(\frac{\hat{h}_{m_1}}{\hat{h}_{m_3}} \right)^{2|s|} + \left(\frac{\hat{h}_{m_2}}{\hat{h}_{\cu{m_3}}} \right)^{2|s|} + \left(\frac{\hat{h}_{m_3}}{\hat{h}_{\cu{m_2}}} \right)^{2|s|} \right)  \leq \cu{5+4}\left(\frac{\hat{h}_{m_1}}{\hat{h}_{m_3}} \right)^{2|s|},}
\end{split}
\end{equation*}
where we use the rearrangement inequality. We conclude that $(C3)$ is satisfied 
for $\mathcal{T}_h^{(0)}$ provided that
\begin{equation}\label{e:h1h3bound}
 \textstyle{\left(\frac{\hat{h}_{m_1}}{\hat{h}_{m_3}} \right)^{2|s|} < \frac{1129}{49}.}
%\textstyle{\left(\frac{\hat{h}_{m_1}}{\hat{h}_{m_3}} \right)^{2|s|} < \cu{\frac{589}{49}.}}
\end{equation} 
A simple calculation using the mesh conditions yields $\frac{\hat{h}_{m_1}}{
\hat{h}_{m_3}} \leq (c_L^{(0)})^2$, so that \eqref{e:h1h3bound} holds and $(C3)$ 
is satisfied for $\mathcal{T}_h^{(0)}$.\\
For the inductive step, assume that conditions $(C1)$--$(C3)$ are satisfied on 
an adaptively refined triangulation $\mathcal{T}_h^{(\ell)}$ using red-green 
refinements subject to $1$--irregularity and $2$--neighbour rules. In order to 
generate a new triangulation $\mathcal{T}_h^{(\ell+1)}$, the appropriate triangles 
are marked.

We note that red-refinement does not change the shape regularity constant, but
green refinement worsens the shape regularity constant by at most a factor of 
$\frac{1}{\sqrt{2}}$. However, due to the removal of green edges, the constant 
does not degenerate as $\ell \to \infty$. Thus condition $(C1)$ is satisfied 
with $c_R^{(\ell+1)} \geq \frac{1}{\sqrt{2}} c_R^{(0)}$ for $\mathcal{T}_h^{(\ell
+1)}$.

Condition $(C2)$ remains satisfied due to the $1$--irregularity condition 
in the refinement procedure. This restriction guarantees that $\frac{h_i}{h_j} 
\leq c_L^{(\ell+1)} \leq 2 c_L^{(0)}$.

As for the initial triangulation $\mathcal{T}_h^{(0)}$, we know that condition 
$(C3)$ is satisfied for $\mathcal{T}_h^{(\ell+1)}$ when \eqref{e:h1h3bound} holds.
Due to the $1$--irregularity condition, we have that $\frac{\hat{h}_{m_1}}{\hat{
h}_{m_3}} \leq (2 c_L^{(0)})^2$, so the estimate \eqref{e:h1h3bound} is satisfied 
provided 
\js{$c_L^{(0)} < \frac{1}{2}{\left( \frac{1129}{49} \right)}^{{1}/{4 \vert s 
\vert}}$.}

We conclude that $(C1)$, $(C2)$, $(C3)$ are satisfied for $\lbrace \mathcal{T}_h^{
(\ell)} \rbrace_{\ell=0}^{\infty}$ independently of $\ell$.
\rcu{\qed}
\end{proof}
\begin{remark}
\begin{enumerate}[label=\alph*)]
 \item We note that the estimates in Lemma~\ref{l:Adaptive} are not sharp. Still, 
the local quasi-uniformity assumption on the initial triangulation $\mathcal{T
}^{(0)}$ becomes more restrictive as $\vert s \vert$ increases. Thus, the initial
mesh needs to be of increasingly higher regularity for higher values of $|s|$.

\item Let $\Gamma \rcu{\subset} \mathbb{R}^n$ \rcu{be} a polyhedral domain which 
satisfies an interior
cone condition. Then the assumptions in Lemma~\ref{l:Adaptive} can be satisfied 
for a sufficiently fine  $\mathcal{T}_h^{(0)}$.
\end{enumerate}
\end{remark}

\begin{remark}
\jak{Similar results can be shown for alternative refinement strategies, such as 
the newest vertex bisection \rc{\cite[Section 2.2]{dirk_nvb}}. See \cite[Chapter 4]{jsthesis} for details.}
%b) For $|s|=\frac{1}{2}$, starting from an initial mesh $\mathcal{T}^{(0)}$ with $c_R^{(0)} \geq 0.5, c_L^{(0)} \leq 1.475$, and provided that $\max_{z \in \mathcal{P}_h} \# I(z) \leq 6$, then red-green refinement does not need the additional assumptions.)
\end{remark}

%%%%%%%%%%%%%%%%%%%%%%%%%%%%%%%%%%%%%%%%%%%%%%%%%%%%%%%%%%%%%%%%%%%%%%%%%%%%%%%
\section{Proof of Proposition~\ref{prop:blowup}}
\label{app:blowup}

% In order to prove Proposition~\ref{prop:blowup}, we first need to show some 
% auxiliary results and definitions. 
The idea for the proof is like in \cite{CHR02} where the case $\mathbb{W}_{h}= 
\widetilde{\mathbb{V}}_{h}$ is shown. Here we generalize the proof to different 
discrete test and trial space. For the sake of brevity we will discuss the case 
when $s\in(1/2,1]$ and remark that the proof for $s\in[-1,-1/2)$ follows 
analogously. We remind the reader that in this setting $\widetilde{H}^s(\Omega)
\equiv H^s_0(\Omega)\neq H^s(\Omega)$, but that $\Vert u \Vert_{\widetilde{H
}^s(\Omega)} \equiv \Vert u \Vert_{{H}^s(\Omega)}, \: \forall u \in \widetilde{
H}^s(\Omega)$. 
%(which we are going to use here, see $C_2$ below). 

Let $\mathcal{T}_h$, $\mathbb{S}^p(\mathcal{T}_h), p\in\N$ be as in 
Section~\ref{sec:precond}. Moreover, we recall that for this setting we consider 
the finite element spaces $\widetilde{\mathbb{V}}_{h}=\mathbb{S}^1(\mathcal{T}_h
)\cap\widetilde{H}^{s}(\Omega)$ and $\mathbb{W}_{h}\subset {H}^{-s}(\Omega)$.
Additionally, we denote $\mathbb{V}_{h}=\mathbb{S}^1(\mathcal{T}_h)\subset
{H}^{s}(\Omega)$ and note that $\widetilde{\mathbb{V}}_{h}\subset \mathbb{V}_{h
}$. Indeed, $\widetilde{\mathbb{V}}_{h}$ is the space of affine continuous 
functions that are zero on the boundary, while $\mathbb{V}_{h}$ is analogous to 
$\widetilde{\mathbb{V}}_{h}$, but admits non-zero values on $\partial \Omega$.

% Let $\OP_h \, : L^2(\Omega)\to\widetilde{\mathbb{V}}_{h}$ be the orthogonal 
% $L^2$-projection.
Let us introduce the generalized $L^2$-projection $\widetilde{Q}_h:L^2(\Omega)
\to \widetilde{\mathbb{V}}_{h}$ for a given $u\in L^2(\Omega)$, as the 
solution of the variational problem
\begin{equation}
\dual{\widetilde{Q}_h u}{ \psi_h}_{\Omega} = \dual{u}{\psi_h}_{\Omega}, 
\quad \forall \psi_h \in \mathbb{W}_{h}.
\end{equation}
From \cite[Chapter~2]{s2}, \cite{hut}, we know that it satisfies
 \begin{equation}
  \Vert \widetilde{Q}_h u \Vert_{\widetilde{H}^s(\Omega)} \leq 
  \cu{\beta_{\bd}^{-1}} \Vert u \Vert_{\widetilde{H}^s(\Omega)}, \qquad \forall 
  u\in \widetilde{H}^s(\Omega).
 \end{equation} 
where $\cu{\beta_{\bd}}$ is the inf-sup constant from 
\eqref{eq:infsupdh}.

Given that we are interested in the case where we have a space mismatch, i.e. 
when $u \in H^s(\Omega)$ but $u \notin \widetilde{H}^s(\Omega)$, we 
additionally prove the following:
\begin{lemma}\label{lem:L2projVh}
The projection $\widetilde{Q}_h$ satisfies
 \begin{equation}
\Vert \widetilde{Q}_h u_h \Vert_{H^s(\Omega)} \leq ( 1 + \cu{\frac{c_2}{s-{1/2}
}}\,h^{1/2-s} )\Vert u_h \Vert_{H^s(\Omega)}, \qquad \forall u_h\in \mathbb{V
}_{h},
\end{equation} 
with $c_2>0$ and independent of $h$.
\end{lemma}
\begin{proof}
Set $u^0_h \in \widetilde{\mathbb{V}}_{h}$ to be the function defined by
\begin{equation}
 u^0_h := \begin{cases}
              u_h, \quad \text{ in all interior nodes},\\
              0, \qquad \text{on }\partial \Omega.
             \end{cases}
\end{equation}
Then, by definition
\begin{align*}
 \Vert u_h - \widetilde{Q}_h u_h \Vert_{L^2(\Omega)} &=\Vert u_h - u_h^0 \Vert_{
 L^2(\Omega)}\leq h^{1/2}\Vert u_h \Vert_{L^2(\partial\Omega)},
\end{align*}
where the last inequality holds by basic computations (\textit{c.f.} 
\cite[Equation 1.3.27]{CHR02}).

From the trace theorem, we have that $\Vert u_h \Vert_{L^2(\partial\Omega)} 
\leq \dfrac{c_{tt}}{s-1/2} \Vert u_h \Vert_{H^s(\Omega)}$, \cu{with $c_{tt}>0$ 
independent of $h$.}

Therefore, combining all the above, we obtain
\begin{align*}
 \Vert \widetilde{Q}_h u_h \Vert_{H^s(\Omega)} &\leq \Vert u_h \Vert_{H^s(\Omega)}
 + \Vert \widetilde{Q}_h u_h - u_h \Vert_{H^s(\Omega)}\\
 &\leq \Vert u_h \Vert_{H^s(\Omega)}
 + \cu{c_1} h^{-s}\Vert \widetilde{Q}_h u_h - u_h \Vert_{L^2(\Omega)}\\
 &\leq
 \left( 1 + \frac{\cu{c_1 c_{tt}}}{s-1/2}h^{1
  /2-s} \right)\Vert u_h \Vert_{H^s(\Omega)}.
\end{align*}
\rcu{\qed}
\end{proof}

Now, let us also introduce the finite element space $\widetilde{\mathbb{W}}_{h}
\subset \widetilde{H}^{-s}(\Omega)$. We consider the generalized $L^2$-projection 
$\widetilde{P}_h:L^2(\Omega)\to \widetilde{\mathbb{W}}_{h}$ for a given $\varphi
\in L^2(\Omega)$, as the solution of the variational problem
\begin{equation}
\dual{\widetilde{P}_h \varphi}{ v_h}_{\Omega} = \dual{\varphi}{v_h}_{\Omega}, 
\quad \forall v_h \in \mathbb{V}_{h}.
\end{equation}
Then, in analogy with Lemma~\ref{lem:L2projVh}, we have that
\begin{lemma}
The projection $\widetilde{P}_h$ satisfies
 \begin{equation}
  \Vert \widetilde{P}_h \Phi_h \Vert_{H^{-s}(\Omega)} \leq c_3( 1 + \cu{\frac{c_2}{s-{1/2}
}}\,h^{1/2
  -s} )\Vert \Phi_h \Vert_{H^{-s}(\Omega)}, \qquad \forall 
  \Phi_h\in \mathbb{W}_{h},
 \end{equation} 
  with $c_2,c_3>0$ and independent of $h$.
\end{lemma}
\begin{proof}
Let us use the norms' properties and write
\begin{align*}
 \Vert \widetilde{P}_h \Phi_h \Vert_{H^{-s}(\Omega)} \leq 
 \Vert \widetilde{P}_h \Phi_h \Vert_{\widetilde{H}^{-s}(\Omega)}
 =  \underset{0\neq u \in H^{s}(\Omega)}{\mathrm{sup}} 
  \dfrac{\dual{\widetilde{P}_h \Phi_h}{u}_{\Omega}}{\Vert u \Vert_{H^{s}
  (\Omega)}}.
\end{align*}
Then, using the definition of $\widetilde{Q}_h$ and the estimates above, 
we get
\begin{align*}
\Vert \widetilde{P}_h \Phi_h \Vert_{H^{-s}(\Omega)} &\leq (1 +\cu{\frac{c_2}
{s-{1/2}}}\,h^{1/2-s}) \underset{0\neq u \in H^{s}(\Omega)}{\mathrm{sup}} \dfrac{
\dual{\widetilde{P}_h \Phi_h}{\widetilde{Q}_h u}_{\Omega}}{\Vert \widetilde{Q}_h 
u \Vert_{H^{s}(\Omega)}}\\
&\leq (1 + \cu{\frac{c_2}{s-{1/2}}}\,h^{1/2-s}) \underset{0\neq u_h \in 
\widetilde{\mathbb{V}}_h}{\mathrm{sup}} \dfrac{\dual{\widetilde{P}_h \Phi_h}{u_h
}_{\Omega}}{\Vert u_h \Vert_{H^{s}(\Omega)}}.
\end{align*}
Now, by definition of $\widetilde{P}_h$, and since $\widetilde{\mathbb{V}}_h
\subset \mathbb{V}_h$, we have
\begin{align*}
\Vert \widetilde{P}_h \Phi_h \Vert_{H^{-s}(\Omega)} &\leq (1 +\cu{\frac{c_2}{s
-{1/2}}}\,h^{1/2-s}\,) \underset{0\neq u_h \in \widetilde{\mathbb{V}}_h}{
\mathrm{sup}} \dfrac{\dual{\Phi_h}{u_h}_{\Omega}}{\Vert u_h \Vert_{H^{s}(\Omega)
}}\\ &\leq \cu{c_3}(1 + \cu{\frac{c_2}{s-{1/2}}}\,h^{1/2-s}\,) \Vert\Phi_h
\Vert_{H^{-s}(\Omega)}.
\end{align*}
\rcu{\qed}
\end{proof}

\begin{lemma}
 Let $ s\in(1/2,1)$. Then, the following inf-sup condition holds
 \begin{equation}
\underset{\phi_h \in \widetilde{\mathbb{W}}_h}{\mathrm{sup}}\dfrac{\dual{v_h}{
\phi_h}_{\Omega}}{\Vert \phi_h \Vert_{H^{-s}(\Omega)}}\geq \cu{\frac{\beta_{\bd}
}{c_3}}\left(1 + \cu{\frac{c_2}{s-{1/2}}} \,h^{1/2-s} \right)^{-1} \Vert v_h 
\Vert_{\widetilde{H}^{s}(\Omega)}, \quad \forall v_h \in \widetilde{\mathbb{V}
}_{h},
 \end{equation} 
   with $\cu{c_3,c_2}>0$ and independent of $h$.
\end{lemma}
\begin{proof}
Let us introduce the operator $\Pi_h^s: \widetilde{H}^s(\Omega) \to 
\mathbb{W}_h\subset H^{-s}(\Omega)$ for $s\in(0,1]$, defined by the variational 
formulation
\begin{equation}
\dual{\Pi_h^s u}{ v_h}_{\Omega} = (u, v_h)_{\widetilde{H}^s(\Omega)}, 
\quad \forall v_h \in \widetilde{\mathbb{V}}_{h},
\end{equation}
where $(\cdot,\cdot)_{\widetilde{H}^s(\Omega)}$ denotes the $\widetilde{H}^s(
\Omega)$-inner product. This operator is analogous to \cite[Equation~1.75]{s2}
\cite[Equation~4.22]{hju}, and thus it verifies
 \begin{equation}
\Vert \Pi_h^s u \Vert_{{H}^{-s}(\Omega)} \leq \cu{\beta_{\bd}^{-1}} \Vert u 
\Vert_{\widetilde{H}^s(\Omega)}, \qquad \forall u\in \widetilde{H}^s(\Omega).
 \end{equation} 

Next, 
% by definition of the $\widetilde{H}^{s}$-norm and the properties of $\Pi_h^s$, 
we have that for any $v_h \in \widetilde{\mathbb{V}}_h$ 
\begin{align*}
\Vert v_h \Vert_{\widetilde{H}^{s}(\Omega)} &= \dfrac{(v_h, v_h)_{\widetilde{H
}^{s}(\Omega)}}{\Vert v_h \Vert_{\widetilde{H}^{s}(\Omega)}} 
= \dfrac{\dual{v_h}{\Pi_hv_h}_{\Omega}}{\Vert v_h \Vert_{\widetilde{H}^{s}(
\Omega)}} \leq \beta_{\bd}^{-1}\dfrac{\dual{v_h}{\Pi_hv_h}_{\Omega}}{\Vert 
\Pi_h v_h\Vert_{H^{-s}(\Omega)}} = \beta_{\bd}^{-1} \dfrac{\dual{v_h}{\widetilde{
P}_h\Pi_hv_h}_{\Omega}}{\Vert \Pi_hv_h\Vert_{{H}^{-s}(\Omega)}},
\end{align*}
where in the last step we used that $\Pi_hv_h \in \mathbb{W}_h$ and the 
definition of $\widetilde{P}_h$.

Now, let us use our previous estimates to derive
\begin{align*}
\Vert v_h \Vert_{\widetilde{H}^{s}(\Omega)} \leq \frac{c_3}{\beta_{\bd}}(1 + 
\frac{c_2}{s-1/2}\, h^{1/2-s})\dfrac{\dual{v_h}{\widetilde{P}_h\Pi_hv_h}_{\Omega
}}{\Vert \widetilde{P}_h\Pi_hv_h\Vert_{{H}^{-s}(\Omega)}}.
\end{align*}
Set $\varphi_h := \widetilde{P}_h\Pi_hv_h$ and note that $\varphi_h\in \widetilde{
\mathbb{W}}_h$. Therefore, this gives
\begin{align*}
\Vert v_h \Vert_{\widetilde{H}^{s}(\Omega)} 
&\leq \frac{c_3}{\beta_{\bd}}(1 + \frac{c_2}{s-1/2}\, h^{1/2-s}\,)\dfrac{\dual{
v_h}{\varphi_h}_{\Omega}}{\Vert \varphi_h\Vert_{{H}^{-s}(\Omega)}}
\leq \frac{c_3}{\beta_{\bd}}(1 + \frac{c_2}{s-1/2}\, h^{1/2-s}\,)\underset{\phi_h 
\in \widetilde{\mathbb{W}}_h}{\mathrm{sup}}\dfrac{\dual{v_h}{\phi_h}_{\Omega}}
{\Vert \phi_h\Vert_{{H}^{-s}(\Omega)}}. 
\end{align*}
Finally, move the factors to the other side and one gets the desired result.
\rcu{\qed}
\end{proof}
{\it Proof of Proposition~\ref{prop:blowup}} 
 First notice that in this context the inf-sup 
constant of $\bd$ is $$\cu{\tilde{\beta}_\bd := \frac{\beta_{\bd}}{c_3}\left(1 + 
\frac{c_2}{s-1/2}\, h^{1/2-s} \right)^{-1}}.$$ 
Then, we plug this 
in \eqref{mainresult} and get
 \begin{equation}
 \kappa\left( \mathbf{D}^{-1} \tilde{\mathbf{C}}_s \mathbf{D}^{-T} \mathbf{A}
 \right) \leq \dfrac{C_{\gamma} C_A \Vert \bd \Vert^2 c_3^2\left( 1 + \dfrac{c_2
 }{s-1/2}\, h^{1/2-s} \right)^2}{\beta_A \beta_{\gamma} \beta_{\bd}^2} \sim 
 \mathcal{O}(h^{1-2s}).
\end{equation}
\rcu{\qed}

\end{document}